\documentclass[10pt]{amsart}
\usepackage{amsmath}
\usepackage{amssymb}
\usepackage{amsfonts}
\usepackage{amsthm}
\usepackage{amsxtra}
\usepackage{fontenc}
\usepackage{fullpage}
\usepackage{hyperref}

\theoremstyle{plain}

\newtheorem{corollary}{Corollary}[section]
\newtheorem{lemma}{Lemma}[section]
\newtheorem{proposition}{Proposition}[section]
\theoremstyle{definition}
\newtheorem{definition}{Definition}[section]

\theoremstyle{remark}
\newtheorem{remark}{Remark}[section]

\newcommand{\C}{\mathbb C}

\newcommand{\Z}{\mathbb Z}
\newcommand{\N}{\mathbb N}
\newcommand{\+}{\!+\!}

\newcommand{\half}{
        {\lower0.00ex\hbox{\raise.6ex\hbox{\the\scriptfont0 1}
                           \kern-.5em\slash\kern-.1em\lower.45ex
                                     \hbox{\the\scriptfont0 2}}}}
\newcommand{\quarter}{
        {\lower0.00ex\hbox{\raise.6ex\hbox{\the\scriptfont0 1}
                           \kern-.5em\slash\kern-.1em\lower.45ex
                                     \hbox{\the\scriptfont0 4}}}}
\newcommand{\tquarter}{
        {\lower0.00ex\hbox{\raise.6ex\hbox{\the\scriptfont0 3}
                           \kern-.5em\slash\kern-.1em\lower.45ex
                                     \hbox{\the\scriptfont0 4}}}}
\newcommand{\eighth}{
        {\lower0.00ex\hbox{\raise.6ex\hbox{\the\scriptfont0 1}
                           \kern-.5em\slash\kern-.1em\lower.45ex
                                     \hbox{\the\scriptfont0 8}}}}
\newcommand{\othird}{
        {\lower0.00ex\hbox{\raise.6ex\hbox{\the\scriptfont0 1}
                           \kern-.5em\slash\kern-.1em\lower.45ex
                                     \hbox{\the\scriptfont0 3}}}}

\newcommand{\ddoAW}{\mathbb{D}}
\newcommand{\moAW}{\mathbb{M}}
\newcommand{\Gt}{\mathfrak{N}}

\begin{document}

\title[]{Semi-classical Orthogonal Polynomial Systems on Non-uniform Lattices, Deformations of the Askey Table and Analogs of Isomonodromy}

\author{N.S.~Witte}
\address{Department of Mathematics and Statistics,
University of Melbourne,Victoria 3010, Australia}
\email{\tt n.witte@ms.unimelb.edu.au}

\begin{abstract}
A $\ddoAW$-semi-classical weight is one which satisfies a particular linear, first
order homogeneous equation in a divided-difference operator $ \ddoAW $. It is known
that the system of polynomials, orthogonal with respect to this weight, and the
associated functions satisfy a linear, first order homogeneous matrix equation in
the divided-difference operator termed the spectral equation. Attached to the 
spectral equation is a structure which constitutes a number of relations such as 
those arising from compatibility with the three-term recurrence relation. Here 
this structure is elucidated in the general case of quadratic lattices. The simplest examples of the 
$\ddoAW$-semi-classical orthogonal polynomial systems are precisely those in the
Askey table of hypergeometric and basic hypergeometric orthogonal polynomials.
However within the $\ddoAW$-semi-classical class it is entirely natural to define 
a generalisation of the Askey table weights which involve a deformation with 
respect to new deformation variables. We completely construct the analogous 
structures arising from such deformations and their relations with the other elements 
of the theory. As an example we treat the first non-trivial deformation of the Askey-Wilson orthogonal
polynomial system defined by the $q$-quadratic divided-difference operator, the Askey-Wilson operator, 
and derive the coupled first order divided-difference equations characterising its evolution
in the deformation variable. We show that this system is a member of a sequence of 
classical solutions to the $ E^{(1)}_7 $ $q$-Painlev\'e system.
\end{abstract}

\maketitle

\tableofcontents

\section{Background and Motivation}\label{Start}
\setcounter{equation}{0}

We propose a method for constructing systems of linear divided-difference equations which are analogs 
of isomonodromic linear differential equations and therefore "isomonodromic interpretations" or Lax pairs of the known
$q$-Painlev\'e and difference Painlev\'e equations, such as those in the Sakai scheme
\cite{Sa_2001}. What is meant by an analog of a monodromy matrix turns out to be a connection matrix,
appropriate to the class of lattices defining the divided-difference operators
under consideration. In essence our method constructs a particular "isomonodromic analog" system
from an orthogonal polynomial system, orthogonal with respect to a generalisation or
deformation of a weight with discrete or countable support on a class of
non-uniform quadratic lattices. We then deduce a number of linear divided-difference equations that
this system satisfies and show that their pair-wise compatibility holds provided the coefficients
of the linear system obey evolution equations of the difference or $q$-Painlev\'e type.
This is a very natural extension of the Fokas-Its-Kitaev construction
\cite{IKF_1991,FIK_1991,FIK_1992} at the heart of Riemann-Hilbert techniques.

Our method is independent of and distinct from other approaches which we briefly recount here.
The first studies to construct Lax pairs for the $q$-Painlev\'e equations were those
of Jimbo and Sakai \cite{JS_1996} and Sakai \cite{Sa_2005,Sa_2006} using
the Birkhoff theory of linear $q$-difference equations $ Y(qx)=A(x)Y(x) $ and imposing the condition 
that the connection matrix was independent of the zeros of $ \det A(x) $.
However this approach has not been extended 
beyond the $ D^{(1)}_5 $ or $ E^{(1)}_6 $ cases (the latter case only found from a degeneration of
the two-variable extension of the former). 
Another distinct approach which is founded upon the notion of the $\tau$-function of a
rational d-connection is the Arinkin-Borodin theory \cite{Bo_2006,AB_2006,AB_2009},
which has been applied to the difference Painlev\'e equations.
Our approach is similar in spirit to that of Rains \cite{Ra_2007}, who has treated the
master elliptic Painlev\'e equation, in that an explicit construction is made of the 
solution to the linear problem which contains a multiple integral representation of the 
orthogonal rational function. However we will not make direct contact with this theory because it requires
us to consider orthogonal rational functions on elliptic lattices, which is a generalisation beyond
the class of lattices considered here. 
In addition we believe that the discrete Riemann-Hilbert approach, as formulated by Borodin in 
\cite{Bo_2003} and applied to two examples and extended to further cases in \cite{BB_2003}, 
shares many features with the present study
and one should be able to construct a Riemann-Hilbert formulation of our own theory. 
Recently Yamada has constructed Lax pairs for the $q$-Painlev\'e equations 
for the systems with symmetries $ E^{(1)}_8, E^{(1)}_7, E^{(1)}_6 $ by a reduction from the
elliptic form of the $  E^{(1)}_8 $ Painlev\'e equation \cite{Ya_2009a,Ya_2010}
however no theoretical construction from first principles was proposed in the individual cases. 

The approach we propose here has been successfully employed for the isomonodromic systems and
Painlev\'e equations, see for example \cite{Ma_1985,Ma_1994,Ma_1995a,BR_1994,IKF_1991,CI_1997,CIvA_1998,FW_2003b,vA_2007,FW_2006a,Wi_2009b}, 
and while most of the findings are reproductions of known ones it has lead to novel results
hitherto not found using other treatments, such as the discrete Garnier systems \cite{Wi_2009a}.
An important feature of our approach is that it is strongly motivated by a probabilistic
setting, namely that of the theory of random matrices and more generally determinantal point processes 
where the classical weights in the Askey table appear
in the one-body factors of the eigenvalue probability density functions \cite{rmt_Fo}.
Some preliminary exploration of the program we propose here has already been initiated but
not carried through to its logical conclusion, and we will delay citing this work
until the body of our paper where it is directly relevant. However we should point out
that we are most indebted to the pioneering work of Magnus \cite{Ma_1988,Ma_1995}.

The essential elements of our approach are the following.

(A)
The classification of special non-uniform lattices {\it SNUL} of quadratic type \cite{Ma_1988}, 
their associated divided-difference operators
$ \ddoAW_x $ and $ \moAW_x $ and their rules of calculus, which applies to orthogonal or bi-orthogonal 
polynomial systems. Of significance is the fact that in general these lattices possess two fixed points
which we denote $ x_L, x_R $.

(B)
The notion of a $ \ddoAW$-semi-classical weight $ w(x) $ which is characterised by an analog of the 
Pearson equation (see Definition \ref{spectral_DD}) \cite{Ma_1988, Ma_1995}
\begin{equation}
    \ddoAW_{x} w(x) = \frac{2V(x)}{W(x)}\moAW_{x} w(x) ,
\end{equation}
where $ W(x), V(x) $ are polynomials in $ \C[x] $.

(C)
The orthogonal polynomial system (OPS) defined by such a weight on the {\it SNUL} $ Y_{n}(x) \in \C^{2\times 2} $
satisfies a three-term recurrence relation, which in our context is a particular Schlesinger transformation
$ n \mapsto n+1 $ (see \ref{ops_Yrecur})
\begin{equation}
     Y_{n+1}(x) = K_n(x)Y_{n}(x) .
\end{equation}

(D)
The semi-classical character implies, under fairly weak conditions, a spectral structure on the "isomonodromic" system $ Y_{n}(x) $ \cite{Ma_1988,Ma_1995}
(see Proposition \ref{spectral_DDO}), namely that it satisfies the linear divided-difference equation
\begin{equation}
    \ddoAW_{x} Y_n(x) = A_n(x) \moAW_{x} Y_n(x) .
\end{equation}
Here $ A_n(x) $ is rational in $ x $ and the degrees of its numerator and denominator with respect to 
$ x $ are independent of $ n $.

(E)
Parallel to the spectral structure is a deformation structure, whereby the weight and the system acquires a $ u $
dependence, constrained by another Pearson relation (see Definition \ref{deform_wgt})
\begin{equation}
    \ddoAW_{u} w(x;u) = \frac{2S(x;u)}{R(x;u)}\moAW_{u} w(x;u) ,
\end{equation}
where $ R(x;u), S(x;u) $ are polynomials in $ \C[x] $. 

(F)
This also has a direct consequence for the "isomonodromic" system $ Y_{n}(x;u) $, a deformation structure
(see Proposition \ref{deform_DDeqn}) and a second associated linear divided-difference equation
\begin{equation}
    \ddoAW_{u} Y_n(x;u) = B_n(x;u) \moAW_{u} Y_n(x;u) ,
\end{equation}
where $ B_n(x;u) $ is also rational in $ x $.

(G)
The compatibility relations implied by this over-determined "isomonodromic" system $ Y_{n}(x;u) $
then lead to a number of conclusions -
the deformation matrix $ B_n(x;u) $ is expressible in terms of the spectral matrix $ A_n(x;u) $ at neighbouring
lattice nodes which we refer to as {\it closure} (see Proposition \ref{AB_resolve_z}), 
and furthermore relations exist between 
components of the spectral matrix at two consecutive nodes on the $ u $-lattice which, given
a suitable parameterisation of this matrix by appropriate co-ordinates, is an recurrence relation on
the deformation lattice. Our key results for such recurrence relations are given in Propositions \ref{evolution_1}, \ref{evolution_2}, 
\ref{evolution_3} and \ref{evolution_4}.

(H)
In our approach we derive a preservation property for the connection matrix defined as
\begin{equation}
   P(x;u) := \left( Y_{R}(x;u) \right)^{-1}Y_{L}(x;u) ,
\end{equation}
where $ Y_{L,R}(x;u) $ are local fundamental solutions of the spectral equation about $ x_{L,R} $.  
Specifically this implies
\begin{equation}
   \ddoAW_{x} P(x;u) = 0, \qquad \ddoAW_{u} P(x;u) = 0 ,
\end{equation}
as a consequence of our assumptions. 

In our approach we can sidestep a number of issues, which are important to be sure, but don't affect the outcome.
So we postpone deeper considerations of an analytic or algebraic theory of linear systems of divided-difference 
operators on these lattices for subsequent studies, however our present work will provide 
concrete illustrative examples for such an investigation.
We are referring to, for example, issues of a Galois theory for $\ddoAW$-difference equations,
a Birkhoff theory for the local character of the solutions to systems of linear first order $\ddoAW$-difference equations
and analogs of monodromy for $\ddoAW$-difference equations.

Our approach poses the question concerning a correspondence between the system of hypergeometric and 
basic hypergeometric orthogonal polynomial systems generalising the classical systems, known as
the Askey table, as represented in reference work \cite{KS_1998},
and the most complete system of elliptic, $q$-difference and difference analogs of the Painlev\'e 
equations, the Sakai scheme \cite{Sa_2001}.
A correspondence between the Sakai scheme and the Askey table would
explain the occurrence of many features that have been discovered
recently such as the appearance of basic hypergeometric functions in their classical solutions.
We have already referred to the correspondence between the classical orthogonal polynomial
systems (OPS) Hermite, Laguerre and Jacobi and the classical solutions to
Painlev\'e IV, V and VI respectively. Amongst the difference and $q$-difference OPS the
current evidence for such a correspondence can be summarised in the Table \ref{AT_SS_correspondence}. 

\begin{table}[h!]
\renewcommand{\arraystretch}{1.2}
\begin{tabular}{|c|c|c|}\hline
 Base OPS & Integrable system & Reference \\
\hline
 little $q$-Jacobi & $D^{(1)}_5$, full $q-{\rm P}_{\rm VI}$ & \cite{OWF_2010} \\
\hline
 Pastro \cite{Pa_1985} & $D^{(1)}_5$, full $q-{\rm P}_{\rm VI}$ & \cite{Bi_2009} \\
\hline
 little $q$-Jacobi & $D^{(1)}_5$, special $q-{\rm P}_{\rm VI}$ & \cite{BB_2003} \\
\hline
 $q$-Krawtchouk & $D^{(1)}_5$, special $q-{\rm P}_{\rm VI}$ & \cite{BB_2003} \\
\hline
 $q$-Charlier & $D^{(1)}_5$, degenerate $q-{\rm P}_{\rm VI}$ & \cite{BB_2003} \\
\hline
 $q$-Freud & $A^{(1)}_4$, degenerate $q-{\rm P}_{\rm V}$ & \cite{BSvA_2008} \\
\hline
\hline
 Meixner, Krawtchouk & $d-{\rm P}_{\rm V}$ & \cite{BB_2003} \\
\hline
 Charlier & $d-{\rm P}_{\rm IV}$ & \cite{BB_2003} \\
\hline
\end{tabular}
\vskip0.5cm
\caption{The Askey table-Sakai scheme correspondence ranked according to the degeneration pathway from the master case.
The affine Weyl group refers to the symmetry group of B\"acklund transformations for the integrable system.
Only the discrete lattices are included here.}\label{AT_SS_correspondence}
\end{table}

The layout of this work loosely follows the plan given above. 

In Section \ref{Qlattices} we describe the classification of the quadratic lattices and their
divided-difference calculus, with special emphasis on the example of the master class, the
$q$-quadratic lattice. 
Section \ref{OPS_NUL} is devoted to the reformulation of orthogonal polynomial system theory
necessary for systems with a weight having support on a general quadratic lattice, and 
some analytical results for the $q$-quadratic lattice.
The spectral structures are laid out in Section\ref{SpectralS} for a general quadratic lattice, starting with the definition of
a $\ddoAW$-semi-classical weight and developing the consequences of this for the 
orthogonal polynomial system. We also describe the compatibility relations of this
structure with the three-term recurrence relations, and this leads to a generalisation of
the Laguerre-Freud equations.
An explicit example of the forgoing theory is the lowest case, and from our point of view the trivial
case, of the $q$-quadratic lattice which is denoted by the label $ M=2 $. This is dealt with in Section \ref{M=2AW}. 
We recover the Askey-Wilson OPS and demonstrate that every aspect of this system can be derived
in an efficient manner using the theoretical tools developed in the previous section.

In Section \ref{DeformS} we introduce an (or many) auxiliary variable and demand that the
weight satisfies an analogous $\ddoAW$-semi-classical relation to the spectral relation,
on a general lattice not necessarily the same as that for the spectral variable. 
This constitutes a {\it deformation} of the weight which leaves it having the same structure
as the original.
From this we develop a parallel theoretical analysis of the orthogonal polynomial system
with respect to this deformation. We now have two compatibility relations with this deformation
structure - one with the three-term recurrence and other with the spectral structure.
For the present purposes this completes our task for the general theory.

In Section \ref{M=3AW} we treat a natural deformation of the Askey-Wilson weight, but not by any means the only one,
to the $q$-quadratic lattice. Here we construct parameterisations of the spectral and
deformation matrices, the closure relations between the two structures, determine the co-ordinate transformations to 
a set of variables and derive the evolution equations for this system
(see Propositions \ref{evolution_1}, \ref{evolution_2},\ref{evolution_3} and \ref{evolution_4}).
Finally we offer evidence that our system is a classical solution of the $ E^{(1)}_7 $ $q$-Painlev\'e system.

\section{Divided-Difference Calculus of the Quadratic Lattices}\label{Qlattices}
\setcounter{equation}{0}

Divided-difference operators associated with the special non-uniform lattices have appeared in
many studies of orthogonal polynomials of a discrete variable. For example see the early 
studies by Hahn \cite{Ha_1949a,Ha_1949b,Ha_1950,Ha_1952},
the foundational work by Askey and Wilson \cite{AW_1985}
and the monograph of Nikiforov, Suslov and Uvarov \cite{ops_NSU}.
The employment of divided-difference operators such as the Askey-Wilson and Wilson operators
has been common place in studies of the Askey table of hypergeometric orthogonal polynomial systems
too numerous to cite here (see the monographs by Ismail \cite{Ismail_2005} and Lesky \cite{Lesky_2005}). 

Let $ \Pi_{n}[x] $ denote the linear space of polynomials in $ x $ over $ \C $ with
degree at most $ n\in\Z_{\geq 0} $.
In pioneering investigations Magnus \cite{Ma_1988,Ma_1995} provided a geometrical understanding 
of these lattices and their divided-difference operators which we relate here
briefly. If we define the {\it divided-difference operator} (DDO) $ \ddoAW_{x} $ by
\begin{equation}
   \ddoAW_{x} f(x) = \frac{f(\iota_{+}(x))-f(\iota_{-}(x))}{\iota_{+}(x)-\iota_{-}(x)} ,
\end{equation}
then a simple consequence of the condition that $ \ddoAW_{x}: \Pi_{n}[x] \to \Pi_{n-1}[x] $ 
for all $ n \in \N $ is that $ \iota_{\pm}(x) $ are the two $y$-roots of the quadratic equation
\begin{equation}
   \mathcal{A} y^2+2\mathcal{B} xy+\mathcal{C} x^2+2\mathcal{D} y+2\mathcal{E} x+\mathcal{F} = 0 .
\label{snul_Quad}
\end{equation}
The functions $ \iota_{\pm}(x) $ satisfy
\begin{gather}
   \iota_{+}(x)+\iota_{-}(x) = -2\frac{\mathcal{B}x+\mathcal{D}}{\mathcal{A}} ,
 \\
   \iota_{+}(x)\iota_{-}(x) = \frac{\mathcal{C}x^2+2\mathcal{E}x+\mathcal{F}}{\mathcal{A}} ,  
\end{gather}
and their inverse functions $ \iota_{\pm}^{-1} $ are defined by $ \iota_{\pm}^{-1}(\iota_{\pm}(x)) = x $.
For a given $ y $-value the quadratic (\ref{snul_Quad}) defines two $ x $-roots if $ \mathcal{C}\neq 0 $,
which are consecutive points on the {\it $ x $-lattice}, $ x_s:=x(s), x_{s+1}:=x(s+1) $ 
parameterised by the variable $ s $ and therefore defines a map $ x_s \mapsto x_{s+1} $.
Conversely for a given $ x $-value the quadratic defines two $ y $-roots if $ \mathcal{A}\neq 0 $ which 
are consecutive points on a {\it dual lattice}, the {\it $ y $-lattice}, 
$ y_s:=y(s)=\iota_{-}(x(s)), y_{s+1}:=y(s+1)=\iota_{+}(x(s)) $
which generally is distinct from the $ x $-lattice.
We will employ an operator notation for the mappings from points on the direct lattice to the
dual lattice $ E^{\pm}_{x} f(x):= f(\iota_{\pm}(x)) $ so that
\begin{equation}
   \ddoAW_{x} f(x)
 := \frac{f(\iota_{+}(x))-f(\iota_{-}(x))}{\iota_{+}(x)-\iota_{-}(x)}
  = \frac{E^{+}_{x}f-E^{-}_{x}f}{E^{+}_{x}x-E^{-}_{x}x} ,
\end{equation}
for arbitrary functions $ f(x) $.
The inverse functions $ \iota_{\pm}^{-1}(x) $ define operators $ (E^{\pm})^{-1} $ which map
points on the dual lattice to the direct lattice and also an adjoint to the 
divided-difference operator $ \ddoAW_{x} $ 
\begin{equation}
   \ddoAW_{x}^{*} f(x)
 := \frac{f(\iota_{+}^{-1}(x))-f(\iota_{-}^{-1}(x))}{\iota_{+}^{-1}(x)-\iota_{-}^{-1}(x)}
  = \frac{(E^{+}_{x})^{-1}f-(E^{-}_{x})^{-1}f}{(E^{+}_{x})^{-1}x-(E^{-}_{x})^{-1}x} .
\end{equation}
The composite operators 
$ E_{x}:=(E^{-}_{x})^{-1} E^{+}_{x} $ and $ E^{-1}_{x}=(E^{+}_{x})^{-1} E^{-}_{x} $
map between consecutive points on the direct lattice.

However in the situation of a symmetric quadratic $ \mathcal{A}=\mathcal{C} $ and $ \mathcal{D}=\mathcal{E} $, which
entails no loss of generality, then we have $ (E^{+}_{x})^{-1}=E^{-}_{x} $
and $ (E^{-}_{x})^{-1}=E^{+}_{x} $. Consequently there is no distinction
between the divided-difference operator and its adjoint, and hereafter 
we adopt this simplification. For example one useful consequence of this choice is that the ratio
\begin{equation}
   \frac{(x-E^{+}u)(x-E^{-}u)}{(E^{-}x-u)(E^{+}x-u)}
 = \frac{ \mathcal{A} x^2+2\mathcal{B} xu+\mathcal{C} u^2+2\mathcal{D} x+2\mathcal{E} u+\mathcal{F}}
        { \mathcal{A} u^2+2\mathcal{B} ux+\mathcal{C} x^2+2\mathcal{D} u+2\mathcal{E} x+\mathcal{F}} = 1 ,
\end{equation}
for all $ x,u $.
A companion operator to the divided-difference operator $ \ddoAW_{x} $ is the mean or
average operator $ \moAW_{x} $ defined by
\begin{equation}
   \moAW_{x} f(x) = \tfrac{1}{2}\left[ f(\iota_{+}(x))+f(\iota_{-}(x)) \right] ,
\end{equation}
so that the property $ \moAW_{x}: \Pi_{n}[x] \to \Pi_{n}[x] $ is ensured by the condition
we imposed upon $ \ddoAW_{x} $.

\begin{definition}
Henceforth we use the 
shorthand for the difference in consecutive $ y $ points $ \Delta y(x):= \iota_{+}(x)-\iota_{-}(x) $.
We introduce the notion of "fixed points", whereby $ \Delta y^2(x_F) = 0 $, or $ \iota_{+}(x_F) = \iota_{-}(x_F) $
and are given as the roots of 
$ (\mathcal{B}^2-\mathcal{A}\mathcal{C})x_F^2+2(\mathcal{B}\mathcal{D}-\mathcal{A}\mathcal{E})x_F+\mathcal{D}^2-\mathcal{A}\mathcal{F}=0 $.
\end{definition}

Assuming $ \mathcal{A}\mathcal{C} \neq 0 $
one can classify these non-uniform quadratic lattices (or {\it SNUL}, special non-uniform lattices) 
according to two parameters - the discriminant $ \mathcal{B}^2-\mathcal{A}\mathcal{C} $ and 
\begin{equation}
    \Theta = \det
               \begin{pmatrix} \mathcal{A} & \mathcal{B} & \mathcal{D} \\ \mathcal{B} & \mathcal{C} & \mathcal{E} \\ \mathcal{D} & \mathcal{E} & \mathcal{F} 
               \end{pmatrix} ,
\end{equation}
or $ \mathcal{A}\Theta = (\mathcal{B}^2-\mathcal{A}\mathcal{C})(\mathcal{D}^2-\mathcal{A}\mathcal{F})-(\mathcal{B}\mathcal{D}-\mathcal{A}\mathcal{E})^2 $.
There are four primary classes - 
the linear lattice,
the linear $q$-lattice,
the quadratic lattice and
the $q$-quadratic lattice, which are given in Table \ref{SNUL}. 
The $q$-quadratic lattice, in its general non-symmetrical form, is the most general 
case and the other lattices can be found from this by limiting processes.

\begin{table}[h!]
\renewcommand{\arraystretch}{1.2}
\begin{tabular}{|c|c|c|c|c|c|}\hline
$ \mathcal{B}^2-\mathcal{A}\mathcal{C} $ & $ \Theta $	& conic	& lattice & canonical DDO & Notes \\
\hline
$ 0 $ & $ 0 $ & parallel lines & linear & forward difference & 
 \\
\hline
$ >0 $ & $ 0 $ & intersecting lines & $q$-linear & $q$-difference & 
 \\
\hline
$ 0 $ & $ <0 $ & parabola & quadratic & Wilson &
\\
\hline
$ >0 $ & $ <0 $ & hyperbola & $q$-quadratic & Askey-Wilson &  
$ q $ real \\
\hline
$ <0 $ & $ <0 $ & ellipse & $q$-quadratic & Askey-Wilson & 
$ |q|=1 $ \\
\hline
\end{tabular}
\vskip0.5cm
\caption{The non-uniform lattices of quadratic type.}\label{SNUL}
\end{table}

This classification of lattices for polynomial systems can be extended to rational 
function systems \cite{SZ_2001,SZ_2007,Ma_2009} and in this case one has a lattice characterised by a bi-quadratic
relation and parameterised by elliptic functions \cite{IR_2002} (see also Section 15.10 of \cite{Ba_1982}). 
However for the purposes of studying the Askey table we will not pursue this direction.

For the quadratic class of lattices the parameterisation on $ s $ can be made explicit
through the trigonometric/hyperbolic functions or their degenerations so we can employ
a parameterisation such that $ \iota_{-}(x(s))=y(s)=x(s-1/2) $ and $ \iota_{+}(x(s))=y(s+1)=x(s+1/2) $.
We denote the totality of lattice points
by $ G[x] := \{x(s):s\in \Z\} $ with the point $ x(0)=x $ as the {\it basal point},
and of the dual lattice by $ \tilde{G}[x] := \{x(s):s\in \Z\+\frac{1}{2}\} $.

Having established the basic properties of the divided-difference operators 
we can deduce key elements of their calculus. 
A consequence of the general definition of the divided-difference operators are
the following identities
\begin{enumerate}
\item
the product or Leibniz formulae
\begin{gather}
  \ddoAW_{x} fg = \ddoAW_{x} f \moAW_{x} g + \moAW_{x} f \ddoAW_{x} g,
\label{DD_calculus:a}\\
  \moAW_{x} fg =  \moAW_{x} f\moAW_{x} g + \tfrac{1}{4}\Delta y^2\ddoAW_{x} f\ddoAW_{x} g,
\label{DD_calculus:b}
\end{gather}
\item
the inverse formulae
\begin{equation}
  \moAW_{x} \frac{1}{f} = \frac{\moAW_{x} f}{E^{+}_{x}fE^{-}_{x}f},
  \qquad
  \ddoAW_{x} \frac{1}{f} = -\frac{\ddoAW_{x} f}{E^{+}_{x}fE^{-}_{x}f},
\label{DD_calculus:c}
\end{equation}
\item
and the commutativity formulae
\begin{equation}
  \moAW_{u} \moAW_{x} = \moAW_{x} \moAW_{u} , \qquad 
  \moAW_{u} \ddoAW_{x} = \ddoAW_{x} \moAW_{u} ,\qquad 
  \ddoAW_{u} \ddoAW_{x}  = \ddoAW_{x} \ddoAW_{u} .
\label{DD_calculus:d}
\end{equation} 
\end{enumerate}

The other side of our divided-difference calculus concerns the definition and 
properties of analogs to integrals. We define the {\it $ \ddoAW $-Integral} of a 
function defined on the $x$-lattice $ f: G[x] \to \C $ with basal point $ x_0 $ by the 
Riemann sum over the lattice points
\begin{equation}
   I[f](x_0) = \int_{G}\, \ddoAW x\, f(x) 
  := \sum_{s\in \Z} (\iota_{+}(s)-\iota_{-}(s))f(x(s))
   = \sum_{s\in \Z} \Delta y(x_s)f(x_s) ,
\end{equation} 
where the sum is either a finite subset of $ \Z $, namely $ \{ 0,\ldots,\Gt \} $, or $ \Z_{\geq 0} $, $ \Z $.
This definition reduces to the usual definition of the difference integral and
the Thomae-Jackson $q$-integrals in the canonical forms of the linear and 
$q$-linear lattices respectively. 
A number of properties flow from this definition -
\begin{enumerate}
\item
an analog of the fundamental theorem of calculus 
\begin{equation}
   \int_{x_0\leq x_{s} \leq x_{\Gt}} \ddoAW x\, \ddoAW_{x} f(x) = f(E^{+}_{x}x_{\Gt})-f(E^{-}_{x}x_{0}) ,
\end{equation}
\item
an analog of integration by parts for two functions $ f(x),g(x) $
\begin{equation}
  \int_{x_0\leq x_{s} \leq x_{\Gt}}\ddoAW x\,f(x)\ddoAW_{x}g(x)
   = -\int_{x_0\leq x_{s} \leq x_{\Gt}}\ddoAW({E^{+}_{x}x})\,\ddoAW_{x}f(E^{+}_{x}x)g(E^{+}_{x}x)
      +f(E_{x}^{+2}x_{\Gt})g(E_{x}^{+}x_{\Gt})-f(x_{0})g(E_{x}^{-}x_{0}) ,
\end{equation}
\item
and the parameterisation "invariance" property
\begin{align}
   \int_{x_0\leq x_{s} \leq x_{\Gt}}\ddoAW x\,f(x)
  & =  \int_{x_0\leq x_{s} \leq x_{\Gt}}\ddoAW (E_{x}^{+2}x)\,f(E_{x}^{+2}x)+\Delta y(x_{0})f(x_{0})-\Delta y(E_{x}^{+2}x_{\Gt})f(E_{x}^{+2}x_{\Gt}) ,
  \\
  & =  \int_{x_0\leq x_{s} \leq x_{\Gt}}\ddoAW (E_{x}^{-2}x)\,f(E_{x}^{-2}x)-\Delta y(E_{x}^{-2}x_{0})f(E_{x}^{-2}x_{0})+\Delta y(x_{\Gt})f(x_{\Gt}) .
\end{align}
\end{enumerate}

We will apply our theory to the case of the $q$-quadratic lattice and the Askey-Wilson
divided-difference calculus, and in order to simplify the description and to conform to
convention we will employ the canonical, that is to say the centred and symmetrised forms
of the lattice and the divided-difference operators.
Let us define the base $ q=\exp(2i\eta) $ although we will not restrict ourselves to 
$q$-domains such as $ 0< \Re(q) < 1 $ except to avoid special degenerate 
cases and to ensure convergence.
Consider the projection map from the unit circle $ z=e^{i\theta}, \quad \theta \in [-\pi,\pi) $
onto $ [-1,1] $ by $ x=\tfrac{1}{2}(z+z^{-1})=\cos\theta \in [-1,1] $.
We denote the unit circle by $ \mathbb{T} $ and the unit open disc by $ \mathbb{D} $.
The inverse of the projection map defines a two-sheeted Riemann surface, one of which
corresponds to the interior of the unit circle, and the other to the exterior. Thus we take
$ x $-plane to be cut along $ [-1,1] $ and will usually give results for the second sheet
i.e. when $ |x|\to \infty $ as $ z\to \infty $.
In the symmetrised and canonical form of the lattice we have
$ \mathcal{A} = \mathcal{C} $, arbitrary and non-zero, $ \mathcal{B}=-\cos\eta\mathcal{A} $,
$ \mathcal{D}=\mathcal{E}=0 $, $ \mathcal{F}=-\sin^2\eta\mathcal{A} $, and
$ \theta=2s\eta $. 
Define the shift operators $ E^{\pm}_{x} $ by
$ E^{\pm}_{x}f(x) = f(\tfrac{1}{2}[q^{1/2}z+q^{-1/2}z^{-1}]) $
and set $ y_{\pm} = E^{\pm}_{x} x $. 
This implies that 
\begin{gather}
   y_{+}+y_{-} = (q^{1/2}+q^{-1/2})x = 2\cos\eta\; x ,
\\
   \Delta y := y_{+}-y_{-} = \tfrac{1}{2}(q^{1/2}-q^{-1/2})(z-z^{-1}) = -2\sin\eta \sin\theta ,
\\
   \Delta y^2 = (q^{1/2}-q^{-1/2})^2(x^2-1) ,
\\
   y_{+}y_{-} = x^2+\tfrac{1}{4}(q^{1/2}-q^{-1/2})^2 = x^2-\sin^2\eta .
\end{gather}
The Askey-Wilson divided-difference operators are defined as
\begin{align}
  \ddoAW_{x}f(x) & = \frac{f(\tfrac{1}{2}[q^{1/2}z+q^{-1/2}z^{-1}])-f(\tfrac{1}{2}[q^{-1/2}z+q^{1/2}z^{-1}])}
                             {\tfrac{1}{2}(q^{1/2}-q^{-1/2})(z-z^{-1})}
\label{DDoperator:a} \\
  \moAW_{x}f(x) & = \tfrac{1}{2}\left[f(\tfrac{1}{2}[q^{1/2}z+q^{-1/2}z^{-1}])+f(\tfrac{1}{2}[q^{-1/2}z+q^{1/2}z^{-1}])\right]
  \nonumber
\label{DDoperator:b}
\end{align} 
There is an explicit parameterisation of the $q$-quadratic lattice 
\begin{gather}
  x(s) = \tfrac{1}{2}(q^{s}+q^{-s}) = \cos(2\eta s) ,
 \\
  y_{\pm}(s) = \tfrac{1}{2}(q^{s\pm 1/2}+q^{-s\mp 1/2}) = \cos(2\eta [s\pm \tfrac{1}{2}]) .
\end{gather}
Here the direct lattice is 
$ G[x=\tfrac{1}{2}(a+a^{-1})] = \{\tfrac{1}{2}(q^{r/2}a+q^{-r/2}a^{-1}): r\in 2\Z\} $.
For $ |q|=1 $ and $ \eta $ not a rational multiple of $ \pi $ then the lattice densely fills
the interval $ [-1,1] $. If $ \eta $ is a rational multiple of $ \pi $ then one 
has the root of unity case $ q^N=1 $ and a finite lattice.
In the generic case we will assume that we are dealing with functions $ f(x) $ in the class where
\begin{equation}
	\int\ddoAW x\,f(x) = \frac{\sin\eta}{\eta}\int^{1}_{-1}dx\,f(x) ,
\end{equation}
is applicable.
The reader should note that we will not distinguish a function of $ x $, $ f(x) $, from the
function of $ z $, $ \check{f}(z) = f(\tfrac{1}{2}[z+z^{-1}]) $ as done by some authors, 
and it should be clear from the context which is meant.

\section{Orthogonal Polynomial Systems on the Non-uniform Lattice}\label{OPS_NUL}
\setcounter{equation}{0}

\subsection{General orthogonal polynomial systems}
Our study requires the revision of a number of standard results in orthogonal polynomial 
theory \cite{ops_Sz,ops_Fr,Ismail_2005} so we recount our formulation.
Let $ \{l_n(x;a)\}^{\infty}_{n=0} $ be a polynomial basis of $ L^2(w(x)\ddoAW x,G) $, 
where $ l_n $ is of exact degree $ n $ and the support is $ G = \{E^{+k}_x x: k \in 2\Z \} $
or if finite $ G = \{ x_{0},\ldots,x_{\Gt} \} $
and $ a $ denotes the set of parameters characterising the lattice. 
The appropriate canonical basis is dependent on the lattice type through the general requirements 
that
\begin{enumerate}
\item
$ l_n $ is of precise degree $ n $ so that 
$ l_n(x;a) = g_n(a)x^n + {\rm O}(x^{n-1})$ with $ g_n \neq 0 $,
\item
$ \ddoAW_x $ is an exact lowering operator in this basis
\begin{equation}
   \ddoAW_{x} l_n(x;a) = c_n(a) l_{n-1}(x;a') ,
\label{SNUL:b}
\end{equation}
where $ c_n $ is constant with respect to $ x $ and the transformed parameter set
$ a' $ is related to the original $ a $ depending on the lattice type. 
\end{enumerate}
We will also require the linearisation formula
$ xl_n(x;a) = d_n(a)l_{n+1}(x;a)+e_n(a)l_{n}(x;a) $.
A general solution to the two requirements above is the following product expression
\begin{equation}
   l_n(x;a) = g_n(a)\prod^{n-1}_{k=0}[x-(E^{+}_{x})^{2k} x(a)] ,
\end{equation}
where the basal point $ x(a) $ is parameterised by $ a $.
We note that for some lattices the limit $ \lim_{n\to\infty}l_n(x;a) $ exists and in 
this case we denote it by $ l_{\infty}(x;a) $.
For the classes of quadratic non-uniform lattices the basis choices are tabulated in 
Table \ref{snul_basis}.

\begin{table}[h!]
\renewcommand{\arraystretch}{1.2}
\begin{tabular}{|c|c|c|c|}\hline
DDO & Lattice Type & Basis $ l_r $ & Notes \\
\hline
\begin{minipage}[c][1.2cm][c]{1cm}
{\begin{equation*}
    \frac{d}{dx}
 \end{equation*}}
\end{minipage}
 & continuous & $ x^r $ & $ c_r=r $ \\
\hline
$ \Delta_x $ & linear &
\begin{minipage}[c][1.5cm][c]{6cm}
{\begin{equation*}
   x^{(r)} = \prod^{r-1}_{k=0}(x-k) = \frac{\Gamma(x+1)}{\Gamma(x-r+1)} 
 \end{equation*}}
\end{minipage} & 
$ c_r=r $\\
\hline
$ D_q $ & $q$-linear &
\begin{minipage}[c][1cm][c]{6cm}
{\begin{equation*}
  (ax;q)_r  = \prod^{r-1}_{k=0}(1-aq^{k}x) = \frac{(ax;q)_{\infty}}{(aq^{r}x;q)_{\infty}} 
 \end{equation*}}
\end{minipage} & 
\begin{minipage}[c][2.0cm][c]{2cm}
{\begin{gather*}
   c_r=-\frac{1-aq^r}{q-1} \\ a'=qa
 \end{gather*}}
\end{minipage} \\
\hline
$ W $ & quadratic &
\begin{minipage}[c][2cm][c]{8cm}
{\begin{equation*}
   \prod^{r-1}_{k=0}\left[ x+(k+a)^2 \right] = \frac{\Gamma(r+a-i\sqrt{x})\Gamma(r+a+i\sqrt{x})}{\Gamma(a-i\sqrt{x})\Gamma(a+i\sqrt{x})}
 \end{equation*}}
\end{minipage} &
\begin{minipage}[c][2cm][c]{2cm}
{\begin{gather*}
   c_r= r \\ a'=a+\tfrac{1}{2}
 \end{gather*}}
\end{minipage} \\
\hline
$ \ddoAW_{x} $ & $q$-quadratic &
\begin{minipage}[c][2cm][c]{5cm}
{\begin{equation*}
   (az,az^{-1};q)_r = \frac{(az,az^{-1};q)_{\infty}}{(aq^rz,aq^rz^{-1};q)_{\infty}}
 \end{equation*}}
\end{minipage} &
\begin{minipage}[c][4.3cm][c]{3.5cm}
{\begin{gather*}
   c_r=-2a\frac{q^r-1}{q-1} \\ a'=q^{1/2}a \\ g_r=(-2a)^rq^{\frac{1}{2}r(r-1)} \\ d_r=-\frac{1}{2aq^r} \\ e_r=\tfrac{1}{2}(aq^r+a^{-1}q^{-r}) 
 \end{gather*}}
\end{minipage} \\
\hline
\end{tabular}
\vskip0.5cm
\caption{Canonical Bases for the non-uniform lattices of quadratic type.}\label{snul_basis}
\end{table}
 
Consider the general {\it orthogonal polynomial system} $ \{p_n(x)\}^{\infty}_{n=0} $
defined by the orthogonality relations
\begin{equation}
  \int_{G} \ddoAW x\; w(x)p_n(x)\,l_m(x;b) = \begin{cases} 0 \quad 0\leq m<n \\ h_n(b) \quad m=n \end{cases},
  \quad n \geq 0 ,
\label{ops_orthog}
\end{equation}
with $ G $ denoting the support of the {\it weight} $ w(x) $. 
Our system of orthogonal polynomials and their associated functions (to be defined in
(\ref{ops_eps})) have a distinguished singular point at $ x=\infty $ and possess
expansions about this point which can characterise solutions uniquely.
This is related to the fact that orthogonal polynomials are the denominators of
single point Pad\'e approximants and that point is conventionally set at 
$ x=\infty $.  
We give special notation for the coefficients of $ x^{n} $ and $ x^{n-1} $ in $ p_n(x) $,
\begin{equation}
   p_n(x) = \gamma_n x^{n} + \gamma_{n,1}x^{n-1} + \ldots,
  \quad n \geq 0 .
\label{ops_poly}
\end{equation}
The corresponding {\it monic} polynomials are then
$ \pi_n(x) = \gamma_n^{-1}p_n(x) $ given that $ n \geq 0 $.
A consequence of the orthogonality relation is the three term recurrence relation
\begin{equation}
   a_{n+1}p_{n+1}(x) = (x-b_n)p_n(x) - a_np_{n-1}(x), \quad n \geq 0,
\label{ops_threeT}
\end{equation}
and we consider the set of orthogonal polynomials with initial values $ p_{-1} = 0 $ 
and $ p_0 = \gamma_0 $.
The three term recurrence coefficients are related to the leading and sub-leading
polynomial coefficients by \cite{ops_Sz,ops_Fr}
\begin{equation}
   a_n = \frac{\gamma_{n-1}}{\gamma_n}, \quad
   b_n = \frac{\gamma_{n,1}}{\gamma_n}-\frac{\gamma_{n+1,1}}{\gamma_{n+1}}, \quad n \geq 1 .
\label{ops_coeffRn}
\end{equation}
The initial values of the recurrence coefficients are
\begin{equation}
   b_0 = -\frac{\gamma_{1,1}}{\gamma_{1}}, \quad \gamma_{0,1}=0 .
\label{ops_coeffRn1}
\end{equation}
and where $ a_0 $ is not fixed by the initial polynomials but rather by the initial
associated functions (see after (\ref{ops_AF_threeT}). 

The orthogonality relation (\ref{ops_orthog})
is derived from the linear functional on the space of polynomials
$ {\mathcal L}: p\in \Pi \mapsto \C $
and we employ our basis polynomials as an expansion basis although not necessarily with the
same parameter as above
\begin{equation}
   p_n(x) = \sum^{n}_{k=0} c_{n,k}(a)l_k(x;a)
  \quad n \geq 0 .
\label{ops_expand}
\end{equation}
Consequently we define the {\it moments} $ \{m_{j,k}\}_{j,k=0,1,\ldots,\infty} $ of 
the weight as the action of this functional on products of the basis polynomials, 
defined as
\begin{equation}
   m_{j,k}(b,a) := \int_{G} \ddoAW x\; w(x)\,l_{j}(x;b)l_{k}(x;a) ,
  \quad j,k \geq 0 .
\label{ops_moment}
\end{equation}
Central objects in our theory are the {\it Moment determinants}
\begin{equation}
   \Delta_n := \det[ m_{j,k} ]_{j,k=0,\ldots,n-1}, \quad n\geq 1,
   \quad \Delta_0 := 1 ,
\label{ops_Hdet}
\end{equation}
and
\begin{equation}
   \Sigma_{n,j} := \det\left(
               \begin{array}{ccccccc}
               m_{0,0}   & \cdots & m_{0,j-1}   & []     & m_{0,j+1}   & \cdots & m_{0,n}   \\
               \vdots    & \vdots & \vdots      & \vdots & \vdots      & \cdots & \vdots    \\
               m_{n-1,0} & \cdots & m_{n-1,j-1} & []     & m_{n-1,j+1} & \cdots & m_{n-1,n} \\
               \end{array} \right) , \quad n\geq 1, j=0,\ldots n-1 .
\label{ops_Sdet}
\end{equation}
defined in terms of the moments above. Obviously $ \Delta_n = \Sigma_{n,n} $ and 
we set $ \Sigma_{0,0} := 0 $.
The expansion coefficients are given in terms of these determinants
\begin{equation}
   c_{n,j}(a) = (-)^{n+j}h_n(b)\frac{\Sigma_{n,j}}{\Delta_{n+1}} ,\quad
   c_{n,n}(a) = h_n(b)\frac{\Delta_{n}}{\Delta_{n+1}} .
\label{ops_expCff}
\end{equation} 
It follows from (\ref{ops_orthog}) that
\begin{equation*}
  \int_{G} \ddoAW x\; w(x)[p_n(x)]^2 = c_{n,n}(b)h_n(b),
  \quad n \geq 0 ,
\end{equation*}
and thus for $ p_n(x) $ to be normalised as well as orthogonal we set $ c_{n,n}(a)h_n(a)=1 $.
We have moment determinant representations of the polynomials
\begin{equation}
   p_{n}(x) = \frac{c_{n,n}(a)}{\Delta_{n}}
               \det\left(
               \begin{array}{cccccc}
               m_{0,0}   & \cdots & m_{0,j}   & \cdots & m_{0,n}   \\
               \vdots    & \vdots & \vdots    & \cdots & \vdots    \\
               m_{n-1,0} & \cdots & m_{n-1,j} & \cdots & m_{n-1,n} \\
               l_{0}     & \cdots & l_{j}     & \cdots & l_{n}     \\
               \end{array} \right) , \quad n\geq 0.
   \label{ops_polyDet}
\end{equation}
The three-term recurrence coefficients are related to these determinants
\begin{align}
   a^2_n & = d_{n-1}(a)d_{n-1}(b)\frac{\Delta_{n+1}\Delta_{n-1}}{\Delta^2_{n}} , \quad n\geq 1 ,
   \label{ops_aSQDelta}\\
   b_n & = e_{n}(a)+d_{n}(a)\frac{\Sigma_{n+1,n}}{\Delta_{n+1}}-d_{n-1}(a)\frac{\Sigma_{n,n-1}}{\Delta_n} , \quad n\geq 0 , 
   \label{ops_bDelta}\\
   \gamma_n^2 & = g_n(a)g_n(b)\frac{\Delta_{n}}{\Delta_{n+1}}, \quad n\geq 0 ,
   \label{ops_gammaDelta}
\end{align}
where each coefficient is independent of the choices of $ a, b $ as can be easily
verified from their determinantal definitions given previously.

Another set of polynomial solutions to the three term recurrence relation are the 
{\it associated polynomials} $ \{p^{(1)}_n(x)\}^{\infty}_{n=0} $, defined by
\begin{equation}
   p^{(1)}_{n-1}(x) := \int_{G} \ddoAW y\; w(y)\frac{p_n(y)-p_n(x)}{y-x}, \quad n \geq 0 .
\label{ops_assoc}
\end{equation}
In particular these polynomials satisfy
\begin{equation}
   a_{n+1}p^{(1)}_{n}(x) = (x-b_n)p^{(1)}_{n-1}(x)-a_np^{(1)}_{n-2}(x) ,
\label{ops_assoc_threeT}
\end{equation}
with the initial conditions $ p^{(1)}_{-1}(x)=0 $, $ p^{(1)}_{0}(x)=m_{0,0}\gamma_{1} $.
Note the shift by one decrement in comparison to the three-term recurrence 
(\ref{ops_threeT}) for the polynomials $ \{p_n(x)\}^{\infty}_{n=0} $. 
We also need the definition of the {\it Stieltjes function}
\begin{equation}
   f(x) \equiv \int_{G} \ddoAW y\;\frac{w(y)}{x-y} , \quad x \notin G ,
\label{ops_stieltjes}
\end{equation}
which is a moment generating function in the following sense
\begin{equation}
   f(x) = \frac{f_{\infty}(x;a)}{l_{\infty}(x;a)}+\sum^{\infty}_{n=0} \frac{m_{0,n}(a)}{d_n(a)l_{n+1}(x;a)}, \quad x \notin G,\quad x \to \infty ,
\label{Sf_expand}
\end{equation}
and splits into two parts - one part being a series with inverse basis polynomials and a remainder,
which may be absent for some lattices.
We define non-polynomial {\it associated functions} or {\it functions of the second kind} 
$ \{q_n(x)\}^{\infty}_{n=0} $ by
\begin{equation}
   q_n(x) := f(x)p_n(x)-p^{(1)}_{n-1}(x),
  \quad n \geq 0 ,
\label{ops_eps}
\end{equation}
which also satisfy the three term recurrence relation (\ref{ops_threeT}), namely
\begin{equation}
   a_{n+1}q_{n+1}(x) = (x-b_n)q_{n}(x)-a_nq_{n-1}(x),
  \quad n \geq 0 ,
\label{ops_AF_threeT}
\end{equation}
subject to the initial values $ q_{-1}(x)=1/a_0\gamma_0 $, $ q_{0}(x)=\gamma_0f(x) $.
The initial value of $ a_0 $ is irrelevant and therefore arbitrary in so far as the 
polynomials are concerned, however many relations for the whole system extend from 
$ n \geq 1 $ to include $ n=0 $ if we allow this to be finite, non-zero and satisfying
the above initial condition.
The associated functions also have a determinantal representation
\begin{equation}
   q_{n}(x) = \frac{c_{n,n}(a)}{\Delta_{n}}
               \det\left(
               \begin{array}{ccccc}
               m_{0,0}   & \cdots & m_{0,j}   & \cdots & m_{0,n}   \\
               \vdots    & \vdots & \vdots    & \cdots & \vdots    \\
               m_{n-1,0} & \cdots & m_{n-1,j} & \cdots & m_{n-1,n} \\
               f_{0}     & \cdots & f_{j}     & \cdots & f_{n}     \\
               \end{array} \right) , \quad n\geq 0,
   \label{ops_Afun}
\end{equation}
where 
\begin{equation}
    f_{j}(x;a) := \int_{G} \ddoAW y\; w(y)\frac{l_j(y;a)}{x-y}, \quad j \geq 1 ,\quad f_0=g_0 f .
\label{ops_}
\end{equation}
Likewise this function has an expansion analogous to (\ref{Sf_expand})
\begin{equation}
   f_j(x;b) = \frac{f_{\infty,j}(x;b)}{l_{\infty}(x;a)}+\sum^{\infty}_{n=0} \frac{m_{j,n}(b,a)}{d_n(a)l_{n+1}(x;a)}, \quad x \notin G,\quad x \to \infty .
\label{Sfj_expand}
\end{equation}

The polynomials and their associated functions satisfy the {\it Casoratian relation}
\begin{equation}
   p_n(x)q_{n-1}(x)-p_{n-1}(x)q_{n}(x) = \frac{1}{a_n} , \quad n \geq 0.
\label{ops_casoratian}
\end{equation}

Central to our analysis is a composite of polynomial and non-polynomial solutions
of (\ref{ops_threeT}), the $ 2\times 2 $ matrix variable
\begin{equation}
   Y_n(x) = \begin{pmatrix} p_n(x) & \frac{\displaystyle q_{n}(x)}{\displaystyle w(x)} \\
                              p_{n-1}(x) & \frac{\displaystyle q_{n-1}(x)}{\displaystyle w(x)}
              \end{pmatrix},
  \quad n \geq 0 .
\label{ops_Ydefn}
\end{equation}
We will refer to this as the orthogonal polynomial system (OPS). From (\ref{ops_casoratian})
we note that
\begin{equation}
   \det Y_n(x) = \frac{1}{a_nw(x)},
  \quad n \geq 0 .
\label{ops_Ydet}
\end{equation}
The three term recurrence is then recast as the matrix equation
\begin{equation}
   Y_{n+1}(x) = K_n Y_n(x),
  \quad n \geq 0 ,
\label{ops_Yrecur}
\end{equation}
with the {\it recurrence matrix} given by
\begin{equation}
   K_n(x) = \frac{1}{a_{n+1}}
              \begin{pmatrix} x-b_n & -a_{n} \\
                              a_{n+1} & 0
              \end{pmatrix}, \qquad \det K_n = \frac{a_n}{a_{n+1}},
  \quad n \geq 0 .
\label{ops_Kmatrix}
\end{equation}

A well known consequence of (\ref{ops_threeT}) are the {\it Christoffel-Darboux summation}
formulae
\begin{align}
   \sum^{n-1}_{j=0}p_{j}(x)p_{j}(y) 
  & = a_n\frac{[p_{n}(x)p_{n-1}(y)-p_{n-1}(x)p_{n}(y)]}{x-y},
  \quad n \geq 0 ,
  \\  
   \sum^{n-1}_{j=0}q_{j}(x)p_{j}(y) 
  & = a_n\frac{[q_{n}(x)p_{n-1}(y)-q_{n-1}(x)p_{n}(y)]}{x-y}+\frac{1}{x-y},
  \quad n \geq 0 ,
  \\  
   \sum^{n-1}_{j=0}q_{j}(x)q_{j}(y) 
  & = a_n\frac{[q_{n}(x)q_{n-1}(y)-q_{n-1}(x)q_{n}(y)]}{x-y}-\frac{f(x)-f(y)}{x-y},
  \quad n \geq 0 .
\label{ops_C-D:c}
\end{align}

If one is only interested in the leading orders of the large $ x $ expansion rather
than a systematic expansion then it is convenient to employ an expansion in monomials
rather than in basis functions.
Extending (\ref{ops_poly}) we have expansions about the fixed singularity at $ x=\infty $
\begin{equation}
   p_n(x) = \gamma_n \Bigg[ x^{n} - \left( \sum^{n-1}_{i=0}b_i \right)x^{n-1}
   + \left( \sum_{0\leq i<j<n}b_ib_j-\sum^{n-1}_{i=1}a^2_i \right)x^{n-2} 
   + {\rm O}(x^{n-3}) \Bigg] ,
   \label{ops_pExp}
\end{equation}
valid for $ n \geq 1 $, while for the associated functions
\begin{equation}
   q_n(x) = \gamma^{-1}_n \Bigg[ x^{-n-1} + \left( \sum^{n}_{i=0}b_i \right)x^{-n-2}
   + \left( \sum_{0\leq i\leq j\leq n}b_ib_j+\sum^{n+1}_{i=1}a^2_i \right)x^{-n-3} 
   + {\rm O}(x^{-n-4}) \Bigg] ,
   \label{ops_eExp}
\end{equation}
valid for $ n \geq 0 $.

\subsection{$q$-Quadratic lattice}
In Sections \ref{M=2AW} and \ref{M=3AW} we intend to apply our theory to the $q$-quadratic lattice and will draw upon
numerous properties of the corresponding basis, which we discuss here.
Firstly we recall the analytic continuation of $ \phi_n(x;a) $ 
which is
\begin{equation}
  \phi_r(x;a) = \frac{(az,az^{-1};q)_{\infty}}{(aq^rz,aq^rz^{-1};q)_{\infty}} ,
\end{equation} 
for all $ r\in\C $ but subject to $ |q|<1 $.
We will employ the shorthand notation $ (az^{\pm 1};q)_{\infty}=(az,az^{-1};q)_{\infty} $.
Implicit in the above formula is the elliptic-like function
$ \phi_{\infty}(x;a) $ which has meaning for all $ a,x \in \C $ for  $ |q|<1 $.
The action of the divided-difference operators are
\begin{align}
   \ddoAW_{x} \phi_r(x;a)
 & = 2a\frac{1-q^r}{q-1}\phi_{r-1}(x;q^{1/2}a) ,
\label{DDO_basis:a}
 \\
   \moAW_{x} \phi_r(x;a)
 & = \tfrac{1}{2}(1+q^{-r})\phi_r(x;q^{1/2}a)+\tfrac{1}{2}(1-q^{-r})(1-a^2q^{2r-1})\phi_{r-1}(x;q^{1/2}a) , 
\label{DDO_basis:b}
 \\
   \ddoAW_{x} \phi_{\infty}(x;a)
 & = \frac{2a}{q-1}\phi_{\infty}(x;q^{1/2}a) ,
\label{DDO_basis:c}
 \\
   \moAW_{x} \phi_{\infty}(x;a)
 & = \phi_{\infty}(x;q^{-1/2}a) .
\label{DDO_basis:d}
\end{align}
The expansion theorem of Ismail \cite{{Is}_1995,{Is}_2001} states that
\begin{proposition}
Let $ p $ be a polynomial of degree $ n $, then
\begin{equation}
  p(x) = \sum^n_{k=0} p_k \phi_k(x;a)
\end{equation}
for any $ a\in\C $ where
\begin{equation}
  p_k = \frac{(q-1)^k}{(2a)^k(q;q)_k}q^{-k(k-1)/4}(\ddoAW^{k}_{x} p)(x_k)
\end{equation}
with $ x_k = \tfrac{1}{2}(q^{k/2}a+q^{-k/2}a^{-1}) $.
\end{proposition}
In these works, see (2.2) and the proof of Theorem {1.1} in \cite{{Is}_1995}, we also 
have the change of base formula
\begin{equation}
  \phi_n(x;b) 
  =  \sum^n_{k=0} \left[ {n}\atop{k} \right]_q (abq^k,b/a;q)_{n-k}\left(\frac{b}{a}\right)^k \phi_k(x;a) ,
\label{Bchange}
\end{equation}
where we use the standard definition of the $q$-binomial coefficient.
This result allows us to derive the following linearisation formula.
\begin{lemma}
The product of two basis polynomials with the same base $ a $ has the following 
expansion in terms of the same basis
\begin{equation}
  \phi_k(x;a)\phi_l(x;a) 
  = (-)^{k+l}q^{-kl} \sum^{k+l}_{m=\max(k,l)}(-)^m(a^2q^m;q)_{k+l-m}
                     \frac{(q;q)_l(q;q)_k}{(q;q)_{k+l-m}(q;q)_{m-k}(q;q)_{m-l}}
                     \phi_m(x;a) ,
\label{Blinearisation}
\end{equation}
for all $ k,l \in\N $.
\end{lemma}
Lastly we have the Cauchy expansion formula of Ismail and Stanton \cite{IS_2003}.
\begin{proposition}
The Cauchy kernel has the expansion
\begin{equation}
   \frac{1}{y-x} = 
   \frac{1}{y-x}\frac{\phi_{\infty}(x;a)}{\phi_{\infty}(y;a)}
  -2a\sum^{\infty}_{n=0}\frac{\phi_{n}(x;a)}{\phi_{n+1}(y;a)}q^n ,
\end{equation}
for all $ y $ such that $ y\neq x $ and $ \phi_{\infty}(y;a) \neq 0 $. The expansion
also holds for $ y=y_0 $ with $ \phi_{\infty}(y_0;a)=0 $ but $ y_0 \neq x $ in the 
sense that the left-hand side $ (y_0-x)^{-1} $ equals the limit of the right-hand 
side as $ y \to y_0 $. 
\end{proposition}
This implies the following expansion of the Stieltjes function.
\begin{corollary}
The Stieltjes function has the following expansion as $ x \to \infty $ with
$ x \neq x_k(a) $, with $ k\in 2\mathbb{Z} $ (i.e. $ \phi_{\infty}(x;a)\neq 0 $)
\begin{equation}
  f(x) = \frac{f_{\infty}(x)}{\phi_{\infty}(x;a)}-2a\sum^{\infty}_{n=0}\frac{q^n}{\phi_{n+1}(x;a)}m_{0,n}(a) ,
\label{xLarge_SF:a}
\end{equation}
where
\begin{equation}
  f_{\infty}(x) = \int_{G} \ddoAW y\;w(y)\frac{\phi_{\infty}(y;a)}{x-y}
\label{xLarge_SF:b}
\end{equation}
\end{corollary}

Conforming with standard notation \cite{GR_2004} we define the basic hypergeometric function 
$ {}_{r+1}\varphi_{r} $ by the series 
\begin{equation}
   {}_{r+1}\varphi_{r}\left[
                   \begin{array}{cccc}
                     a_{1}, & a_{2}, & \ldots, & a_{r+1} \\
                     b_{1}, & b_{2}, & \ldots, & b_{r} 
                   \end{array} ; q,z
                   \right] 
  = \sum^{\infty}_{n=0} \frac{(a_{1},a_{2},\ldots,a_{r+1};q)_n}{(q,b_{1},b_{2},\ldots,b_{r};q)_n} z^n ,
\end{equation}
which is convergent for $ |z| < 1 $. 
The {\it very-well-poised} basic hypergeometric function $ {}_{r+1}W_{r} $ is a specialisation 
of the above
\begin{equation}
   {}_{r+1}W_{r}(a_1;a_4,a_5,\ldots,a_{r+1};q,z)
  = {}_{r+1}\varphi_{r}\left[
                   \begin{array}{cccccc}
                     a_{1}, & q\sqrt{a_{1}}, & -q\sqrt{a_{1}}, & a_4, & \ldots, & a_{r+1} \\
                     \sqrt{a_{1}}, & -\sqrt{a_{1}}, & qa_{1}/a_4, & \ldots, & qa_{1}/a_{r+1} & \\
                   \end{array} ; q,z
                   \right]  
\end{equation}
so that $ qa_{1} = a_{2}b_{1} = a_{3}b_{2} = \cdots = a_{r+1}b_{r} $ and 
$ a_{2} = q\sqrt{a_{1}} $, $ a_{3} = -q\sqrt{a_{1}} $.
The $ {}_{r+1}\varphi_{r} $  or $ {}_{r+1}W_{r} $ functions may also be {\it balanced}, whereby $ z=q $ and
$ \prod^{r}_{j=1}b_{j} = q\prod^{r+1}_{j=1}a_j $.

\section{Spectral Differences}\label{SpectralS}
\setcounter{equation}{0}

In this section we lay out the structures of the spectral divided-difference operator in the
context of the orthogonal polynomial system for a general lattice type. Our starting point is the notion of the 
{\it $\ddoAW$-semi-classical weight}, as given by the following definition of Magnus \cite{Ma_1988}.
\begin{definition}[\cite{Ma_1988}]\label{spectral_DD}
Let the {\it $\ddoAW$-semi-classical weight} satisfy
\begin{equation}
    W\ddoAW_{x} w = 2V\moAW_{x} w ,
\label{spectral_DD_wgt:a}
\end{equation}
or equivalently
\begin{equation}
    \frac{w(y_+)}{w(y_-)} = \frac{W+\Delta yV}{W-\Delta yV}(x) ,
\label{spectral_DD_wgt:b}
\end{equation}
for $ W(x),V(x) $ irreducible polynomials, which we will call {\it spectral data polynomials}.
Furthermore we assume $ W\pm \Delta yV \neq 0 $ for all $ x \in G $.
For minimal degrees of $ W, V $ this is the analog of the Pearson equation.
\end{definition}

\begin{remark}
On the finite lattice $ x \in \{x_{0},\dots,x_{\Gt}\} $ we naturally require $ w(x) \neq 0 $, however
we will impose {\it upper} and {\it lower terminating conditions}
\begin{equation}
   w(E_{x}^{+2}x_{\Gt}) = w(E_{x}^{-2}x_{0}) = 0 ,
\label{terminate:a}
\end{equation}
respectively. These are the conditions analogous to Eqs. (2.3.2) and (3.3.3) in 
Nikiforov, Suslov and Uvarov \cite{ops_NSU}.
To be consistent with (\ref{spectral_DD_wgt:b}) this implies
\begin{equation}
  (W+\Delta yV)(E_{x}^{+}x_{\Gt}) = (W-\Delta yV)(E_{x}^{-}x_{0}) = 0 .
\label{terminate:b}
\end{equation}
\end{remark}

\begin{remark}
In a series of works Suslov and collaborators \cite{AS_1988,Su_1989,AS_1990,RS_1994a,RS_1994b}
have studied the Pearson equation for
all of the lattices admissible in the classification, and sought solutions for the weight functions
given suitable polynomials for $ W, V $. However they limited their choices to examples
of minimal degree for $ W,V $ which made contact with those weights found in the Askey
table.  
In terms of our own variables those found in \cite{Su_1989} are given by
\begin{equation}
  \sigma = E^{-}_x(W-\Delta yV), \qquad
  \tau = \frac{E^{+}_x(W+\Delta yV)(x)-E^{-}_x(W-\Delta yV)(x)}{\Delta y} .
\end{equation}
\end{remark}

From (\ref{Sf_expand}) we recognise $ f(x) $ as a moment generating function and a 
key element in our theory are the systems of linear divided-difference equations
satisfied by the moments. Before we state such systems we need to note the following result.

\begin{lemma}
Let $ l_k(x;a) $ be a canonical basis polynomial.
For any $ k\in\mathbb{Z}_{\geq 0} $ and $ a\in\mathbb{C} $ the weight satisfies the integral 
equation
\begin{equation}
   \int_{x_{0}\leq x\leq x_{\Gt}}\ddoAW{x}\,\frac{w(x)}{\Delta y(x)}
   \left\{ \left( l_k[W+\Delta yV] \right)(E^{+}_{x}x)-\left( l_k[W-\Delta yV] \right)(E^{-}_{x}x) \right\} = 0 . 
\label{SCwgtEqn}
\end{equation}
\end{lemma}
\begin{proof}
We will prove this statement accounting for the boundary terms explicitly. We 
establish two preliminary identities first, namely
\begin{gather}
   \int_{x_{0}\leq x\leq x_{\Gt}}\ddoAW{x}\,w(x)
   \left\{ \frac{1}{E^{+}_x\Delta y}E^{+}_x\left( l_k[W+\Delta yV] \right)
           -\frac{E^{-2}_x\Delta y}{\Delta y(E^{-}_{x}\Delta y)}E^{-}_x\left( l_k[W-\Delta yV] \right) \right\} = 0 . 
\label{IE:a}
\\
   \int_{x_{0}\leq x\leq x_{\Gt}}\ddoAW{x}\,w(x)
   \left\{ \frac{E^{+2}_x\Delta y}{\Delta y(E^{+}_{x}\Delta y)}E^{+}_x\left( l_k[W+\Delta yV] \right)
           -\frac{1}{E^{-}_x\Delta y}E^{-}_x\left( l_k[W-\Delta yV] \right) \right\} = 0 , 
\label{IE:b}
\end{gather}
with $ l_k := l_k(x;a) $. Taking the first of these we compute
\begin{align*}
  0 & = \int_{x_{0}\leq x\leq x_{\Gt}}\ddoAW{x}\,l_k(E^{+}_{x}x)( W\ddoAW_xw-2VM_xw )(E^{+}_{x}x) ,
   \\
    & = \sum_{0\leq s\leq \Gt}\Delta y(x_{s})l_k(E^{+}_{x}x_{s})W(E^{+}_{x}x_s)\frac{w(E^{+2}_{x}x_s)-w(x_{s})}{\Delta y(E^{+}_{x}x_{s})}
   \\
    &  \hspace{2cm}
       -\sum_{0\leq s\leq \Gt}\Delta y(x_{s})l_k(E^{+}_{x}x_{s})V(E^{+}_{x}x_{s})\left[ w(E^{+2}_{x}x_{s})+w(x_{s}) \right] ,
   \\
     & =   \sum_{0\leq s\leq \Gt}\frac{\Delta y(E^{-2}_{x}x_{s})}{\Delta y(E^{-}_{x}x_{s})}l_k(E^{-}_{x}x_{s})W(E^{-}_{x}x_{s})w(x_{s})
          -\sum_{0\leq s\leq \Gt}\frac{\Delta y(x_{s})}{\Delta y(E^{+}_{x}x_{s})}l_k(E^{+}_{x}x_{s})W(E^{+}_{x}x_{s})w(x_{s})
   \\
     &  \hspace{1cm}
          +\frac{\Delta y(x_{\Gt})}{\Delta y(E^{+}_{x}x_{\Gt})}l_k(E^{+}_{x}x_{\Gt})W(E^{+}_{x}x_{\Gt})w(E^{+2}_{x}x_{\Gt})
          -\frac{\Delta y(E^{-2}_{x}x_{0})}{\Delta y(E^{-}_{x}x_{0})}l_k(E^{-}_{x}x_{0})W(E^{-}_{x}x_{0})w(x_{0})
   \\
     &  \phantom{=}
          -\sum_{0\leq s\leq \Gt}\Delta y(E^{-2}_{x}x_{s})l_k(E^{-}_{x}x_{s})V(E^{-}_{x}x_{s})w(x_{s})
          -\sum_{0\leq s\leq \Gt}\Delta y(x_{s})l_k(E^{+}_{x}x_{s})V(E^{+}_{x}x_{s})w(x_{s})
   \\
     &  \hspace{1cm}
          -\Delta y(x_{\Gt})l_k(E^{+}_{x}x_{\Gt})V(E^{+}_{x}x_{\Gt})w(E^{+2}_{x}x_{\Gt})
          +\Delta y(E^{-2}_{x}x_{0})l_k(E^{-}_{x}x_{0})V(E^{-}_{x}x_{0})w(x_{0}) ,
   \\
     & = \sum_{0\leq s\leq \Gt}\Delta y(x_{s})w(x_s)\left\{ -\left( l_k\left[\frac{W}{\Delta y}+V\right] \right)(E^{+}_{x}x_s)
                                      +\frac{\Delta y(E^{-2}_{x}x_{s})}{\Delta y(x_{s})}\left( l_k\left[\frac{W}{\Delta y}-V\right] \right)(E^{-}_{x}x_s) \right\}
   \\
     &  \hspace{5mm}
        +\frac{\Delta y(x_{\Gt})}{\Delta y(E^{+}_{x}x_{\Gt})}l_k(E^{+}_{x}x_{\Gt})w(E^{+2}_{x}x_{\Gt})[W-\Delta yV](E^{+}_{x}x_{\Gt})
        -\frac{\Delta y(E^{-2}_{x}x_{0})}{\Delta y(E^{-}_{x}x_{0})}l_k(E^{-}_{x}x_{0})w(x_{0})[W-\Delta yV](E^{-}_{x}x_{0}) , 
\end{align*}
where the boundary terms arise from using the integration by parts and the change of variables 
$ x_s \mapsto E^{\pm 2}_x x_s $ formulae. Both boundary terms are in fact zero due to the upper
and lower terminating conditions. The second formula (\ref{IE:b}) is just a variation of the first
with the integrand evaluated at $ E^{-}_{x}x_{s} $ instead. Now for a general quadratic lattice
\begin{equation*}
   \frac{\Delta y+E^{\pm 2}_{x}\Delta y}{E^{\pm}_{x}\Delta y} = -2\frac{\mathcal B}{\mathcal A} ,
\end{equation*}
so we simply add (\ref{IE:a}) and (\ref{IE:b}) and note that a common, constant factor appears in the
integrand of precisely this form. Then (\ref{SCwgtEqn}) immediately follows.
\end{proof}

\begin{remark}
An equivalent system for general lattices has been derived by Suslov \cite{Su_1989}
in the special case of $ k=0 $ of (\ref{SCwgtEqn}).
\end{remark}

We also need to express $ W\pm\Delta yV $ in terms of canonical basis
polynomials and formulate the following definitions for the coefficients $ \kappa_{N,l}(a), \delta_{N,l}(a) $
\begin{align} 
  E^{+}_x(W+\Delta yV)(x)+E^{-}_x(W-\Delta yV)(x) & =: 2\sum^{N}_{k=0}\kappa_{N,k}(a)l_k(x;a) ,
\label{Sum_Bexp}
  \\
  E^{+}_x(W+\Delta yV)(x)-E^{-}_x(W-\Delta yV)(x) & =: \Delta y\sum^{N-1}_{k=0}\delta_{N,k}(a)l_k(x;a) .
\label{Diff_Bexp}
\end{align} 
Here $ N $ is the cutoff determined by the degrees of $ W, V $ and the lattice type.

\begin{corollary}
Let us assume that $ b' \neq a $, where $ b, b' $ are related as $ a, a' $ are by (\ref{SNUL:b}). 
For all $ k\in\mathbb{Z}_{\geq 0} $ the moments $ m_{k,l}(b',a) $ 
are characterised by the linear, homogeneous recurrence relations
\begin{multline}
   \left[ d_{k}(b)c_{k+1}(b)+\frac{\mathcal{B}}{\mathcal{A}}d_{k-1}(b')c_{k}(b) \right]
   \sum^{N-1}_{l=0}\delta_{N,l}(a)m_{k,l}(b',a)
  \\
  +c_{k}(b)\sum^{N}_{l=0} 
   \left[ \kappa_{N,l}(a)+\left( e_{k}(b)+\frac{\mathcal{B}}{\mathcal{A}}e_{k-1}(b')+\frac{\mathcal{D}}{\mathcal{A}} \right)\delta_{N,l}(a)
   \right] m_{k-1,l}(b',a)
  = 0 .
\label{Mrecur}
\end{multline}
For $ k=0 $ the last term is absent and therefore $ b,b' $ do not appear.
\end{corollary}
\begin{proof}
We start with the integral formula (\ref{SCwgtEqn}) and employ the resolutions of 
$  E^{+}_x(W+\Delta yV) $ and $ E^{-}_x(W-\Delta yV) $ as given by (\ref{Sum_Bexp})
and (\ref{Diff_Bexp}).
In addition we require expressions for $ \ddoAW_x l_k(x;b) $ and $ \moAW_x l_k(x;b) $, and use 
(\ref{SNUL:b}) for the former and
\begin{equation}
  \moAW_x l_k(x;b) = \left( d_{k}(b)c_{k+1}(b)+\frac{\mathcal{B}}{\mathcal{A}}d_{k-1}(b')c_{k}(b) \right)l_{k}(x;b')
                    +c_{k}(b)\left( e_{k}(b)+\frac{\mathcal{B}}{\mathcal{A}}e_{k-1}(b')+\frac{\mathcal{D}}{\mathcal{A}} \right)l_{k-1}(x;b') ,     
\end{equation}
for the latter.
\end{proof}

If a specialisation of the parameters $ b', a $ is made to effect 
some cancellation with corresponding factors in the weight then the moments are given
by an integral whose integrand has the same structure and (\ref{Mrecur})
is now a linear divided-difference equation with respect to the internal parameters
of the weight. 

Now we continue to develop the consequences of the $\ddoAW$-semi-classical weight
on the orthogonal polynomial system,
and in particular for the Stieltjes function. The following result will be crucial for our
approach. 
\begin{proposition}
Given a {\it $\ddoAW$-semi-classical class} of an orthogonal polynomial system fulfilling 
(\ref{spectral_DD_wgt:a}) and the conditions therein then the Stieltjes function satisfies
\begin{equation}
    W\ddoAW_{x} f = 2V\moAW_{x} f + U ,
\label{spectral_DD_st}
\end{equation}
where $ U(x) $ in (\ref{spectral_DD_st}) is a polynomial in $ x $ (as are $ V(x) $ and $ W(x) $).
\end{proposition}
\begin{proof}
Starting with the definition (\ref{ops_stieltjes}) we compute
\begin{equation*}
   \ddoAW_{x}f = \int_{u_{0}\leq u\leq u_{\Gt}}\ddoAW{u}\,w(u)\ddoAW_{x}\frac{1}{x-u} ,
\end{equation*}
and make the evaluation of the factor
\begin{equation*}
  \ddoAW_{x}\frac{1}{x-u} = -\frac{1}{(E^{+}_{x}x-u)(E^{-}_{x}x-u)} .
\end{equation*}
However we note that
\begin{equation*}
  \ddoAW_{u}\frac{1}{x-u} = \frac{1}{(x-E^{+}_{u}u)(x-E^{-}_{u}u)} ,
\end{equation*}
so that
\begin{equation*}
  \ddoAW_{x}\frac{1}{x-u} = -\frac{(x-E^{+}_{u}u)(x-E^{-}_{u}u)}{(E^{+}_{x}x-u)(E^{-}_{x}x-u)}\ddoAW_{u}\frac{1}{x-u} = -\ddoAW_{u}\frac{1}{x-u} ,
\end{equation*}
where we have assumed a symmetric quadratic lattice for simplicity, but point out that 
our argument could be extended to the general case.
We deduce then that
\begin{align*}
  \ddoAW_{x}f & = -\int_{u_{0}\leq u\leq u_{\Gt}}\ddoAW{u}\,w(u)\ddoAW_{u}\frac{1}{x-u} ,
  \\
  & = \int_{u_{0}\leq u\leq u_{\Gt}}\ddoAW{(E^{+}_{u}u)}\,\frac{1}{x-E^{+}_{u}u}\ddoAW_{u}w(E^{+}_{u}u)
       -\frac{w(E^{+2}_{u}u_{\Gt})}{x-E^{+}_{u}u_{\Gt}}+\frac{w(u_{0})}{x-E^{-}_{u}u_{0}} ,
\end{align*}
after employing the formula for integration by parts. We next use the definition of the $\ddoAW$-semi-classical
weight (\ref{spectral_DD_wgt:a}) and perform a sequence of subtractions in the numerator of the 
integrand
\begin{align*}
  W(x)\ddoAW_{x}f(x) 
  & = \int_{u_{0}\leq u\leq u_{\Gt}}\ddoAW{(E^{+}_{u}u)}\,\frac{1}{x-E^{+}_{u}u}\moAW_{u}w(E^{+}_{u}u)\frac{2V(E^{+}_{u}u)}{W(E^{+}_{u}u)}W(x)
      -\frac{w(E^{+2}_{u}u_{\Gt})}{x-E^{+}_{u}u_{\Gt}}W(x)+\frac{w(u_{0})}{x-E^{-}_{u}u_{0}}W(x)  ,
  \\
  & =  2V(x)\moAW_{x}f(x)
      +2V(x)\left\{ \int_{u_{0}\leq u\leq u_{\Gt}}\ddoAW{(E^{+}_{u}u)}\,\frac{\moAW_{u}w(E^{+}_{u}u)}{x-E^{+}_{u}u}-\moAW_{x}f(x) \right\}
  \\
  & \hspace{2cm}
      +2\int_{u_{0}\leq u\leq u_{\Gt}}\ddoAW{(E^{+}u)}\,\frac{\moAW_{u}w(E^{+}_{u}u)}{W(E^{+}_{u}u)}\frac{V(E^{+}_{u}u)W(x)-V(x)W(E^{+}_{u}u)}{x-E^{+}_{u}u}
      +{\fontencoding{T1}\selectfont "} .
\end{align*}
Since the third term on the right-hand side of the previous equation is clearly a polynomial in $ x $
we shall focus next on the second term. The last factor of this latter term can be written as
\begin{multline*}
  \moAW_{x}f(x)-\int_{u_{0}\leq u\leq u_{\Gt}}\ddoAW{(E^{+}_{u}u)}\,\frac{\moAW_{u}w(E^{+}_{u}u)}{x-E^{+}_{u}u}
    = \int_{u_{0}\leq u\leq u_{\Gt}} \ddoAW{u}\, w(u)\tfrac{1}{2}\left\{ \frac{1}{E^{+}_{x}x-u}+\frac{1}{E^{-}_{x}x-u} \right\}
    \\
      -\tfrac{1}{2}\int_{u_{0}\leq u\leq u_{\Gt}}\ddoAW{(E^{+}_{u}u)}\,\frac{w(E^{+2}_{u}u)}{x-E^{+}_{u}u}
      -\tfrac{1}{2}\int_{u_{0}\leq u\leq u_{\Gt}}\ddoAW{(E^{+}_{u}u)}\,\frac{w(u)}{x-E^{+}_{u}u} ,
\end{multline*}
The last two terms of the above equation are 
\begin{multline*}
  \int_{u_{0}\leq u\leq u_{\Gt}}\ddoAW{(E^{+}_{u}u)}\,\frac{w(E^{+2}_{u}u)}{x-E^{+}_{u}u}
      +\int_{u_{0}\leq u\leq u_{\Gt}}\ddoAW{(E^{+}_{u}u)}\,\frac{w(u)}{x-E^{+}_{u}u}
  \\
  =  \sum_{0\leq s\leq \Gt}[u_{s}-E^{-2}_{u}u_{s}]\frac{w(u_{s})}{x-E^{-}_{u}u_{s}}
    +\sum_{0\leq s\leq \Gt}[E^{+2}_{u}u_{s}-u_{s}]\frac{w(u_{s})}{x-E^{+}_{u}u_{s}}
    +\Delta y(E^{+}_{u}u_{\Gt})\frac{w(E^{+2}_{u}u_{\Gt})}{x-E^{+}_{u}u_{\Gt}}-\Delta y(E^{-}_{u}u_{0})\frac{w(u_{0})}{x-E^{-}_{u}u_{0}} ,
 \\
  =  \sum_{0\leq s\leq \Gt}w(u_{s})\left\{ \frac{E^{+2}_{u}u_{s}-u_{s}}{x-E^{+}_{u}u_{s}}+\frac{u_{s}-E^{-2}_{u}u_{s}}{x-E^{-}_{u}u_{s}} \right\} 
      +{\fontencoding{T1}\selectfont "} ,
\end{multline*}
where we have used the identity for integrals under the re-parameterisation of the lattice
$ u \mapsto E^{-2}_{u}u $ and introducing the additional boundary terms. 
Combining these results we have
\begin{multline*}
  \moAW_{x}f(x)-\int_{u_{0}\leq u\leq u_{\Gt}}\ddoAW{(E^{+}_{u}u)}\,\frac{\moAW_{u}w(E^{+}_{u}u)}{x-E^{+}_{u}u}
  \\
  = \tfrac{1}{2}\sum_{0\leq s\leq \Gt}w(u_{s}) \left\{
               (E^{+}_{u}u_{s}-E^{-}_{u}u_{s})\left(\frac{1}{E^{+}_{x}x-u_{s}}+\frac{1}{E^{-}_{x}x-u_{s}}\right)
                   -\frac{E^{+2}_{u}u_{s}-u_{s}}{x-E^{+}_{u}u_{s}}-\frac{u_{s}-E^{-2}_{u}u_{s}}{x-E^{-}_{u}u_{s}} \right\}
  \\
    -\tfrac{1}{2}\Delta y(E^{+}_{u}u_{\Gt})\frac{w(E^{+2}_{u}u_{\Gt})}{x-E^{+}_{u}u_{\Gt}}+\tfrac{1}{2}\Delta y(E^{-}_{u}u_{0})\frac{w(u_{0})}{x-E^{-}_{u}u_{0}} ,
  \\
  = \tfrac{1}{2}\sum_{0\leq s\leq \Gt}w(u_{s}) \left\{
               (E^{+}_{u}u_{s}-E^{-}_{u}u_{s})\frac{(E^{+}_{x}+E^{-}_{x})x-2u_{s}}{(x-E^{+}_{u}u_{s})(x-E^{-}_{u}u_{s})}
                   -\frac{E^{+2}_{u}u_{s}-u_{s}}{x-E^{+}_{u}u_{s}}-\frac{u_{s}-E^{-2}_{u}u_{s}}{x-E^{-}_{u}u_{s}} \right\}+{\fontencoding{T1}\selectfont "} ,
\end{multline*}
where we have used the ratio identity again. We observe that the summand appears to have 
two sets of simple poles at $ x=E^{\pm}u_{s} $, however when we compute their residues they 
vanish identically, implying that the summand is in reality a polynomial in $ x $. 
This leaves us with the two sets of boundary terms to account for. We find that their combination is
\begin{equation*}
   -\frac{w(E^{+2}_{u}u_{\Gt})}{x-E^{+}_{u}u_{\Gt}}[W(x)-\Delta y(E^{+}_{u}u_{\Gt})V(x)]
   +\frac{w(u_{0})}{x-E^{-}_{u}u_{0}}[W(x)-\Delta y(E^{-}_{u}u_{0})V(x)] .
\end{equation*}
The first term vanishes under the upper terminating condition whilst the second is a polynomial
in $ x $ due to the lower terminating condition. In summary this leads us to conclude that
$ W(x)\ddoAW_{x}f(x)-2V(x)\moAW_{x}f(x) $ is a polynomial in $ x $.
\end{proof}

We offer the following analog to the concept of a regular semi-classical weight.
\begin{definition}\label{regDSCwgt}
A {\it generic} or {\it regular $\ddoAW$-semi-classical weight} has two properties -
\begin{itemize}
\item[(i)]
  strict inequalities in the degrees of the spectral data polynomials, i.e.
  $ {\rm deg}_x W = M $, $ {\rm deg}_x V = M-1 $ and $ {\rm deg}_x U = M-2 $, and
\item[(ii)]
  the lattice generated by any zero of $ (W^2-\Delta y^2V^2)(x) $ does not coincide
  with the lattice generated by another zero, i.e. for any $ \tilde{x} $ such that
  $ (W^2-\Delta y^2V^2)(\tilde{x})=0 $ then 
  $ (W^2-\Delta y^2V^2)(E^{2\mathbb{Z}}_x\tilde{x}) \neq 0 $.
\end{itemize}
Conversely any weight with data not satisfying these conditions will be
termed an {\it irregular $\ddoAW$-semi-classical weight}.
Note that the zeros of $ W $ are not relevant and no interpretation as a
singularity can be placed upon a zero of $ (W^2-\Delta y^2V^2)(x) $.
\end{definition}

\begin{proposition}[\cite{Ma_1995}]\label{spectral_Cff}
The spectral coefficients $ W_n(x) $, $ \Theta_n(x) $ and $ \Omega_n(x) $ are defined
in terms of bilinear formulae of the polynomials and associated functions
\begin{multline}
  \frac{2}{a_{n}}\big( 2W_{n}(x)-W(x) \big)
  = (W+\Delta yV)\left[ p_{n}(y_{+})q_{n-1}(y_{-})-p_{n-1}(y_{+})q_{n}(y_{-}) \right]
  \\
   +(W-\Delta yV)\left[ p_{n}(y_{-})q_{n-1}(y_{+})-p_{n-1}(y_{-})q_{n}(y_{+}) \right],
  \quad n \geq 0 ,
\label{spectral_coeff:a}
\end{multline}
\begin{equation}
  \Delta y\,\Theta_{n}(x) = (W+\Delta yV)p_{n}(y_{+})q_{n}(y_{-})-(W-\Delta yV)p_{n}(y_{-})q_{n}(y_{+}),
  \quad n \geq 0 ,
\label{spectral_coeff:b}
\end{equation}
\begin{multline}
  \frac{2\Delta y}{a_{n}}\big( \Omega_{n}(x)+V(x) \big)
  = (W+\Delta yV)\left[ p_{n}(y_{+})q_{n-1}(y_{-})+p_{n-1}(y_{+})q_{n}(y_{-}) \right]
  \\
   +(W-\Delta yV)\left[-p_{n}(y_{-})q_{n-1}(y_{+})-p_{n-1}(y_{-})q_{n}(y_{+}) \right],
  \quad n \geq 0 .
\label{spectral_coeff:c}
\end{multline}
For the $\ddoAW$-semi-classical class of weights these coefficients are polynomials in $ x $.
\end{proposition}
\begin{proof}
Our proof will be an extension of Laguerre's reasoning for the differential case.
However we first need to appreciate one important fact. Any polynomial in $ y_- $ and
$ y_+ $ is an element of the Ring $ \C[x]+\Delta y\C[x] $ because 
$ y_{\pm}=\moAW_{x}x\pm \tfrac{1}{2}\Delta y $ and $ \moAW_{x}x $ is a linear polynomial in
$ x $ and $ \Delta y^2 $ is a quadratic polynomial in $ x $.
The starting point will be the definition for $ q_n $, but inverted to solve for
the Stieltjes function
\begin{equation}
   f = \frac{q_n+p^{(1)}_{n-1}}{p_n} ,
\end{equation}
for arbitrary $ n \geq 0 $. Now consider the formulation of (\ref{spectral_DD_st})
using this expression
\begin{align*}
   0 & = (W-\Delta yV)f(y_+)-(W+\Delta yV)f(y_-)-\Delta yU ,
   \\
     & = \frac{(W-\Delta yV)q_n(y_+)p_n(y_-)-(W+\Delta yV)q_n(y_-)p_n(y_+)}{p_n(y_+)p_n(y_-)}
   \\
     & \phantom{=}\qquad
       + \frac{(W-\Delta yV)p^{(1)}_{n-1}(y_+)p_n(y_-)-(W+\Delta yV)p^{(1)}_{n-1}(y_-)p_n(y_+)-\Delta yUp_n(y_+)p_n(y_-)}
              {p_n(y_+)p_n(y_-)} .
\end{align*}
Now $ (W-\Delta yV)p^{(1)}_{n-1}(y_+)p_n(y_-)-(W+\Delta yV)p^{(1)}_{n-1}(y_-)p_n(y_+)-\Delta yUp_n(y_+)p_n(y_-) \in \C[x]+\Delta y\C[x] $ 
and furthermore because it is odd under $ y_+ \leftrightarrow y_- $ this quantity is 
actually an element of $ \Delta y\C[x] $ only. 
Consequently $ (W-\Delta yV)q_n(y_+)p_n(y_-)-(W+\Delta yV)q_n(y_-)p_n(y_+) \in \Delta y\C[x] $
and we define the polynomial $ \Theta_n(x) $ by
\begin{equation} 
  \Delta y\Theta_n(x) = (W+\Delta yV)q_n(y_-)p_n(y_+)-(W-\Delta yV)q_n(y_+)p_n(y_-) ,
\label{auxA3}
\end{equation}
which is (\ref{spectral_coeff:b}). To find the other relations we employ 
the alternative form of $ \Theta_n(x) $
\begin{equation}
   \Delta y\Theta_n = 
  (W-\Delta yV)p^{(1)}_{n-1}(y_+)p_n(y_-)-(W+\Delta yV)p^{(1)}_{n-1}(y_-)p_n(y_+)-\Delta yUp_n(y_+)p_n(y_-) ,
\end{equation}
and multiply the left-hand sides by the Casoratians
$ 1 = a_n[p_{n-1}(y_{\pm})p^{(1)}_{n-1}(y_{\pm})-p_n(y_{\pm})p^{(1)}_{n-2}(y_{\pm})] $
so that upon re-organising we have
\begin{multline} 
  p_n(y_+)\left[ (W+\Delta yV)p^{(1)}_{n-1}(y_-)+\Delta yUp_n(y_-)-\Delta y a_n\Theta_np^{(1)}_{n-2}(y_+)
          \right]
  \\
  = p^{(1)}_{n-1}(y_+)\left[ (W-\Delta yV)p_n(y_-)-\Delta y a_n\Theta_np_{n-1}(y_+) \right] ,
\label{aux1}
\end{multline} 
and
\begin{multline} 
  p_n(y_-)\left[ (W-\Delta yV)p^{(1)}_{n-1}(y_+)-\Delta yUp_n(y_+)+\Delta y a_n\Theta_np^{(1)}_{n-2}(y_-)
          \right]
  \\
  = p^{(1)}_{n-1}(y_-)\left[ (W+\Delta yV)p_n(y_+)+\Delta y a_n\Theta_np_{n-1}(y_-) \right] .
\label{aux2}
\end{multline}
Relation (\ref{aux1}) implies there must exist polynomials $ \pi_1,\pi_2\in \C[x] $ such that
\begin{align}    
   p_n(y_+)p^{(1)}_{n-1}(y_+)(\pi_1+\Delta y\pi_2)
  & = p_n(y_+)\left[ (W+\Delta yV)p^{(1)}_{n-1}(y_-)+\Delta yUp_n(y_-)-\Delta y a_n\Theta_np^{(1)}_{n-2}(y_+)
          \right] ,
  \\
  & = p^{(1)}_{n-1}(y_+)\left[ (W-\Delta yV)p_n(y_-)-\Delta y a_n\Theta_np_{n-1}(y_+) \right] ,
\end{align}
while (\ref{aux2}) implies an identical relation with $ y_+ \leftrightarrow y_- $ and 
$ \Delta y \mapsto -\Delta y $. Therefore we have the pair of relations 
\begin{align}
   p^{(1)}_{n-1}(y_+)(\pi_1+\Delta y\pi_2)
  & = (W+\Delta yV)p^{(1)}_{n-1}(y_-)+\Delta yUp_n(y_-)-\Delta y a_n\Theta_np^{(1)}_{n-2}(y_+) ,
  \label{aux3} \\
   p_n(y_+)(\pi_1+\Delta y\pi_2)
  & = (W-\Delta yV)p_n(y_-)-\Delta y a_n\Theta_np_{n-1}(y_+) ,
  \label{aux4}
\end{align}
and an analogous pair with $ y_+ \leftrightarrow y_- $. We form the combination of
$ p_{n-1}(y_+)\times\text{(\ref{aux3})}-p^{(1)}_{n-2}(y_+)\times\text{(\ref{aux4})} $
and by employing the Casoratian and eliminating the associated polynomials
we arrive at the solution
\begin{equation} 
   \pi_1+\Delta y\pi_2 
  = -a_n(W+\Delta yV)p_{n-1}(y_+)q_{n}(y_-)+a_n(W-\Delta yV)p_{n}(y_-)q_{n-1}(y_+) .
\label{auxA1}
\end{equation}
A similar procedure on the other pair yields the partner relation
\begin{equation} 
   \pi_1-\Delta y\pi_2 
  = a_n(W+\Delta yV)p_{n}(y_+)q_{n-1}(y_-)-a_n(W-\Delta yV)p_{n-1}(y_-)q_{n}(y_+) .
\label{auxA2}
\end{equation} 
The sum and difference of these are simply (\ref{spectral_coeff:a}) and 
(\ref{spectral_coeff:c}) with the identifications 
$ \pi_1 = 2W_n-W $ and $ \pi_2 = -\Omega_n-V $. 
\end{proof}

Knowledge of the large $ x $ expansions of the spectral coefficients will be important 
in the ensuing analysis.
\begin{proposition}[\cite{Ma_1988}]\label{spectral_Cff_Exp}
The spectral coefficients have terminating expansions about $ x=\infty $ with 
leading order terms
\begin{equation}
  W_n(x) = \tfrac{1}{2}W + \tfrac{1}{4}[W+\Delta yV]\left( \frac{y_+}{y_-} \right)^{n}
                     + \tfrac{1}{4}[W-\Delta yV]\left( \frac{y_-}{y_+} \right)^{n}
                     + {\rm O}(x^{M-1}),
  \quad n \geq 0 ,
\label{spectral_coeffEXP:a}
\end{equation}
\begin{equation}
  \Theta_n(x) =   \frac{1}{y_-\Delta y}[W+\Delta yV]\left( \frac{y_+}{y_-} \right)^{n}
             - \frac{1}{y_+\Delta y}[W-\Delta yV]\left( \frac{y_-}{y_+} \right)^{n}
             + {\rm O}(x^{M-3}),
  \quad n \geq 0 ,
\label{spectral_coeffEXP:b}
\end{equation}
\begin{equation}
  \Omega_n(x)+V(x) =   \frac{1}{2\Delta y}[W+\Delta yV]\left( \frac{y_+}{y_-} \right)^{n}
               - \frac{1}{2\Delta y}[W-\Delta yV]\left( \frac{y_-}{y_+} \right)^{n}
               + {\rm O}(x^{M-2}),
  \quad n \geq 0 .
\label{spectral_coeffEXP:c}
\end{equation}
\end{proposition}
\begin{proof}
These expansions follow from the substitution of the expansions for the polynomials
(\ref{ops_pExp}) and associated functions (\ref{ops_eExp}) into the
formulae (\ref{spectral_coeff:a})-(\ref{spectral_coeff:c}). One should note that
$ y_{+}/y_{-} = {\rm O}(1) $ as $ x \to \infty $.
\end{proof}

\begin{corollary}[\cite{Ma_1988}]\label{spectralCff_polynom}
In the generic or regular $\ddoAW$-semi-classical case $ W_{n}, \Theta_{n}, \Omega_{n} $ are
polynomials in $ x $ with fixed degrees independent of $n$,
specifically $ {\rm deg}_x W_n = M $, $ {\rm deg}_x \Omega_n = M-1 $ and $ {\rm deg}_x \Theta_n = M-2 $.
\end{corollary}
\begin{proof}
This follows from the previous proposition after recalling that $ y_{\pm} = {\rm O}(x) $
as $ x \to \infty $.
\end{proof}

The expressions for the spectral coefficients (\ref{spectral_coeff:a}-\ref{spectral_coeff:c})
can be inverted yielding a system of coupled first order linear divided-difference 
equations for the polynomials and associated functions with respect to $ x $.
\begin{proposition}[\cite{Ma_1988,Ma_1995}]\label{spectral_DDO}
The polynomials with a $\ddoAW$-semi-classical weight satisfy
\begin{equation}
  W_{n}\ddoAW_{x} p_{n} = \Omega_{n}\moAW_{x} p_{n}-a_{n}\Theta_{n}\moAW_{x} p_{n-1},
  \quad n \geq 0 ,
\label{DDO:a}
\end{equation}
for certain {\it spectral coefficients}, 
$ W_{n}(x), \Theta_{n}(x), \Omega_{n}(x) $.
The partner equation to (\ref{DDO:a}) is 
\begin{equation}
  W_{n}\ddoAW_{x} p_{n-1} = -(\Omega_{n}+2V)\moAW_{x} p_{n-1}+a_{n}\Theta_{n-1}\moAW_{x} p_{n},
  \quad n \geq 0 .
  \label{DDO:b} 
\end{equation}
The associated functions satisfy divided-difference equations
\begin{align}
  W_{n}\ddoAW_{x} q_{n} & = (\Omega_{n}+2V)\moAW_{x} q_{n}-a_{n}\Theta_{n}\moAW_{x} q_{n-1},
  \quad n \geq 0 ,
  \label{DDO:c} \\
  W_{n}\ddoAW_{x} q_{n-1} & = -\Omega_{n}\moAW_{x} q_{n-1}+a_{n}\Theta_{n-1}\moAW_{x} q_{n},
  \quad n \geq 0  ,
  \label{DDO:d} 
\end{align}
identical in structure to (\ref{DDO:a}) and (\ref{DDO:b}) respectively.
The coupled system (\ref{DDO:a},\ref{DDO:b},\ref{DDO:c},\ref{DDO:d}) can be written in
matrix form as the {\it spectral divided-difference} equation
\begin{equation}
   \ddoAW_{x} Y_n(x) := A_{n}\moAW_{x} Y_n(x)
   = \frac{1}{W_{n}(x)}
     \begin{pmatrix} \Omega_n(x) & -a_n\Theta_n(x) \\
                     a_n\Theta_{n-1}(x) & -\Omega_n(x)-2V(x)
     \end{pmatrix}\moAW_{x} Y_n(x),
  \quad n \geq 0 ,
\label{DDO:e}
\end{equation}
with $ A_{n} $ termed the {\it spectral matrix}.
\end{proposition}
\begin{proof}
Our derivation of the spectral divided-difference equations will employ a simple
method - that of inversion of the relations for the spectral coefficients
(\ref{spectral_coeff:a}-\ref{spectral_coeff:c}).
Consider the following equivalent forms of those relations
\begin{align}
  2W_n-W+\Delta y(\Omega_n+V)
  & = a_n\left[ (W+\Delta yV)p_{n}(y_{+})q_{n-1}(y_{-})-(W-\Delta yV)p_{n-1}(y_{-})q_{n}(y_{+}) \right] ,
  \label{SC:a} \\
  2W_n-W-\Delta y(\Omega_n+V)
  & = a_n\left[-(W+\Delta yV)p_{n-1}(y_{+})q_{n}(y_{-})+(W-\Delta yV)p_{n}(y_{-})q_{n-1}(y_{+}) \right] ,
  \label{SC:b} \\
  \Delta y\Theta_n
  & = (W+\Delta yV)p_{n}(y_{+})q_{n}(y_{-})-(W-\Delta yV)p_{n}(y_{-})q_{n}(y_{+}) .
  \label{SC:c}
\end{align}
For our first task, that of deriving (\ref{DDO:a}), we construct the combination 
$ p_{n}(y_{+})\times\text{(\ref{SC:b})}-p_{n}(y_{-})\times\text{(\ref{SC:a})}+(p_{n-1}(y_{+})+p_{n-1}(y_{-}))a_n\times\text{(\ref{SC:c})} $.
After effecting some initial cancellation we can then employ the Casoratian relation
(\ref{ops_casoratian}) to make further simplification and our final result is precisely
(\ref{DDO:a}) in finite difference form.
The approach taken for the other relations (\ref{DDO:b})-(\ref{DDO:d}) is the same.
For (\ref{DDO:b}) we use the combination
$ p_{n-1}(y_{+})\times\text{(\ref{SC:a})}-p_{n-1}(y_{-})\times\text{(\ref{SC:b})}
-(p_{n}(y_{+})+p_{n}(y_{-}))a_n\times\left.\text{(\ref{SC:c})}\right|_{n\mapsto n-1} $,
while for (\ref{DDO:c}) we use
$ q_{n}(y_{+})\times\text{(\ref{SC:b})}-q_{n}(y_{-})\times\text{(\ref{SC:a})}+(q_{n-1}(y_{+})+q_{n-1}(y_{-}))a_n\times\text{(\ref{SC:c})} $
and for (\ref{DDO:d}) we use
$ q_{n-1}(y_{+})\times\text{(\ref{SC:a})}-q_{n-1}(y_{-})\times\text{(\ref{SC:b})}
-(q_{n}(y_{+})+q_{n}(y_{-}))a_n\times\left.\text{(\ref{SC:c})}\right|_{n\mapsto n-1} $.
\end{proof}

The initial values of the spectral coefficients are specified by the weight data
$ W,V,U $ and the initial norm $ \gamma_0 $
\begin{equation}
   W_{-1} = 0, \quad W_{0} = W , \quad
   \Theta_{-1} = 0, \quad \Theta_{0} = -\gamma^2_0 U , \quad
   \Omega_{0} = 0, \quad \Omega_{1} = -2V-\gamma^2_0(\moAW_{x} x-b_0)U .
\label{spectral_coeff_init}
\end{equation}

\begin{remark}
An alternative line of reasoning to achieve the results above, and in fact to generalise them
beyond the $ \ddoAW$-semi-classical class, is the following.
By positing an explicit quadrature formula for the non-linear lattice,
one upon which the orthogonality condition is founded, one can derive 
(\ref{DDO:a}), (\ref{DDO:b}) and (\ref{DDO:e}) assuming the existence of certain moments of
"log-derivative" of the weight $ \ddoAW_x w/\moAW_x w $.
A byproduct of this approach is that one also constructs an explicit 
quadrature formula for the spectral coefficients, albeit one which involves the polynomials themselves. 
\end{remark}

The spectral structure in Proposition \ref{spectral_DDO} has been derived before for OPS with respect to general
classes of weights on the linear lattice by \cite{FHR_1998,MM_2002,INS_2004,MM_2008}
and on the $q$-linear lattice by \cite{KM_2002,MM_2002,Ismail_2003,Kh_2003,CI_2008,GK_2009,IS_2009,Me_2009,OWF_2010}.
This has also been done for specific weights on the $q$-linear lattice in \cite{Ni_1994,BSvA_2008}
and for OPS on the unit circle subject to $q$-difference relations in \cite{IW_2001,Bi_2009}.
In addition to the studies by Magnus the works \cite{BaFo_2003,Fo_2008} have investigated
the spectral structures for general {\it SNUL} and arbitrary $ {\rm deg}_xW, {\rm deg}_xV $.
Characterisation theorems for the classical OPS of the Askey table have been made by \cite{C-S_2006,A-NM_2001} 
and others using an equivalent formulation to the spectral structures on
general {\it SNUL}, but of course restricted to the $ {\rm deg}_xW=2, {\rm deg}_xV=1 $ case.

Compatibility relation between the matrix recurrence relation and the spectral divided-difference 
equation imposes the following conditions.
The Cayley transform $ \left( 1-\tfrac{1}{2}\Delta y\,A \right)^{-1}\left( 1+\tfrac{1}{2}\Delta y\,A \right) $
will figure prominently is the ensuing analysis.
\begin{proposition}\label{spectral+threeT}
The spectral matrix and the recurrence matrix satisfy
\begin{equation}
   A_{n+1}\cdot \moAW_{x} K_n-\moAW_{x} K_n\cdot A_{n} 
  = \ddoAW_{x} K_n-\tfrac{1}{4}\Delta y^2 A_{n+1}\cdot \ddoAW_{x} K_n\cdot A_{n},
  \quad n \geq 0 ,
\label{spectralA_recur:a}
\end{equation}
or equivalently
\begin{equation}
  K_n(y_+)\left( 1-\tfrac{1}{2}\Delta y\,A_{n} \right)^{-1}\left( 1+\tfrac{1}{2}\Delta y\,A_{n} \right) 
  = 
  \left( 1-\tfrac{1}{2}\Delta y\,A_{n+1} \right)^{-1}\left( 1+\tfrac{1}{2}\Delta y\,A_{n+1} \right)K_n(y_-),
  \quad n \geq 0 .
\label{spectralA_recur:b}
\end{equation}
\end{proposition}
\begin{proof}
The finite difference form of (\ref{DDO:e}) is
\begin{equation}
   \left( 1-\tfrac{1}{2}\Delta yA_n \right)Y_n(y_{+}) =
  \left( 1+\tfrac{1}{2}\Delta yA_n \right)Y_n(y_{-}) ,
\end{equation}
which demonstrates why the Cayley transform arises. The second relation (\ref{spectralA_recur:b})
is found by computing $ Y_{n+1}(y_{+}) $ in two different ways from
$ Y_{n}(y_{-}) $ corresponding to the different orders of the operations
$ n \mapsto n+1 $ and $ y_{-} \mapsto y_{+} $. The first relation (\ref{spectralA_recur:a})
is derived by computing $ \ddoAW_x Y_{n+1} $ in terms of $ \moAW_x Y_{n} $ in two
different orders.
\end{proof}

Compatibility of the three term recurrence relation and the spectral divided-difference 
equation equation implies the following result.
\begin{proposition}[\cite{Ma_1988,Ma_1995}]\label{spectral_Cff_recur}
The spectral coefficients arising in Proposition \ref{spectral_DDO} satisfy
recurrence relations in $ n $,
\begin{gather}
  W_{n+1} = W_{n}+\tfrac{1}{4}\Delta y^2 \Theta_{n},
  \quad n \geq 0 ,
\label{spectral_coeff_recur:a}  \\
  \Omega_{n+1}+\Omega_{n}+2V = (\moAW_{x} x-b_{n})\Theta_{n},
  \quad n \geq 0 ,
\label{spectral_coeff_recur:b}  \\ 
  (W_{n}\Omega_{n+1}-W_{n+1}\Omega_{n})(\moAW_{x} x-b_{n})
  = - \tfrac{1}{4}\Delta y^2 \Omega_{n+1}\Omega_{n}
  + W_{n}W_{n+1} + a^2_{n+1}W_{n}\Theta_{n+1} - a^2_{n}W_{n+1}\Theta_{n-1},
  \quad n \geq 0 .
\label{spectral_coeff_recur:c}
\end{gather}
\end{proposition}
\begin{proof}
We offer a proof in the spirit of our previous proofs.
For the first relation (\ref{spectral_coeff_recur:a}) we construct 
$ 2W_{n+1}-2W_{n}-\tfrac{1}{2}\Delta y^2\Theta_n $ 
and substitute the expressions (\ref{spectral_coeff:a}) and (\ref{spectral_coeff:b}) 
for $ W_{n} $ and $ \Theta_n $ respectively. After collecting terms 
proportional to $ W\pm \Delta yV $ we observe that the coefficients of these
terms separate into two parts, in which the three term recurrence relation
can be employed. The result we find is that they identically vanish. 
The second relation (\ref{spectral_coeff_recur:b}) can be found in an
identical manner starting with
$ \Delta y\left[ \Omega_{n+1}+\Omega_{n}+2V-(\moAW_{x} x-b_{n})\Theta_{n} \right] $.
\end{proof}

The relations of Proposition \ref{spectral_Cff_recur} constitute the analogs of the
Laguerre-Freud equations and have been studied in the special case of linear lattices
by \cite{FHR_1998,MS_1998,INS_2004,MM_2008} and in the $q$-linear case
by \cite{KM_2002,Kh_2003,Me_2009,GK_2009}. For the situation of OPS on the
unit circle with $q$-difference operators some of these relations can be found in 
\cite{IW_2001}. Following on from the Magnus studies the works \cite{BaFo_2003,Fo_2008} have 
investigated these compatibility conditions for general {\it SNUL}.
Together with Proposition \ref{spectral_DDO} this result allows for the construction of 
ladder, or raising and lowering operators and these have been found in the $q$-linear
lattice by \cite{CI_2008,IS_2009}.

A consequence of Proposition \ref{spectral_Cff_recur} is the following. 
\begin{proposition}[\cite{Ma_1988}]\label{spectral_det}
The spectral coefficients satisfy the bilinear recurrence relation
\begin{equation}
 W_{n}(W_{n}-W) = \tfrac{1}{4}\Delta y^2 
  \left[ \Omega_{n}(\Omega_{n}+2V)-a^2_{n}\Theta_{n-1}\Theta_{n} \right]
  = -\tfrac{1}{4}\Delta y^2\det
     \begin{pmatrix} \Omega_n & -a_n\Theta_n \\
                     a_n\Theta_{n-1} & -\Omega_n-2V
     \end{pmatrix},
  \quad n \geq 0 .
\label{spectral_bilinear}
\end{equation}
\end{proposition}
\begin{proof}
Two methods of proof are available here.
The first method starts from the observation that the above bilinear relation can be
expressed as
\begin{equation}
  \left[ 2W_{n}-W+\Delta y(\Omega_{n}+V) \right]
  \left[ 2W_{n}-W-\Delta y(\Omega_{n}+V) \right]
  +a_{n}^2\Delta y^2\Theta_{n}\Theta_{n-1} = W^2-\Delta y^2V^2 .
\end{equation}
Now we have evaluations of the factors appearing in the left-hand side of the 
preceeding equation from the proof of Proposition \ref{spectral_Cff}, 
namely (\ref{auxA1}, \ref{auxA2}, \ref{auxA3}), in terms of products of a polynomial
and associated function. Substituting these evaluations into the left-hand side, 
effecting considerable cancellation of terms and employing the Casoratian relation 
(\ref{ops_casoratian}) the final result simplifies to the right-hand side.  

The second method uses the compatibility relations (\ref{spectral_coeff_recur:a}-
\ref{spectral_coeff_recur:c}) and constructs an integral for this coupled 
system. First one multiplies the left-hand side of 
(\ref{spectral_coeff_recur:b}) by the left-hand side of (\ref{spectral_coeff_recur:c})
and cancels out the common factor of $ \moAW x-b_n $ in the resulting equation.
By judicious use of (\ref{spectral_coeff_recur:a}) one can effect the cancellation
of four terms in this equation and a recasting of the terms linear in $ \Theta_{n} $.
Consequently we arrive at the following equation 
\begin{equation}
  \frac{(\Omega_{n+1}+V)^2}{W_{n+1}}-\frac{(\Omega_{n}+V)^2}{W_{n}}
  = \frac{4}{\Delta y^2}(W_{n+1}-W_{n}) + V^2\left[ \frac{1}{W_{n+1}}-\frac{1}{W_{n}} \right]
    +a_{n+1}^2\frac{\Theta_{n+1}\Theta_{n}}{W_{n+1}}-a_{n}^2\frac{\Theta_{n}\Theta_{n-1}}{W_{n}} ,
\end{equation}
which is manifestly a perfect difference equation in $ n $. By summing this and employing
the initial value evaluations (\ref{spectral_coeff_init}) we have (\ref{spectral_bilinear}).
\end{proof}

The above Proposition must be augmented with the following matrix identities.
\begin{corollary}\label{spectral_M}
The matrix factor appearing in Proposition \ref{spectral+threeT} has the determinant
evaluation
\begin{equation}
  \det(1\pm\tfrac{1}{2}\Delta y\, A_{n}) = \frac{W\mp\Delta y\, V}{W_n},
  \quad n \geq 0 ,
\label{spectral_determinant}
\end{equation}
and its inverse is
\begin{equation}
   (1\pm\tfrac{1}{2}\Delta y\, A_{n})^{-1}
   = \frac{1}{W\mp\Delta y\, V}
     \begin{pmatrix} W_n\mp\tfrac{1}{2}\Delta y\,(\Omega_n+2V) & \pm\tfrac{1}{2}\Delta y\, a_n\Theta_n \\
                     \mp\tfrac{1}{2}\Delta y\,a_n\Theta_{n-1} & W_n\pm\tfrac{1}{2}\Delta y\,\Omega_n
     \end{pmatrix},
  \quad n \geq 0 .
\label{spectral_inverse}
\end{equation}
\end{corollary}
\begin{proof}
From the form for the spectral matrix given by (\ref{DDO:e}) we have
\begin{align*}
  \det(1\pm\tfrac{1}{2}\Delta y\, A_{n})
  & = \frac{1}{W_{n}^2}\left[ (W_{n}\pm\tfrac{1}{2}\Delta y\Omega_{n})(W_{n}\mp\tfrac{1}{2}\Delta y(\Omega_{n}+2V))
                              +\tfrac{1}{4}\Delta y^2 a_{n}^2\Theta_{n}\Theta_{n-1} \right]
  \cr
  & =  \frac{1}{W_{n}^2}\left[ W_{n}^2\mp\Delta yVW_{n}
                              -\tfrac{1}{4}\Delta y^2\left[ \Omega_{n}(\Omega_{n}+2V)-a_{n}^2\Theta_{n}\Theta_{n-1} \right] \right]
  \cr
  & = \frac{W\mp\Delta y\, V}{W_n} ,
\end{align*}
where we have employed the previous Corollary in the last step.
Constructing the matrix inverse given on the left-hand side of (\ref{spectral_inverse})
and using our determinant formula we find (\ref{spectral_inverse}).
\end{proof}

Consequent to the results of Corollary \ref{spectral_M} we require the matrix 
product found in (\ref{spectralA_recur:b}).
\begin{corollary}\label{spectral_product}
The matrix product appearing in (\ref{spectralA_recur:b}) has the evaluation
\begin{multline}
  \left( 1-\tfrac{1}{2}\Delta y\,A_{n} \right)^{-1}\left( 1+\tfrac{1}{2}\Delta y\,A_{n} \right)
  \\ 
  = \frac{1}{W+\Delta y\,V}
     \begin{pmatrix} 2W_n-W+\Delta y\,(\Omega_n+V) & -\Delta y\, a_n\Theta_n \\
                     \Delta y\,a_n\Theta_{n-1} & 2W_n-W-\Delta y\,(\Omega_n+V)
     \end{pmatrix}, 
  \quad n \geq 0 ,
  \\
  = a_n \begin{pmatrix} p_{n}(y_+) \\ p_{n-1}(y_+) \end{pmatrix} \otimes
        \begin{pmatrix} q_{n-1}(y_-), & -q_{n}(y_-) \end{pmatrix}
   -a_n\frac{W-\Delta y\,V}{W+\Delta y\,V}
        \begin{pmatrix} q_{n}(y_+) \\ q_{n-1}(y_+) \end{pmatrix} \otimes
        \begin{pmatrix} p_{n-1}(y_-), & -p_{n}(y_-) \end{pmatrix}  .
\label{spectral_prod}
\end{multline}
\end{corollary}
\begin{proof}
Using (\ref{spectral_inverse}) we multiply out the matrix product and employ the
bilinear identity (\ref{spectral_bilinear}) to simplify the diagonal elements,
resulting in (\ref{spectral_prod}).
\end{proof}

This result motivates the following definitions 
\begin{gather}
  \mathfrak{W}_{\pm} := 2W_n-W\pm\Delta y\,(\Omega_n+V),
  \quad n \geq 1 ,
  \\
  \mathfrak{T}_{+} := \Delta y\, a_n\Theta_n, \quad
  \mathfrak{T}_{-} := \Delta y\, a_n\Theta_{n-1},
  \quad n \geq 1 ,
\label{spectral_def}
\end{gather}
whilst for $ n=0 $ we have
$ \mathfrak{W}_{\pm}(n=0) := W\pm\Delta y\,V $,
$ \mathfrak{T}_{+}(n=0) := -\Delta y\,a_0\gamma^2_0 U $,
$ \mathfrak{T}_{-}(n=0) := 0 $
together with
\begin{equation}
  A^{*}_n :=
     \begin{pmatrix} \mathfrak{W}_{+} & -\mathfrak{T}_{+} \\
                     \mathfrak{T}_{-} &  \mathfrak{W}_{-}
     \end{pmatrix} .
\label{AC_defn}
\end{equation}
The variables $ \mathfrak{W}_{\pm} $ are essentially the variables $ U_n, X_n $ 
employed by Magnus in \cite{Ma_1995} respectively whilst the variables 
$ \mathfrak{T}_{\pm} $ are directly related to his $ Y_n, Z_n $ variables. 
In these new variables the bi-linear relation (\ref{spectral_bilinear}) becomes
\begin{equation}
  \det A^{*}_n = \mathfrak{W}_{+}\mathfrak{W}_{-}+\mathfrak{T}_{+}\mathfrak{T}_{-}
  = W^2-\Delta y^2V^2 ,
\label{Bi-spectral}
\end{equation}
which will be the more useful form.

It is of advantage to write out recurrence-spectral compatibility equations in 
terms of $ \mathfrak{W}_{\pm,n}, \mathfrak{T}_{\pm,n} $, where we append a subscript
to indicate the dependence on the index $ n $.
\begin{corollary}
Solving (\ref{spectralA_recur:b}) for $ A^{*}_{n+1} $ we deduce
\begin{align}
   a_na_{n+1} \mathfrak{T}_{+,n+1}
  & = -a_n(y_{+}-b_n)\mathfrak{W}_{+,n}+a_n(y_{-}-b_n)\mathfrak{W}_{-,n}+(y_{+}-b_n)(y_{-}-b_n)\mathfrak{T}_{+,n}+a^2_n\mathfrak{T}_{-,n} ,
\label{AK_comp:a}  \\
    a_n \mathfrak{T}_{-,n+1}
  & = a_{n+1} \mathfrak{T}_{+,n} ,
\label{AK_comp:b}  \\
   a_n \mathfrak{W}_{+,n+1}
  & = a_n\mathfrak{W}_{-,n}+(y_{+}-b_n)\mathfrak{T}_{+,n} ,   
\label{AK_comp:c}  \\
   a_n \mathfrak{W}_{-,n+1}
  & = a_n\mathfrak{W}_{+,n}-(y_{-}-b_n)\mathfrak{T}_{+,n} .
\label{AK_comp:d}
\end{align}
\end{corollary}
A significant result follows from the observation that the right-hand side of
(\ref{Bi-spectral}) is a polynomial in $ x $ with fixed degree independent of
$ n $ and in fact contains no dependence on $ n $. If we denote one of the zeros
of the spectral polynomial $ W^2-\Delta y^2V^2 $ as $ x_j $ then we can apply the equality for $ n \mapsto n+1 $ 
and use (\ref{AK_comp:b}) to deduce
\begin{equation}
  \mathfrak{W}_{+,n+1}(x_j)\mathfrak{W}_{-,n+1}(x_j)+\frac{a_{n+1}}{a_n}\mathfrak{T}_{+,n+1}(x_j)\mathfrak{T}_{+,n}(x_j) = 0.
\label{spectralDetZero}
\end{equation} 
This then allows us to draw the following conclusion.
\begin{proposition} 
The variable $ \mathfrak{T}_{+,n+1}(x_j) $, when written in terms of the variables at $ n $,
factorises in the following way
\begin{equation}
    -a_na_{n+1}\mathfrak{T}_{+,n+1}(x_j)\mathfrak{T}_{+,n}(x_j)
  = \left[ a_n\mathfrak{W}_{-,n}(x_j)+(y_{+,j}-b_n)\mathfrak{T}_{+,n}(x_j) \right]
    \left[ a_n\mathfrak{W}_{+,n}(x_j)-(y_{-,j}-b_n)\mathfrak{T}_{+,n}(x_j) \right] .
\end{equation}
\end{proposition} 

The coupled pair of first order divided-difference equations (\ref{DDO:a}) and
(\ref{DDO:b}) imply a second order equation for one of the components, say
$ p_n $.
\begin{proposition}[\cite{Ma_1988,Bang_2001}]
The $ \ddoAW $-semi-classical orthogonal polynomials or associated functions
satisfy a second order divided-difference equation in the following two equivalent
forms - either
\begin{multline} 
   \left[ \moAW_x\left( \frac{W_n}{\Theta_n} \right)
             +\tfrac{1}{4}\Delta y^2\ddoAW_x\left( \frac{\Omega_n+2V}{\Theta_n} \right) \right]
   \ddoAW^2_x p_n
 + \left[ \ddoAW_x\left( \frac{W_n}{\Theta_n} \right)+\moAW_x \left( \frac{2V}{\Theta_n} \right)
             +\Delta y^2\moAW_x\left( \frac{W-W_n}{\Delta y^2\Theta_n} \right) \right]
   \moAW_x\ddoAW_x p_n
 \\
 + \left[ \moAW_x\left( 4\frac{W-W_n}{\Delta y^2\Theta_n} \right)
             -\ddoAW_x\left( \frac{\Omega_n}{\Theta_n} \right) \right]
   \moAW^2_x p_n = 0 ,
\label{2ndO_dd0:a}
\end{multline}
or alternatively
\begin{multline}
  E^{+}_x\left( \frac{W+\Delta yV}{\Delta y\Theta_n} \right) (E^{+}_x)^2p_n
  + E^{-}_x\left( \frac{W-\Delta yV}{\Delta y\Theta_n} \right) (E^{-}_x)^2p_n
  \\
  + \left\{
          -E^{+}_x\left( \frac{2W_n-W+\Delta y(\Omega_n+V)}{\Delta y\Theta_n} \right)
          -E^{-}_x\left( \frac{2W_n-W-\Delta y(\Omega_n+V)}{\Delta y\Theta_n} \right)
    \right\} p_n = 0 .
\label{2ndO_dd0:c}
\end{multline}
\end{proposition}
\begin{proof}
Starting with (\ref{DDO:a}) and (\ref{DDO:b}), we can employ (\ref{spectral_bilinear}) to 
write an equation for each of $ \moAW_x p_{n-1} $ and $ \ddoAW_x p_{n-1} $ in terms of
linear combinations of $ \moAW_x p_{n} $ and $ \ddoAW_x p_{n} $. Then utilising the fact that
$ \moAW_x \ddoAW_x p_{n-1} = \ddoAW_x \moAW_x p_{n-1} $ we arrive at (\ref{2ndO_dd0:a}).
The equation (\ref{2ndO_dd0:c}) is the nodal equivalent of the former.
\end{proof}
This concludes our discussion of the spectral structures for a general quadratic lattice and we now 
turn to the simplest explicit example on the $q$-quadratic lattice.

\section{$ M=2, L=0 $ case and the Askey-Wilson polynomials}\label{M=2AW}
\setcounter{equation}{0}

Here we will employ the theory of Section \ref{SpectralS} and explicitly compute the spectral coefficients for the 
Askey-Wilson system itself because it serves both as an instructive
example for the theory of the previous section and clarifies some confusion present in the literature. 
Virtually all of our results presented here have been found by earlier studies - 
firstly by Askey and Wilson \cite{AW_1985}, then most notably by the Soviet school of Nikiforov, Suslov and Uvarov 
\cite{NSU_1984,NS_1985,NSU_1986,NS_1986a,NS_1986b} who were primarily concerned with
hypergeometric type OPS on non-uniform lattices, and in the 1988 work of
Magnus \cite{Ma_1988} whose work and intent is most similar to the spirit of our own. 

We recall the Askey-Wilson weight \cite{AW_1985}
itself has degrees $ 2N=4 $ or $ M=2 $ and $ L=0 $ with 
\begin{equation} 
  w(x) = w(x;\{a_1,\ldots,a_4\}) 
  = \frac{(z^{\pm 2};q)_{\infty}}{\sin(\theta)\prod^{4}_{j=1}(a_jz^{\pm 1};q)_{\infty}} ,
    \quad G=(-1,1) .
\label{AWwgt}
\end{equation}
Let the $j$-th elementary symmetric polynomial of $ a_1,\ldots,a_4 $ be denoted by $ \sigma_j $
for $ j=0,\ldots,4 $. 

\begin{proposition}\label{M=2_SpecCoeff}
Assume that $ q\neq 1 $ and $ |q^{-1/2}a_j| \neq 1 $ for $ j=1,\ldots, 4 $.
The spectral coefficients for the Askey-Wilson OP system are 
\begin{gather}
  W_n(x) = (1+q^{-n})(1+\sigma_4 q^{n-2})(x^2-1) - \left[ q^{-1/2}\sigma_1+q^{-3/2}\sigma_3 \right]x
           +1+q^{-1}\sigma_2+q^{-2}\sigma_4 ,
\label{AW_spec:a} \\
 \Omega_n(x)+V(x) = 2\frac{q^{n-2}\sigma_4-q^{-n}}{q^{1/2}-q^{-1/2}}x
     +\frac{q^{-n-1/2}\sigma_1+(-2+q^{-n})q^{-3/2}\sigma_3+(-2+q^{n})q^{-5/2}\sigma_1\sigma_4+q^{n-7/2}\sigma_3\sigma_4}
           {(q^{1/2}-q^{-1/2})(q^{-n}-q^{n-2}\sigma_4)} ,
\label{AW_spec:c}\\
  \Theta_n(x) = 4\frac{q^{n-3/2}\sigma_4-q^{-n-1/2}}{q^{1/2}-q^{-1/2}} ,
\label{AW_spec:b}
\end{gather}
valid for $ n \geq 0 $.
The three-term recurrence coefficients are found to be given by the standard expressions \cite{KS_1998,Ko_2007}
\begin{equation}
  a_n^2 = \tfrac{1}{4}
  \frac{(1-q^n)(1-\sigma_4 q^{n-2})\prod_{k>j}(1-a_ja_kq^{n-1})}
       {(1-\sigma_4 q^{2n-3})(1-\sigma_4 q^{2n-2})^2(1-\sigma_4 q^{2n-1})} ,
\label{AW_spec:d}
\end{equation}
and
\begin{equation}
  b_n = \left[ \sigma_1(q+\sigma_4 (q^{2n}-q^n-q^{n-1}))+\sigma_3 (1-q^n-q^{n+1}+\sigma_4 q^{2n-1})
        \right]
    \frac{q^{n-1}}{2(1-\sigma_4 q^{2n})(1-\sigma_4 q^{2n-2})} ,
\label{AW_spec:e}
\end{equation}
for $ n \geq 0 $ where we assume $ a_0=0 $.
\end{proposition}
\begin{proof}
From the weight (\ref{AWwgt}) we compute the spectral data
\begin{equation}
   W\pm \Delta yV = z^{\mp 2}\prod^{4}_{j=1}(1-a_j q^{-1/2}z^{\pm 1}) ,
\label{M=2Sdata}
\end{equation}
from which we deduce
\begin{gather}
  W(x) = 2(1+\sigma_4 q^{-2})x^2 - \left[ q^{-1/2}\sigma_1+q^{-3/2}\sigma_3 \right]x
           -1+q^{-1}\sigma_2-q^{-2}\sigma_4 ,
\label{AW_polyW} \\
 V(x) = 2\frac{q^{-2}\sigma_4-1}{q^{1/2}-q^{-1/2}}x
     +\frac{q^{-1/2}\sigma_1-q^{-3/2}\sigma_3}{q^{1/2}-q^{-1/2}} .
\label{AW_polyV}
\end{gather}
We will see that $ W^2-\Delta y^2V^2 $ will play a significant role and therefore
define another set of elementary symmetric polynomials by
\begin{equation}
  W^2-\Delta y^2V^2 = K^2\left[ x^4-e_1x^3+e_2x^2-e_3x+e_4 \right] . 
\end{equation}
We note the evaluation $ K = 4q^{\mu}=4q^{-1}\sqrt{\sigma_4} $.
We parameterise the spectral coefficients in the following way
\begin{gather}
   2W_n-W = w_2x^2+w_1x+w_0 ,
 \\
  \Theta_n = \varpi_+ ,\quad \Theta_{n-1} = \varpi_- ,
 \\
  \Omega_n+V = v_1x+v_0 ,
\end{gather}
where we know from (\ref{spectral_coeffEXP:a}-\ref{spectral_coeffEXP:c}) that 
the leading coefficients are
\begin{equation}
  w_2 = \frac{K}{2}\left( q^{n+\mu}+q^{-n-\mu} \right) ,
 \quad
  v_1 = \frac{K}{2}\frac{q^{n+\mu}-q^{-n-\mu}}{q^{1/2}-q^{-1/2}} ,
 \quad
  \varpi_+ = K \frac{q^{n+\mu+1/2}-q^{-n-\mu-1/2}}{q^{1/2}-q^{-1/2}} .
\end{equation}
From the fundamental bi-linear relation (\ref{spectral_bilinear}) we get a
system of quadratic polynomial equalities
\begin{align}
   w_2^2-\Delta v_1^2 & = K^2 ,
\label{4} \\
   2w_1w_2-2\Delta v_0v_1 & = -K^2e_1 ,
\label{3} \\
   2w_0w_2+w_1^2-\Delta(v_0^2-v_1^2)+a_n^2\Delta\varpi_+\varpi_- & = K^2e_2 ,
\label{2} \\
   2w_0w_1+2\Delta v_0v_1 & = -K^2e_3 ,
\label{1} \\
   w_0^2+\Delta v_0^2-a_n^2\Delta\varpi_+\varpi_- & = K^2e_4 ,
\label{0}
\end{align}
where $ \Delta:=(q^{1/2}-q^{-1/2})^2 $.
Now (\ref{4},\ref{2},\ref{0}) imply $ w_1^2+(w_0+w_2)^2=K^2(1+e_2+e_4) $ while (\ref{1},\ref{3})
imply $ 2w_1(w_0+w_2)=-K^2(e_1+e_3) $. Forming the sum and difference of these
two later relations we conclude that
\begin{equation}
   w_1 = \frac{W(1)-W(-1)}{2}, \quad w_0 = -w_2+\frac{W(1)+W(-1)}{2} .
\end{equation}
Using (\ref{3}) for example, along with the above solutions, we find
\begin{equation}
   v_0 = \frac{K^2e_1+w_2(W(1)-W(-1))}{2\Delta v_1} .
\end{equation}
Therefore we have succeeded in relating the sub-leading coefficients in terms of explicitly
known quantities and after simplification we arrive at (\ref{AW_spec:a}-\ref{AW_spec:c}). 
Finally, using (\ref{0}), we can solve for $ a_n^2 $ and after observing that $ e_4 $
can be expressed as
\begin{equation}
  e_4 = \left[ \frac{q^{\mu}+q^{-\mu}}{2}-\frac{W(1)+W(-1)}{2K} \right]^2
        +\frac{1}{(q^{\mu}-q^{-\mu})^2}
          \left[ e_1+\frac{q^{\mu}+q^{-\mu}}{2}\frac{W(1)-W(-1)}{K} \right]^2 ,
\end{equation}
we can factorise the resulting four terms and arrive at (\ref{AW_spec:d}). 
To find (\ref{AW_spec:e}) we start with 
(\ref{spectral_coeff_recur:b}) and use our previous results for (\ref{AW_spec:b}) and 
(\ref{AW_spec:c}). We note the terms linear in $ x $ cancel identically, as they must, 
and after some factorisation we deduce (\ref{AW_spec:e}).
\end{proof}

It is the coefficients (\ref{AW_spec:a}-\ref{AW_spec:c}) along with first order divided-difference equation (\ref{DDO:a}) 
that constitutes the structural relation for the Askey-Wilson polynomials, that is the 
analog of first order difference or differential relations for the classical orthogonal
polynomials. The divided-difference relations reported in \cite{Ko_2007} are all of
second order. 

It is an easy task to evaluate the Askey-Wilson integral \cite{AW_1985,KS_1998} as 
a $ q$-factorial and therefore compute the moments.  The Askey-Wilson integral is defined
by
\begin{equation}
  I_{2}(a_1,a_2,a_3,a_4) 
  := \int_{\mathbb{T}} \frac{dz}{2\pi iz} \frac{(z^{\pm 2};q)_{\infty}}
                                   {\prod^{4}_{j=1}(a_{j}z^{\pm 1};q)_{\infty}} ,
\end{equation}
with $ |a_j|<1 $ for $ j=1,\ldots,4 $.

Our method is to apply
the general system of moment recurrences to the case $ M=2 $ and we find that it coincides with the
recurrence of Kalnins and Miller \cite{KM_1989} and Koelink and Koornwinder \cite{KK_1992}.
\begin{proposition}[\cite{Wi_2010b}]
The Askey-Wilson integral satisfies the two-term linear recurrence
\begin{equation}
   (\sigma_4-1)I_{2}(qa_1,a_2,a_3,a_4)
   = (a_1a_2-1)(a_1a_3-1)(a_1a_4-1)I_{2}(a_1,a_2,a_3,a_4) , 
\label{AWintegral}
\end{equation}
which is solved by
\begin{equation}
  I_{2}(a_1,a_2,a_3,a_4) 
   = 2\frac{(\sigma_4;q)_{\infty}}{(q;q)_{\infty}\prod_{k>j}(a_{j}a_{k};q)_{\infty}} .
\end{equation}
Consequently the moments are given by
\begin{align}
   m_{0,n}(a_1)
   & = \pi\frac{(a_1a_2,a_1a_3,a_1a_4;q)_{n}}{(\sigma_4;q)_{n}}I_{2}(a_1,a_2,a_3,a_4)
   \\
   & = 2\pi\frac{(q^{n}\sigma_4;q)_{\infty}}
                {(q^{n}a_1a_2,q^{n}a_1a_3,q^{n}a_1a_4,a_2a_3,a_2a_4,a_3a_4,q;q)_{\infty}} .
\label{AWmoment}
\end{align}
\end{proposition}
\begin{proof}
We only require the $ k=0 $ case of (\ref{Mrecur}) with $ a=a_1 $ and from (\ref{M=2Sdata}) we compute the 
relevant coefficients as
$ \delta_{2,1}(a) = q^{-1}a^{-1}(1-\sigma_4) $ and $ \delta_{2,0}(a) = q^{-1}(a+a^{-1})(\sigma_4-1)+q^{-1}(\sigma_1-\sigma_3) $.
The solution of the recurrence (\ref{AWintegral}) follows from the arguments given in \cite{KK_1992}.
\end{proof}

The specialisation of the second-order divided-difference equation (\ref{2ndO_dd0:c})
with the $ M=2 $ spectral coefficients of Proposition \ref{M=2_SpecCoeff} is given by
\begin{equation}
  \frac{\prod^{4}_{j=1}(1-a_jz)}{qz^2-1} \left[ (E^{+}_x)^2p_n-p_n \right]
  +  \frac{\prod^{4}_{j=1}(z-a_j)}{z^2-q} \left[ (E^{-}_x)^2p_n-p_n \right]
  + q^{-n-1}(1-q^n)(q^n\sigma_4-q)(z^2-1)p_n = 0 .
\end{equation}
This is soluble in terms of basic hypergeometric functions and their polynomial solutions have been 
given by \cite{Ma_1988} or using factorisation methods in \cite{Ba_1999,BM_1999,BH_2003} as Askey-Wilson polynomials.
The Askey-Wilson polynomials have an explicit form as a balanced $ {}_4\varphi_3 $ 
function \cite{KS_1998} with the manifest symmetry under $ z \leftrightarrow z^{-1} $
and are given in the monic form by
\begin{equation}
  \pi_n(x) = \frac{(a_1a_2,a_1a_3,a_1a_4;q)_{n}}{(2a_1)^n(q^{n-1}\sigma_4;q)_{n}}
          {}_4\varphi_3\left( \begin{array}{c}
                                   q^{-n}, \sigma_4 q^{n-1}, a_1z, a_1z^{-1} \\
                                   a_1a_2, a_1a_3, a_1a_4
                              \end{array}; q,q
                           \right) ,
\end{equation}
or the alternative form \cite{GM_1994} which is manifestly symmetric under permutations of $ a_1,a_2,a_3,a_4 $
\begin{equation}
  \pi_n(x) = (2z)^{-n}\frac{(a_1z,a_2z,a_3z,a_4z,\sigma_4 q^{-1};q)_{n}}
                         {(z^2;q)_{n}(\sigma_4 q^{-1};q)_{2n}}
   {}_{8}W_{7}(q^{-n}z^{-2};q^{-n},a_1z^{-1},a_2z^{-1},a_3z^{-1},a_4z^{-1};\frac{q^{2-n}}{\sigma_4}) .
\end{equation}

From Eq. (2.6) of \cite{Ra_1986a} or equivalently from Eq. (4.18) of \cite{IR_1991} with $ \alpha=0 $
and Eq. (III.23) of \cite{GR_2004} we deduce the expression for the Stieltjes 
function as a very-well-poised $ {}_{8}W_{7} $
\begin{equation}
  f(x) = 
   \frac{4\pi (q^{-1}\sigma_4;q)_{\infty}}{(q;q)_{\infty}\prod_{k>j}(a_ja_k;q)_{\infty}}
   \frac{(1-qz^{-2})}{z\prod_{1\leq j\leq 4}(1-a_j z^{-1})}
   {}_{8}W_{7}(qz^{-2};\frac{q}{a_1}z^{-1},\frac{q}{a_2}z^{-1},\frac{q}{a_3}z^{-1},\frac{q}{a_4}z^{-1},q;q,q^{-1}\sigma_4) , 
\label{SFdefn_AW}
\end{equation}
which exhibits the parameter symmetry. This expression is the one valid on the second Riemann
sheet of the cut $x$-plane or exterior to the unit circle in $z $, $|z|>1 $.
There are evaluations of $ q_n(x) $ in \cite{Ra_1986a} and \cite{IR_1991} that are simple
generalisations of the above formulae but we do not need to discuss them here.

Having an explicit form for the Stieltjes function we have two tasks at hand - to verify that
is satisfies the inhomogeneous divided-difference equation (\ref{spectral_DD_st}) and the
large $x$ expansion formula (\ref{xLarge_SF:a}). 
\begin{proposition}
The Stieltjes function given by Eq.(\ref{SFdefn_AW}) satisfies the inhomogeneous divided
difference equation (\ref{spectral_DD_st}) with the constant
\begin{equation}
   U =\frac{8\pi}{q-1}\frac{(q^{-1}\sigma_4;q)_{\infty}}{(q;q)_{\infty}\prod_{k>j}(a_ja_k;q)_{\infty}} . 
\end{equation}
\end{proposition}
\begin{proof}
We require an alternative form
\begin{multline}
  f(x) = 4\pi\frac{qa_1}{\sigma_4}
  \frac{1}{(\dfrac{q}{a_2a_3},\dfrac{q}{a_2a_4},\dfrac{q}{a_3a_4};q)_{\infty}}
  \Bigg\{-\frac{(\dfrac{q}{a_2z},\dfrac{q}{a_3z},\dfrac{q}{a_4z},\dfrac{\sigma_4}{qa_1z},\dfrac{q^2a_1z}{\sigma_4};q)_{\infty}}
               {(a_2a_3,a_2a_4,a_3a_4,a_1z^{-1},a_2z^{-1},a_3z^{-1},a_4z^{-1},a_1z;q)_{\infty}}
  \\
         +\frac{(q^{-1}\sigma_4,q\dfrac{a_1}{a_2},q\dfrac{a_1}{a_3},q\dfrac{a_1}{a_4},\dfrac{q^2}{\sigma_4};q)_{\infty}}
               {(q,qa_1^2;q)_{\infty}\prod_{k>j}(a_ja_k;q)_{\infty}(1-a_1 z)(1-a_1 z^{-1})}
    {}_{8}W_{7}(a_1^2;a_1 z,a_1 z^{-1},a_1a_2,a_1a_3,a_1a_4;\frac{q^2}{\sigma_4})
  \Bigg\} ,
\label{Sfunc_AW}
\end{multline}
which makes the $ z \leftrightarrow z^{-1} $ symmetry manifest in the second
term only. This is derived from the previous expression by utilising Eq.(III.37) 
combined with Eq.(III.23) of \cite{GR_2004}, and Eq.(III.36) for the term which
specialises.
We are now in a position to verify that (\ref{Sfunc_AW}) satisfies (\ref{spectral_DD_st}).
This relies on the fact that (\ref{spectral_DD_st}) with the data (\ref{M=2Sdata}) is precisely the 
contiguous relation, Eq. (2.2) of \cite{IR_1991}, which in our context states
\begin{multline} 
   z^2\frac{\prod_{j=2,3,4}(1-q^{-1/2}a_j z^{-1})}{1-q^{1/2}a_1 z}{}_{8}W_{7}(a_1^2;q^{1/2}a_1z,q^{-1/2}a_1z^{-1},a_1a_2,a_1a_3,a_1a_4;\frac{q^2}{\sigma_4})
  \\
  -z^{-2}\frac{\prod_{j=2,3,4}(1-q^{-1/2}a_j z)}{1-q^{1/2}a_1 z^{-1}}{}_{8}W_{7}(a_1^2;q^{-1/2}a_1z,q^{1/2}a_1z^{-1},a_1a_2,a_1a_3,a_1a_4;\frac{q^2}{\sigma_4})
  \\
  = q^{-3/2} a_2a_3a_4(1-\frac{q}{\sigma_4})(z-z^{-1}){}_{8}W_{7}(a_1^2;q^{1/2}a_1z,q^{1/2}a_1z^{-1},a_1a_2,a_1a_3,a_1a_4;\frac{q}{\sigma_4}) .
\end{multline} 
We have also used the specialisation formula when $ aq=bc $
\begin{equation} 
   {}_{8}W_{7}(a;b,c,d,e,f;q,\frac{q^2a^2}{bcdef}) 
  = \frac{(aq,aq/de,aq/df,aq/ef;q)_{\infty}}{(aq/d,aq/e,aq/f,aq/def;q)_{\infty}} ,
\end{equation} 
which applies to the right-hand side of the previous equation.
A consequence one can draw from this calculation is the explicit evaluation of 
the constant spectral coefficient.
\end{proof}

\begin{remark}
The constant $ U $ can be found in another way. This is by noting the initial 
condition for $ \Phi_0 $ is given by (\ref{spectral_coeff_init}) and that $ \Phi_n $ is given purely by the
leading order term in (\ref{AW_spec:b}), so that knowing the normalisation $ \gamma_0^2 $ allows
us to make the evaluation. 
\end{remark}

\begin{proposition}
The Stieltjes function possesses the explicit large $x$ moment generating function formula
\begin{multline}
   f(x) = -\frac{4\pi a_1(\sigma_4;q)_{\infty}}{(q;q)_{\infty}\prod_{k>j}(a_ja_k;q)_{\infty}(1-a_1z)(1-a_1z^{-1})}
          {}_4\varphi_3\left( \begin{array}{c}
                                  q, a_1a_2, a_1a_3, a_1a_4 \\
                                   qa_1z, qa_1z^{-1},\sigma_4 
                              \end{array}; q,q
                           \right)
  \\
   -\frac{4\pi(a_2a_3a_4z^{-1},qz^{-2};q)_{\infty}}{z(a_2a_3,a_2a_4,a_3a_4,a_2z^{-1}, a_3z^{-1}, a_4z^{-1};q)_{\infty}}
    \frac{1}{\phi_{\infty}(x;a_1)}
          {}_3\varphi_2\left( \begin{array}{c}
                                   a_2z^{-1}, a_3z^{-1}, a_4z^{-1} \\
                                   a_2a_3a_4z^{-1},qz^{-2}
                              \end{array}; q,q
                           \right) .
\label{Sfunc_exp}
\end{multline}
\end{proposition}
\begin{proof}
By applying the transformation formula (III.36) of \cite{GR_2004} to the $ {}_8W_{7} $ function
in (\ref{Sfunc_AW}) we get a sum of two $ {}_4\varphi_{3} $ functions one of which reduces to a
$ {}_3\varphi_{2} $ function. After some simplification we get (\ref{Sfunc_exp}). We observe that
the $ {}_4\varphi_{3} $ term is precisely the second term of (\ref{xLarge_SF:a}) given our 
formula for the moments (\ref{AWmoment}). This allows us to conclude that
\begin{equation}
   f_{\infty}(x) =  -\frac{4\pi(a_2a_3a_4z^{-1},qz^{-2};q)_{\infty}}{z(a_2a_3,a_2a_4,a_3a_4,a_2z^{-1}, a_3z^{-1}, a_4z^{-1};q)_{\infty}}
          {}_3\varphi_2\left( \begin{array}{c}
                                   a_2z^{-1}, a_3z^{-1}, a_4z^{-1} \\
                                   a_2a_3a_4z^{-1},qz^{-2}
                              \end{array}; q,q
                           \right) .
\end{equation}
\end{proof}

\section{Deformation Differences}\label{DeformS}
\setcounter{equation}{0}

In this section we discuss the deformation structures based upon a nonlinear 
lattice in the $ u $ variable (or any number of variables for that matter) and the 
corresponding divided-difference operators. We emphasise that 
many of the results for the spectral structure have a parallel result in
the deformation structure although there will be crucial differences in the
details.
Denote the deformation variable $ u $ and its forward and backward shifts by
$ v_{\pm} = E^{\pm}_u u $ and their difference by $ \Delta v := v_+-v_- $.
The reader should be aware that the deformation lattice does not have to be the 
same type as the spectral lattice.
Our $\ddoAW$-semi-classical weight $ w $, defined by Definition \ref{spectral_DD}, 
acquires an additional dependence 
on a deformation variable $ u $ and furthermore satisfies a deformation divided-difference 
equation with respect to $ u $.
\begin{definition}\label{deform_wgt}
The {\it deformed  $\ddoAW$-semi-classical weight} satisfies
\begin{equation}
    R\ddoAW_{u} w = 2S\moAW_{u} w ,
\label{deform_DD_wgt:a}
\end{equation}
or equivalently
\begin{equation}
    \frac{w(x;v_+)}{w(x;v_-)} = \frac{R+\Delta vS}{R-\Delta vS}(x;u) .
\label{deform_DD_wgt:b}
\end{equation}
for polynomials $ S(x;u), R(x;u) $ irreducible in $ x $ and $ u \in J $.
In addition we assume $ R\pm \Delta vS \neq 0 $ for all $ x \in G, u \in J $
and that the deformed OPS exists, i.e. that $ \gamma_n(v_{\pm}) \neq 0 $ for all $ n\in\Z_{\geq 0} $. We also
require the condition $ \gamma_0(v_+)(R-\Delta vS)+\gamma_0(v_-)(R+\Delta vS) \neq 0 $.
\end{definition}
There is a very simple motivation for this relation. In the process of extending the weights beyond those of the classical cases in
the Askey table by increasing the degrees $ {\rm deg}_x W\geq 3, {\rm deg}_x V\geq 2 $ the weights
acquire additional parameters which if suitably chosen sit in a completely reflexive or 
symmetric way to the spectral variable and therefore will satisfy an analogous linear, homogeneous
and first order divided-difference equation in any of these new parameters.

There are conditions imposed on the weight given that it satisfies both
(\ref{spectral_DD_wgt:a}) or (\ref{spectral_DD_wgt:b}) and (\ref{deform_DD_wgt:a}) or (\ref{deform_DD_wgt:b}).
\begin{proposition}\label{spectral+deform_wgt}
The spectral data polynomials $ W(x;u), V(x;u) $ and the deformation data polynomials 
$ R(x;u), S(x;u) $ satisfy the compatibility relations
\begin{equation}
  \left[ \ddoAW_{x}\frac{2S}{R}-\ddoAW_{u}\frac{2V}{W} \right]
  \left[ 1-\tfrac{1}{16}\Delta y^2\Delta v^2 \ddoAW_{x}\frac{2S}{R}\ddoAW_{u}\frac{2V}{W} \right]
 = \tfrac{1}{4}\Delta y^2\ddoAW_{x}\frac{2S}{R}\left( \moAW_{u}\frac{2V}{W} \right)^2
   - \tfrac{1}{4}\Delta v^2\ddoAW_{u}\frac{2V}{W}\left( \moAW_{x}\frac{2S}{R} \right)^2 ,
\label{wgt_consist:a}
\end{equation}
or alternatively
\begin{equation}
  \frac{W+\Delta yV}{W-\Delta yV}(x;v_+)
  \frac{R+\Delta vS}{R-\Delta vS}(y_-;u)
  = 
  \frac{W+\Delta yV}{W-\Delta yV}(x;v_-)
  \frac{R+\Delta vS}{R-\Delta vS}(y_+;u) .
\label{wgt_consist:b}
\end{equation}
\end{proposition}
\begin{proof}
From the fact that the weight satisfies the over-determined system of first order
divided-difference equations (\ref{spectral_DD_wgt:a}) and (\ref{deform_DD_wgt:a})
we first compute that
\begin{gather}
  \left[ 1-\tfrac{1}{16}\Delta y^2\Delta v^2 \ddoAW_{x}\frac{2S}{R}\ddoAW_{u}\frac{2V}{W} \right]
  \moAW_u\ddoAW_x w 
  = \left[ \moAW_u\frac{2V}{W}+\tfrac{1}{4}\Delta v^2\ddoAW_u\frac{2V}{W}\moAW_x\frac{2S}{R} \right]\moAW_u\moAW_x w ,
  \\
  \left[ 1-\tfrac{1}{16}\Delta y^2\Delta v^2 \ddoAW_{x}\frac{2S}{R}\ddoAW_{u}\frac{2V}{W} \right]
  \moAW_x\ddoAW_u w 
  = \left[ \moAW_x\frac{2S}{R}+\tfrac{1}{4}\Delta y^2\ddoAW_x\frac{2S}{R}\moAW_u\frac{2V}{W} \right]\moAW_x\moAW_u w ,
\end{gather}
using (\ref{DD_calculus:b}). Now we compute 
\begin{gather}
  \ddoAW_u\ddoAW_x w = \ddoAW_u\frac{2V}{W}\moAW_u\moAW_x w+\moAW_u\frac{2V}{W}\ddoAW_u\moAW_x w ,
  \\
  \ddoAW_x\ddoAW_u w = \ddoAW_x\frac{2S}{R}\moAW_x\moAW_u w+\moAW_x\frac{2S}{R}\ddoAW_x\moAW_u w ,
\end{gather}
using (\ref{DD_calculus:a})
and use the results of the previous set of equations to derive two independent 
relations linking $ \ddoAW_u\ddoAW_x w $ and $ \moAW_u\moAW_x w $. Comparison of
these two latter relations leads to (\ref{wgt_consist:a}). 

To establish (\ref{wgt_consist:b}) we note
\begin{align}
  w(y_{+};v_{+}) 
 & = \frac{W+\Delta yV}{W-\Delta yV}(x;v_{+})w(y_{-};v_{+})
   = \frac{W+\Delta yV}{W-\Delta yV}(x;v_{+})
     \frac{R+\Delta vS}{R-\Delta vS}(y_{-};u)w(y_{-};v_{-}) ,
\end{align}
whereas
\begin{align}
  w(y_{+};v_{+}) 
 & = \frac{R+\Delta vS}{R-\Delta vS}(y_{+};u)w(y_{+};v_{-})
   = \frac{R+\Delta vS}{R-\Delta vS}(y_{+};u)
     \frac{W+\Delta yV}{W-\Delta yV}(x;v_{-})w(y_{-};v_{-}) .
\end{align}
\end{proof}

\begin{proposition}
As a consequence of Definition \ref{deform_wgt} the Stieltjes transform satisfies 
the inhomogeneous equation
\begin{equation}
    R\ddoAW_{u} f = 2S\moAW_{u} f + T .
\label{deform_DD_st}
\end{equation}
The deformed {\it $\ddoAW$-semi-classical class} of orthogonal
polynomial systems are characterised by the property that $ R(x;u), S(x;u) $ and $ T(x;u) $ 
in (\ref{deform_DD_st}) are polynomials in $ x $. 
\end{proposition}
\begin{proof}
From the definition (\ref{ops_stieltjes}) and (\ref{deform_DD_wgt:a}) we compute
\begin{equation}
  \ddoAW_u f(x;u) = \int \ddoAW y\,\frac{\moAW_u w(y;u)}{x-y}\frac{2S(y;u)}{R(y;u)} .
\end{equation}
Now we observe that the rational function $ 2S(y;u)/[(x-y)R(y;u)] $ has the partial
fraction expansion
\begin{equation}
  \frac{2S(y;u)}{(x-y)R(y;u)} 
  = \frac{2S(x;u)}{(x-y)R(x;u)}+\sum_{j, R(x_j)=0}\frac{2S(x_j;u)}{R'(x_j;u)(x-x_j)}\frac{1}{y-x_j} .
\end{equation}
Consequently
\begin{align}
  \ddoAW_u f(x;u) 
  & = \frac{2S(x;u)}{R(x;u)}\int \ddoAW y\,\frac{\moAW_u w(y;u)}{x-y}
      + \sum_{j, R(x_j)=0}\frac{2S(x_j;u)}{R'(x_j;u)(x-x_j)}\int \ddoAW y\,\frac{\moAW_u w(y;u)}{y-x_j} ,
  \\
  & = \frac{2S(x;u)}{R(x;u)}\moAW_u f(x;u)
      - \sum_{j, R(x_j)=0}\frac{2S(x_j;u)}{R'(x_j;u)(x-x_j)}\moAW_u f(x_j;u) .
\end{align}
We conclude that (\ref{deform_DD_st}) follows with $ {\rm deg}_x T \leq {\rm deg}_xR-1 $.
\end{proof}

\begin{proposition}
Compatibility of (\ref{deform_DD_st}) and (\ref{spectral_DD_st}) implies the
following identity on $ U $ and $ T $.
\begin{multline}
   \Delta y\left[ \frac{(W+\Delta yV)(x;v_+)}{(W+\Delta yV)(x;v_-)}(R+\Delta vS)(y_-;u)U(x;v_-)
                  -(R-\Delta vS)(y_-;u)U(x;v_+) \right]
 \\
 = \Delta v\left[ (W+\Delta yV)(x;v_+)T(y_-;u)
                  -(W-\Delta yV)(x;v_+)\frac{(R-\Delta vS)(y_-;u)}{(R-\Delta vS)(y_+;u)}T(y_+;u) \right] .
\label{ST_consist}
\end{multline}
\end{proposition}
\begin{proof}
We begin by defining the ratio
\begin{equation}
  \chi
    \equiv \frac{(W+\Delta yV)(x;v_{+})}{(W+\Delta yV)(x;v_{-})}
           \frac{(R+\Delta v S)(y_{-};u)}{(R+\Delta v S)(y_{+};u)}
    = \frac{(W-\Delta yV)(x;v_{+})}{(W-\Delta yV)(x;v_{-})}
      \frac{(R-\Delta v S)(y_{-};u)}{(R-\Delta v S)(y_{+};u)} ,
\label{twist}
\end{equation}
by virtue of (\ref{wgt_consist:b}).
Consider (\ref{spectral_DD_st}) in the form
\begin{equation}
  (W-\Delta yV)(x;u)f(y_+;u)-(W+\Delta yV)(x;u)f(y_-;u)-\Delta yU(x;u) = 0 ,
\label{auxUT}
\end{equation}
and construct the combination
$ (W+\Delta yV)(x;v_{+})/(W+\Delta yV)(x;v_{-})(R+\Delta v S)(y_{-};u)\times\text{(\ref{auxUT})}(u\mapsto v_{-}) 
  -(R-\Delta v S)(y_{-};u)\times\text{(\ref{auxUT})}(u\mapsto v_{+}) $.
Using the above identity we find the two terms containing $ f $ possess the factors
$ (R+\Delta v S)(y_{-};u)f(y_{-};v_{-})-(R-\Delta v S)(y_{-};u)f(y_{-};v_{+}) $ and
$ (R+\Delta v S)(y_{+};u)f(y_{+};v_{-})-(R-\Delta v S)(y_{+};u)f(y_{+};v_{+}) $,
into which we can apply (\ref{deform_DD_st}) in the form
\begin{equation}
  (R-\Delta vS)(x;u)f(x;v_{+})-(R+\Delta vS)(x;u)f(x;v_{-})-\Delta vT(x;u) = 0 .
\end{equation}
Then (\ref{ST_consist}) immediately follows.
\end{proof}

We can extend the notion of a generic or regular $\ddoAW$-semi-classical weight, given in
Definition \ref{regDSCwgt}, to the deformed situation by the following definition
\begin{definition}
A {\it regular, deformed $\ddoAW$-semi-classical weight} is one that satisfies the
requirements of the Definition \ref{regDSCwgt} and 
$ {\rm deg}_x R = {\rm deg}_x S = L $ and $ {\rm deg}_x T = {\rm deg}_x R-1 $.
Clearly $ M $ in Definition \ref{regDSCwgt} and $ L $ are related, depending on the
specific case on hand. 
\end{definition}

The analog of Proposition \ref{spectral_Cff} is the following.
\begin{proposition}\label{deform_Cff}
Let the deformation coefficients $ R_n(x;u), \Gamma_n(x;u), \Xi_n(x;u), \Phi_n(x;u), \Psi_n(x;u) $
be defined in terms of bilinear formulae involving products of the polynomials 
and associated functions by
\begin{multline}
  \frac{2}{H_n}R_{n}
  = (R+\Delta vS)\left[
     -\frac{1}{a_n(v_-)}+p_{n-1}(;v_{+})q_{n}(;v_{-})-p_{n}(;v_{+})q_{n-1}(;v_{-})            
                 \right]
  \\
   +(R-\Delta vS)\left[
     -\frac{1}{a_n(v_+)}+p_{n-1}(;v_{-})q_{n}(;v_{+})-p_{n}(;v_{-})q_{n-1}(;v_{+})
                 \right],
\label{deform_coeff:a}
\end{multline}
\begin{multline}
  \frac{\Delta v}{H_n}\;\Gamma_{n}
  = (R+\Delta vS)\left[
      \frac{1}{a_n(v_-)}-p_{n-1}(;v_{+})q_{n}(;v_{-})-p_{n}(;v_{+})q_{n-1}(;v_{-})
                 \right]
  \\
   +(R-\Delta vS)\left[
     -\frac{1}{a_n(v_+)}+p_{n-1}(;v_{-})q_{n}(;v_{+})+p_{n}(;v_{-})q_{n-1}(;v_{+})    
                 \right],
\label{deform_coeff:b}
\end{multline}
\begin{multline}
  \frac{\Delta v}{H_n}\;\Xi_{n}
  = (R+\Delta vS)\left[
      \frac{1}{a_n(v_-)}+p_{n-1}(;v_{+})q_{n}(;v_{-})+p_{n}(;v_{+})q_{n-1}(;v_{-})
                 \right]
  \\
   +(R-\Delta vS)\left[
     -\frac{1}{a_n(v_+)}-p_{n-1}(;v_{-})q_{n}(;v_{+})-p_{n}(;v_{-})q_{n-1}(;v_{+})    
                 \right],
\label{deform_coeff:c}
\end{multline}
\begin{equation}
  \frac{\Delta v}{2H_n}\;\Phi_{n} =
  (R+\Delta vS)p_{n}(;v_{+})q_{n}(;v_{-})-(R-\Delta vS)p_{n}(;v_{-})q_{n}(;v_{+}),
\label{deform_coeff:d}
\end{equation}
and
\begin{equation}
  \frac{\Delta v}{2H_n}\;\Psi_{n} =
  -(R+\Delta vS)p_{n-1}(;v_{+})q_{n-1}(;v_{-})+(R-\Delta vS)p_{n-1}(;v_{-})q_{n-1}(;v_{+}),
\label{deform_coeff:e}
\end{equation}
for $ n \geq 0 $ and where the decoupling factor $ H_n(u) $ has $ \deg_x H_n = 0 $.
Then the deformation coefficients of the deformed {\it $\ddoAW$-semi-classical class}
are polynomials in $ x $.
\end{proposition}
\begin{proof}
We offer a proof in the spirit of that for Proposition \ref{spectral_Cff}, using
Laguerre's method, and apply it to (\ref{deform_coeff:d}) first.
From (\ref{deform_DD_st}) and (\ref{ops_eps}) we have
\begin{align}
  0 = & (R-\Delta vS)f(;v_{+})-(R+\Delta vS)f(;v_{-})-\Delta vT ,
\nonumber\\
    = & (R-\Delta vS)\frac{q_n+p^{(1)}_{n-1}}{p_n}(;v_{+})-(R+\Delta vS)\frac{q_n+p^{(1)}_{n-1}}{p_n}(;v_{-})-\Delta vT ,
\nonumber\\
    = & \frac{(R-\Delta vS)p_n(;v_{-})p^{(1)}_{n-1}(;v_{+})-(R+\Delta vS)p_n(;v_{+})p^{(1)}_{n-1}(;v_{-})-\Delta vTp_n(;v_{+})p_n(;v_{-})}
             {p_n(;v_{+})p_n(;v_{-})}
\nonumber\\
      & \qquad
        +\frac{(R-\Delta vS)p_n(;v_{-})q_{n}(;v_{+})-(R+\Delta vS)p_n(;v_{+})q_{n}(;v_{-})}{p_n(;v_{+})p_n(;v_{-})} .
\end{align}
In this last expression the numerator of the first term is clearly a polynomial in $ x $
which implies the numerator of the second term is likewise. This latter numerator is an
odd function under the exchange $ v_{+} \leftrightarrow v_{-} $ and we denote it by
$ \Delta v\Phi_n(x;u)/(2H_n) $, thus deriving (\ref{deform_coeff:d}). 
The relation (\ref{deform_coeff:e}) is simply the case of (\ref{deform_coeff:d})
under $ n \mapsto n-1 $.

The remaining relations may now be derived from the following argument. From the 
workings of the previous paragraph we know that
\begin{equation}  
  \frac{\Delta v\Phi_n}{2H_n} =
  (R-\Delta vS)p_n(;v_{-})p^{(1)}_{n-1}(;v_{+})-(R+\Delta vS)p_n(;v_{+})p^{(1)}_{n-1}(;v_{-})-\Delta vTp_n(;v_{+})p_n(;v_{-}) .
\label{auxPhi}
\end{equation}  
By multiplying the left-hand side of this relation by the Casoratian
$ a_n(v_{+})[p_{n-1}p^{(1)}_{n-1}-p_{n}p^{(1)}_{n-2}](;v_{+})=1 $ we observe that the
resulting equation separates into two terms with the factorisation
\begin{multline} 
  p_n(;v_{+})\left[ -a_n(v_{+})\frac{\Delta v\Phi_n}{2H_n}p^{(1)}_{n-2}(;v_{+})
                    +(R+\Delta vS)p^{(1)}_{n-1}(;v_{-})+\Delta vTp_n(;v_{-}) \right] 
  \\
  = p^{(1)}_{n-1}(;v_{+})\left[ -a_n(v_{+})\frac{\Delta v\Phi_n}{2H_n}p_{n-1}(;v_{+})
                    +(R-\Delta vS)p_n(;v_{-}) \right] ,
\end{multline}
which implies that this expression contains the polynomial factors $ p_n(;v_{+})p^{(1)}_{n-1}(;v_{+}) $
and therefore can be written as
$ p_n(;v_{+})p^{(1)}_{n-1}(;v_{+})\pi_1 $, where $ \pi_1 $ is a polynomial in $ x $.
Given that the two factors $ p_n(;v_{+}) $, $ p^{(1)}_{n-1}(;v_{+}) $
are non-zero this then leads to two evaluations for
$ \pi_1 $, which by constructing a suitable combination of these and employing
the Casoratian once more we find, upon some simplification,
\begin{equation} 
   \pi_1 = a_n(v_{+})\left[ -(R+\Delta vS)p_{n-1}(;v_{+})q_{n}(;v_{-})+(R-\Delta vS)p_{n}(;v_{-})q_{n-1}(;v_{+}) \right] .
\label{auxPi1}
\end{equation} 
Alternatively we could have multiplied the left-hand side of (\ref{auxPhi}) by the
Casoratian $ a_n(v_{-})[p_{n-1}p^{(1)}_{n-1}-p_{n}p^{(1)}_{n-2}](;v_{-})=1 $ and
deduced the factorisation
\begin{multline} 
  p_n(;v_{-})\left[ -a_n(v_{-})\frac{\Delta v\Phi_n}{2H_n}p^{(1)}_{n-2}(;v_{-})
                    -(R-\Delta vS)p^{(1)}_{n-1}(;v_{+})+\Delta vTp_n(;v_{+}) \right] 
  \\
  = p^{(1)}_{n-1}(;v_{-})\left[ -a_n(v_{-})\frac{\Delta v\Phi_n}{2H_n}p_{n-1}(;v_{-})
                    -(R+\Delta vS)p_n(;v_{+}) \right] ,
\end{multline}
This fact then allows us to conclude that either of these expressions can be written
as $ p_n(;v_{-})p^{(1)}_{n-1}(;v_{-})\pi_2 $, where $ \pi_2 $ is another polynomial in $ x $.
Using an identical procedure to that employed above we can infer that
\begin{equation} 
   \pi_2 = a_n(v_{-})\left[ -(R+\Delta vS)p_{n}(;v_{+})q_{n-1}(;v_{-})+(R-\Delta vS)p_{n-1}(;v_{-})q_{n}(;v_{+}) \right] .
\label{auxPi2}
\end{equation}
Relations (\ref{deform_coeff:a}) and (\ref{deform_coeff:b}) then follow from the 
definitions 
\begin{equation}
  2R_n-\Delta v\Gamma_n+2H_n\frac{R+\Delta vS}{a_n(v_{-})} = -\frac{2H_n}{a_n(v_{+})}\pi_1 ,
  \quad
  2R_n+\Delta v\Gamma_n+2H_n\frac{R-\Delta vS}{a_n(v_{+})} =  \frac{2H_n}{a_n(v_{-})}\pi_2 ,
\end{equation}
whereas (\ref{deform_coeff:c}) follows from
$ \Delta v(\Gamma_n+\Xi_n)/2H_n = (R+\Delta vS)/a_n(v_{-})-(R-\Delta vS)/a_n(v_{+}) $.
\end{proof}

Analogous to the large $ x $ expansions of the spectral coefficients we have the
following expansions for the deformation coefficients.
\begin{proposition}\label{deform_Cff_Exp}
Let $ L = \max({\rm deg}_x R,{\rm deg}_x S) $.
As $ x \to \infty $ we have the leading orders of the terminating expansions of
the deformation coefficients
\begin{multline}
  \frac{2}{H_n}R_n = 
  -(\gamma_n(v_+)+\gamma_n(v_-))\left[ \frac{R-\Delta vS}{\gamma_{n-1}(v_+)}+\frac{R+\Delta vS}{\gamma_{n-1}(v_-)}
                                \right]
  \\
  +\sum^{n-1}_{i=0}(b_i(v_+)-b_i(v_-))\left[ (R+\Delta vS)\frac{\gamma_{n}(v_+)}{\gamma_{n-1}(v_-)}
                                             -(R-\Delta vS)\frac{\gamma_{n}(v_-)}{\gamma_{n-1}(v_+)}
                                      \right]x^{-1} + {\rm O}(x^{L-2})
  \quad n \geq 0 ,
\label{deform_coeff_Exp:a}
\end{multline}
\begin{multline}
  \frac{\Delta v}{2H_n}\Phi_n =
   \left[ (R+\Delta vS)\frac{\gamma_{n}(v_+)}{\gamma_{n}(v_-)}-(R-\Delta vS)\frac{\gamma_{n}(v_-)}{\gamma_{n}(v_+)}
   \right]x^{-1}
  \\
  +\left( (R+\Delta vS)\frac{\gamma_{n}(v_+)}{\gamma_{n}(v_-)}\left[\sum^{n}_{i=0}b_i(v_-)-\sum^{n-1}_{i=0}b_i(v_+)\right]
         +(R-\Delta vS)\frac{\gamma_{n}(v_-)}{\gamma_{n}(v_+)}\left[\sum^{n-1}_{i=0}b_i(v_-)-\sum^{n}_{i=0}b_i(v_+)\right]
   \right)x^{-2}
  \\
  + {\rm O}(x^{L-3}),
  \quad n \geq 0 ,
\label{deform_coeff_Exp:b}
\end{multline}
\begin{multline}
  \frac{\Delta v}{2H_n}\Psi_n =
  -\left[ (R+\Delta vS)\frac{\gamma_{n-1}(v_+)}{\gamma_{n-1}(v_-)}-(R-\Delta vS)\frac{\gamma_{n-1}(v_-)}{\gamma_{n-1}(v_+)}
   \right]x^{-1}
  \\
  -\left( (R+\Delta vS)\frac{\gamma_{n-1}(v_+)}{\gamma_{n-1}(v_-)}\left[\sum^{n-1}_{i=0}b_i(v_-)-\sum^{n-2}_{i=0}b_i(v_+)\right]
         +(R-\Delta vS)\frac{\gamma_{n-1}(v_-)}{\gamma_{n-1}(v_+)}\left[\sum^{n-2}_{i=0}b_i(v_-)-\sum^{n-1}_{i=0}b_i(v_+)\right]
   \right)x^{-2}
  \\
  + {\rm O}(x^{L-3}),
  \quad n \geq 0 ,
\label{deform_coeff_Exp:c}
\end{multline}
\begin{multline}
  \frac{\Delta v}{H_n}\Gamma_n = 
   (\gamma_n(v_-)-\gamma_n(v_+))\left[ \frac{R+\Delta vS}{\gamma_{n-1}(v_-)}+\frac{R-\Delta vS}{\gamma_{n-1}(v_+)}
                                \right]
  \\
  +\sum^{n-1}_{i=0}(b_i(v_+)-b_i(v_-))\left[ (R+\Delta vS)\frac{\gamma_{n}(v_+)}{\gamma_{n-1}(v_-)}
                                             +(R-\Delta vS)\frac{\gamma_{n}(v_-)}{\gamma_{n-1}(v_+)}
                                      \right]x^{-1} + {\rm O}(x^{L-2}),
  \quad n \geq 0 ,
\label{deform_coeff_Exp:d}
\end{multline}
and 
\begin{multline}
  \frac{\Delta v}{H_n}\Xi_n = 
   (\gamma_n(v_-)+\gamma_n(v_+))\left[ \frac{R+\Delta vS}{\gamma_{n-1}(v_-)}-\frac{R-\Delta vS}{\gamma_{n-1}(v_+)}
                                \right]
  \\
  -\sum^{n-1}_{i=0}(b_i(v_+)-b_i(v_-))\left[ (R+\Delta vS)\frac{\gamma_{n}(v_+)}{\gamma_{n-1}(v_-)}
                                             +(R-\Delta vS)\frac{\gamma_{n}(v_-)}{\gamma_{n-1}(v_+)}
                                      \right]x^{-1} + {\rm O}(x^{L-2}),
  \quad n \geq 0 .
\label{deform_coeff_Exp:e}
\end{multline}
\end{proposition}
\begin{proof}
These formulae follow from the substitution of the expansions (\ref{ops_pExp}) and
(\ref{ops_eExp}) into the definitions (\ref{deform_coeff:a}-\ref{deform_coeff:e}).
\end{proof}
\begin{remark}
Unlike the case of the spectral coefficients (see Proposition \ref{spectral_Cff_Exp})
the leading terms of the deformation
coefficients are not determined by the weight data alone but depend non-trivially upon
the three-term recurrence coefficients.
\end{remark}

\begin{corollary}
In the regular, deformed $\ddoAW$-semi-classical case
the deformation coefficients $ R_n, \Gamma_n, \Phi_n, \Psi_n, \Xi_n $ are polynomials
in the spectral variable $ x $ with degrees independent of $ n $,
$ {\rm deg}_x R_n = {\rm deg}_x \Gamma_n = {\rm deg}_x \Xi_n = \max({\rm deg}_x R,{\rm deg}_x S) $ 
and
$ {\rm deg}_x \Phi_n = {\rm deg}_x \Psi_n = \max({\rm deg}_x R,{\rm deg}_x S)-1 $.
\end{corollary}

The expressions for the deformation coefficients (\ref{deform_coeff:e}-\ref{deform_coeff:e})
can be inverted yielding a system of linear divided-difference equations in the
deformation variable for the polynomials and associated functions.
\begin{proposition}\label{deform_DDeqn}
The OPS corresponding to a deformed $\ddoAW$-semi-classical weight satisfies the 
{\it deformation divided-difference} equation
\begin{equation}
   \ddoAW_{u} Y_n := B_{n}\moAW_{u} Y_n
   = \frac{1}{R_{n}}
     \begin{pmatrix} \Gamma_n & \Phi_n \\
                     \Psi_{n} & \Xi_n
     \end{pmatrix}\moAW_{u} Y_n,
  \quad n \geq 0 .
\label{deformDD_Y}
\end{equation}
\end{proposition}
\begin{proof}
The essence of this result involves the inversion of Proposition \ref{deform_Cff} which is
carried out in a manner analogous to the proof of Proposition \ref{spectral_DDO}.
\end{proof}

Of the four coefficients $ \Gamma_n, \Phi_n, \Psi_n, \Xi_n $ only two are 
independent because of the following relations.
\begin{proposition}\label{dMatrixElements}
The deformation coefficients satisfy the linear identity
\begin{equation}
   \Psi_{n} = -\frac{a_{n}}{a_{n-1}}\Phi_{n-1},
  \quad n \geq 1 ,
\label{deform_reln:a} \\
\end{equation}
and the trace identity
\begin{equation}
  \Delta v\;(\Gamma_n+\Xi_n) 
 = 2H_n\left[ 
       \frac{R+\Delta vS}{a_n(v_-)}-\frac{R-\Delta vS}{a_n(v_+)} \right],
  \quad n \geq 0 .
\label{deform_reln:c}
\end{equation}
\end{proposition}
\begin{proof}
The first relation follows by comparison of (\ref{deform_coeff:e}) with (\ref{deform_coeff:d}) and
setting $ H_{n}/H_{n-1}=a_{n}/a_{n-1} $. The second identity is easily seen from the expressions
(\ref{deform_coeff:b}) and (\ref{deform_coeff:c}) and has already been observed in the
conclusions of the workings for the proof of Proposition \ref{deform_Cff}.
\end{proof}

The initial values of the deformation coefficients are
\begin{gather}
   R_{0} = -\tfrac{1}{2}H_0[\gamma_0(v_-)+\gamma_0(v_+)]
            \left( \frac{R+\Delta vS}{a_0(v_-)\gamma_{0}(v_-)}+\frac{R-\Delta vS}{a_0(v_+)\gamma_{0}(v_+)} \right) ,
\label{deform_coeff_init:a}\\
   \Gamma_{0} = \frac{1}{\Delta v}H_0[\gamma_0(v_-)-\gamma_0(v_+)]
            \left( \frac{R+\Delta vS}{a_0(v_-)\gamma_{0}(v_-)}+\frac{R-\Delta vS}{a_0(v_+)\gamma_{0}(v_+)} \right) ,
\label{deform_coeff_init:b} \\
   \Xi_{0} =  \frac{1}{\Delta v}H_0[\gamma_0(v_-)+\gamma_0(v_+)]
            \left( \frac{R+\Delta vS}{a_0(v_-)\gamma_{0}(v_-)}-\frac{R-\Delta vS}{a_0(v_+)\gamma_{0}(v_+)} \right) ,
\label{deform_coeff_init:e} \\
   \Phi_{-1} = 0, \qquad 
   \Phi_{0} = -2 H_0\gamma_0(v_-)\gamma_0(v_+)T ,
\label{deform_coeff_init:c} \\
   \Psi_{0} = 0 .
\label{deform_coeff_init:d}
\end{gather}

Compatibility between the matrix recurrence relation and the deformation divided-difference 
equation implies the next result.
\begin{proposition}\label{deform+threeT}
The recurrence matrix and the deformation matrix satisfy
\begin{equation}
   B_{n+1}\cdot \moAW_{u} K_n-\moAW_{u} K_n\cdot B_{n} 
  = \ddoAW_{u} K_n-\tfrac{1}{4}\Delta v^2 B_{n+1}\cdot \ddoAW_{u} K_n\cdot B_{n},
  \quad n \geq 0 ,
\label{deform_Brecur:a}
\end{equation}
or equivalently
\begin{equation}
  K_n(;v_+)\left( 1-\tfrac{1}{2}\Delta vB_{n} \right)^{-1}\left( 1+\tfrac{1}{2}\Delta vB_{n} \right) 
  = 
  \left( 1-\tfrac{1}{2}\Delta vB_{n+1} \right)^{-1}\left( 1+\tfrac{1}{2}\Delta vB_{n+1} \right)K_n(;v_-),
  \quad n \geq 0 ,
\label{deform_Brecur:b}
\end{equation}
\end{proposition}
\begin{proof}
The first form of the compatibility relation is deduced by comparing
\begin{align}
   \ddoAW_u Y_{n+1} & = B_{n+1}\moAW_u Y_{n+1} \nonumber\cr
  & = B_{n+1}\moAW_u(K_nY_n) \nonumber\cr
  & = B_{n+1}\left( \moAW_u K_n\cdot\moAW_u Y_n+\tfrac{1}{4}\Delta v^2\ddoAW_u K_n\cdot\ddoAW_u Y_n \right) \nonumber\cr
  & = B_{n+1}\left( \moAW_u K_n+\tfrac{1}{4}\Delta v^2\ddoAW_u K_n\cdot B_n \right)\moAW_u Y_n \nonumber ,
\end{align}
and
\begin{equation}
   \ddoAW_u Y_{n+1} = \ddoAW_u K_nY_{n}
  = \left( \ddoAW_u K_n\cdot\moAW_u Y_n+\moAW_u K_n\cdot\ddoAW_u Y_n \right)
  = \left( \ddoAW_u K_n+\moAW_u K_n\cdot B_n \right)\moAW_u Y_n .
\end{equation}
The second form is found by computing $ Y_{n+1}(x;v_{+}) $
in terms of $ Y_{n}(x;v_{-}) $ in the two possible ways corresponding to the orders of
the operations $ n \mapsto n+1 $ and $ v_{-} \mapsto v_{+} $.
\end{proof}

Compatibility of the three term recurrence relation and the deformation divided-difference 
equation implies the following result which can be found by considering 
the representations of the deformation coefficients given in Proposition \ref{deform_Cff}. 
\begin{proposition}\label{deform_Cff_recur}
The deformation coefficients $ R_n, \Gamma_n, \Phi_n $ satisfy recurrence relations in $ n $,
\begin{multline}
  \frac{a_{n+1}(v_-)}{H_{n+1}}(-2R_{n+1}+\Delta v\Gamma_{n+1})
 +\frac{a_{n}(v_-)}{H_{n}}(2R_{n}+\Delta v\Gamma_{n})
 \\
 = -[x-b_n(v_-)]\frac{\Delta v}{H_n}\Phi_n
   +2a_n(v_-)\left( \frac{R+\Delta vS}{a_n(v_-)}-\frac{R-\Delta vS}{a_n(v_+)} \right) ,
  \quad n \geq 0 ,
\label{Defm_recur:a}
\end{multline}
\begin{multline}
  \frac{a_{n+1}(v_+)}{H_{n+1}}(2R_{n+1}+\Delta v\Gamma_{n+1})
 +\frac{a_{n}(v_+)}{H_{n}}(-2R_{n}+\Delta v\Gamma_{n})
 \\
 = -[x-b_n(v_+)]\frac{\Delta v}{H_n}\Phi_n
   +2a_n(v_+)\left( \frac{R+\Delta vS}{a_n(v_-)}-\frac{R-\Delta vS}{a_n(v_+)} \right) ,
  \quad n \geq 0 ,
\label{Defm_recur:b}
\end{multline}
Recurrences (\ref{Defm_recur:a}) and (\ref{Defm_recur:b}) are the analogs of 
combinations of (\ref{spectral_coeff_recur:a}) and (\ref{spectral_coeff_recur:b}).
\end{proposition}
\begin{proof}
The simplest proof of these relations is the one employing the definitions 
(\ref{deform_coeff:a},\ref{deform_coeff:b},\ref{deform_coeff:d}) of Proposition \ref{deform_Cff}.
Taking the first relation we deduce
\begin{align*}
 &
  \frac{a_{n+1}(v_-)}{H_{n+1}}(-2R_{n+1}+\Delta v\Gamma_{n+1})
 +\frac{a_{n}(v_-)}{H_{n}}(2R_{n}+\Delta v\Gamma_{n})
 +[x-b_n(v_-)]\frac{\Delta v}{H_n}\Phi_n
 \\
 & = 2(R+\Delta vS)\left[ 1-p_n(;v_+)(a_{n+1}(v_-)q_{n+1}(;v_-)+b_n(v_-)q_n(v_-)+a_n(v_-)q_{n-1}(v_-)-xq_n(v_-)) \right]
 \\
 & \phantom{=}
    -2(R-\Delta vS)\left[ \frac{a_n(v_-)}{a_n(v_+)}-q_n(;v_+)(a_{n+1}(v_-)p_{n+1}(;v_-)+b_n(v_-)p_n(v_-)+a_n(v_-)p_{n-1}(v_-)-xp_n(v_-)) \right]
 \\
 & = 2a_n(v_-)\left( \frac{R+\Delta vS}{a_n(v_-)}-\frac{R-\Delta vS}{a_n(v_+)} \right) ,
\end{align*}
where essential use of the three-term recurrence relations (\ref{ops_threeT},\ref{ops_AF_threeT})
has been made. The second relation can be found in an identical manner.
\end{proof}

We find bilinear identities for the deformation matrix analogous to those of
Proposition \ref{spectral_det}.
\begin{proposition}\label{deform_Cff_Id}
The deformation coefficients satisfy the bilinear or determinantal identity
\begin{equation}
 R^2_n+\tfrac{1}{4}\Delta v^2 \left[ \Gamma_{n}\Xi_{n}-\Phi_{n}\Psi_{n} \right]
 = - H_nR_n\left[ 
       \frac{R+\Delta vS}{a_n(v_-)}+\frac{R-\Delta vS}{a_n(v_+)} \right],
  \quad n \geq 0 .
\label{deform_reln:d}
\end{equation}
\end{proposition}
\begin{proof}
A direct method of proof is possible substituting the expressions for the deformation coefficients
$ R_n, \Gamma_{n}, \Xi_{n}, \Phi_{n}, \Psi_{n} $ in terms of products of polynomials and
associated functions as given by (\ref{deform_coeff:a}-\ref{deform_coeff:e}) into the left-hand
side of the (\ref{deform_reln:d}). After expansion and considerable cancellation we recognise the
the form for the right-hand side.
\end{proof}

From the above result we can deduce matrix identities analogous to Corollary \ref{spectral_M}.
\begin{corollary}\label{deform_Matrix}
The matrix factor appearing in (\ref{deform_Brecur:b}) has the determinant evaluation
\begin{equation}
  \det(1\pm\tfrac{1}{2}\Delta v B_{n})
  = -\frac{2H_n(R\mp\Delta vS)}{R_na_n(v_{\pm})},
  \quad n \geq 0 ,
\label{deform_determinant}
\end{equation}
and its inverse is
\begin{equation}
   (1\pm\tfrac{1}{2}\Delta v B_{n})^{-1}
   = -\frac{a_n(v_{\pm})}{2H_n(R\mp\Delta v S)}
     \begin{pmatrix} R_n\pm\tfrac{1}{2}\Delta v\;\Xi_n & \mp\tfrac{1}{2}\Delta v\;\Phi_n \\
                     \mp\tfrac{1}{2}\Delta v\;\Psi_{n} & R_n\pm\tfrac{1}{2}\Delta v\;\Gamma_n
     \end{pmatrix},
  \quad n \geq 0 .
\label{deform_inverse}
\end{equation}
\end{corollary}
\begin{proof}
We note from the matrix formula for $ B_n $ (\ref{deformDD_Y}) that
\begin{align}
  \det(1\pm\tfrac{1}{2}\Delta v B_{n})
  & = \frac{1}{R_n^2}\left\{ R^2_n+\tfrac{1}{4}\Delta v^2 \left[ \Gamma_{n}\Xi_{n}-\Phi_{n}\Psi_{n} \right] 
                               \pm\tfrac{1}{2}\Delta v\;R_n(\Gamma_n+\Xi_n) \right\} 
 \nonumber\cr
  & = -\frac{2H_n}{R_n}\frac{R\mp\Delta v S}{a_n(v_{\pm})} ,
\end{align}
where we have used (\ref{deform_reln:d}) and (\ref{deform_reln:c}) in the last step.
The inversion formula follows from this last result and the standard formula for an inverse.
\end{proof}

Consequent to the results of Corollary \ref{deform_Matrix} we have the following expression
for the matrix product appearing in (\ref{deform_Brecur:b}).
\begin{corollary}\label{deform_product}
The matrix product (\ref{deform_Brecur:b}) has the evaluation
\begin{multline}
  \left( 1-\tfrac{1}{2}\Delta v\,B_{n} \right)^{-1}\left( 1+\tfrac{1}{2}\Delta v\,B_{n} \right) 
  \\
  = \frac{a_n(v_-)}{2H_n(R+\Delta v\,S)}
     \begin{pmatrix} 2R_n+2H_n\dfrac{R-\Delta v\,S}{a_n(v_+)}+\Delta v\,\Gamma_n & \Delta v\,\Phi_n \\
                     \Delta v\,\Psi_n & 2R_n+2H_n\dfrac{R+\Delta v\,S}{a_n(v_-)}-\Delta v\,\Gamma_n .
     \end{pmatrix},
  \quad n \geq 0
\label{deform_prod} \\
  =-a_n(v_-)
        \begin{pmatrix} p_{n}(;v_+) \\ p_{n-1}(;v_+) \end{pmatrix} \otimes
        \begin{pmatrix} q_{n-1}(;v_-), & -q_{n}(;v_-) \end{pmatrix}
   +a_n(v_-)\frac{R-\Delta v\,S}{R+\Delta v\,S}
        \begin{pmatrix} q_{n}(;v_+) \\ q_{n-1}(;v_+) \end{pmatrix} \otimes
        \begin{pmatrix} p_{n-1}(;v_-), & -p_{n}(;v_-) \end{pmatrix}   .
\end{multline}
\end{corollary}
\begin{proof}
Using the inverse (\ref{deform_inverse}) we form the matrix product and employ (\ref{deform_reln:d})
to simplify the diagonal elements, with the result given by (\ref{deform_prod}).
\end{proof}

This result also motivates the following definitions 
\begin{gather}
  \mathfrak{R}_{+} := 2R_n+2H_n\frac{R-\Delta v\,S}{a_n(v_+)}+\Delta v\,\Gamma_n,
\label{deform_def:a}\\
  \mathfrak{R}_{-} := 2R_n+2H_n\frac{R+\Delta v\,S}{a_n(v_-)}-\Delta v\,\Gamma_n,
\label{deform_def:b}\\
  \mathfrak{P}_{+} :=-\Delta v\,\Phi_n, \quad
  \mathfrak{P}_{-} := \Delta v\,\Psi_n ,
\label{deform_def:c}
\end{gather}
valid for $ n \geq 1 $, and for $ n=0 $ we have
\begin{gather}
  \mathfrak{R}_{+}(n=0) := -2\frac{H_0}{a_0(v_-)}\frac{\gamma_0(v_+)}{\gamma_0(v_-)}(R+\Delta v\,S) ,
  \\
  \mathfrak{R}_{-}(n=0) := -2\frac{H_0}{a_0(v_+)}\frac{\gamma_0(v_-)}{\gamma_0(v_+)}(R-\Delta v\,S) ,
  \\
  \mathfrak{P}_{+}(n=0) := 2\Delta v H_0\gamma_0(v_+)\gamma_0(v_-)T ,\quad
  \mathfrak{P}_{-}(n=0) := 0 ,
\label{deform_def_n=0}
\end{gather} 
together with
\begin{equation}
  B^{*}_n :=
     \begin{pmatrix} \mathfrak{R}_{+} & -\mathfrak{P}_{+} \\
                     \mathfrak{P}_{-} &  \mathfrak{R}_{-}
     \end{pmatrix} .
\label{BC_defn}
\end{equation}
In these variables the deformation bi-linear relation (\ref{deform_reln:d}) takes
the form
\begin{equation}
  \det B^{*}_n = \mathfrak{R}_{+}\mathfrak{R}_{-} + \mathfrak{P}_{+}\mathfrak{P}_{-} 
  = \frac{4H^2_n}{a_n(v_+)a_n(v_-)}(R^2-\Delta v^2 S^2) .
\label{Bi-deform}
\end{equation}

It is of advantage to write out recurrence-deformation compatibility equations in 
terms of $ \mathfrak{R}_{\pm,n}, \mathfrak{P}_{\pm,n} $, where again we append a subscript
to indicate the dependence on the index $ n $. We will henceforth denote variables evaluated
at the advanced and retarded deformation variable by $ \hat{a}_n = a_n(v_+), \check{a}_n = a_n(v_-) $.
\begin{corollary}\label{Bn+1}
Solving (\ref{deform_Brecur:b}) for $ B^{*}_{n+1} $ we deduce
\begin{align}
   \frac{H_n}{H_{n+1}}\hat{a}_{n+1}\check{a}_{n+1} \mathfrak{P}_{+,n+1}
  & = -\check{a}_n(x-\hat{b}_n)\mathfrak{R}_{+,n}+\hat{a}_n(x-\hat{b}_n)\mathfrak{R}_{-,n}+(x-\hat{b}_n)(x-\check{b}_n)\mathfrak{P}_{+,n}+\hat{a}_n\check{a}_n\mathfrak{P}_{-,n} ,
\label{BK_comp:a}  \\
    \frac{H_n}{H_{n+1}} \mathfrak{P}_{-,n+1}
  & = \mathfrak{P}_{+,n} ,
\label{BK_comp:b}  \\
   \frac{H_n}{H_{n+1}} \mathfrak{R}_{+,n+1}
  & = \frac{\hat{a}_n}{\hat{a}_{n+1}}\mathfrak{R}_{-,n}+\frac{x-\hat{b}_n}{\hat{a}_{n+1}}\mathfrak{P}_{+,n} ,   
\label{BK_comp:c}  \\
   \frac{H_n}{H_{n+1}} \mathfrak{R}_{-,n+1}
  & = \frac{\check{a}_n}{\check{a}_{n+1}}\mathfrak{R}_{+,n}-\frac{x-\check{b}_n}{\check{a}_{n+1}}\mathfrak{P}_{+,n} .
\label{BK_comp:d}
\end{align}
\end{corollary}

The compatibility relation between the spectral and deformation divided-difference 
equations takes the following form.
\begin{proposition}\label{spectral+deform}
The spectral matrix $ A_n(x;u) $ and the deformation matrix $ B_n(x;u) $ satisfy the
$\ddoAW$-Schlesinger equation
\begin{multline}
  \ddoAW_{u}A_n-\ddoAW_{x}B_n+\moAW_{u}A_n\cdot\moAW_{x}B_n-\moAW_{x}B_n\cdot\moAW_{u}A_n \\
  = 
   \tfrac{1}{4}\Delta v^2\moAW_{x}B_n\cdot\ddoAW_{u}A_n\cdot
   \left( 1-\tfrac{1}{16}\Delta y^2\Delta v^2\ddoAW_{x}B_n\cdot\ddoAW_{u}A_n \right)^{-1}
   \left( \moAW_{x}B_n+\tfrac{1}{4}\Delta y^2\ddoAW_{x}B_n\cdot\moAW_{u}A_n \right) \\
  -\tfrac{1}{4}\Delta y^2\moAW_{u}A_n\cdot\ddoAW_{x}B_n\cdot 
   \left( 1-\tfrac{1}{16}\Delta y^2\Delta v^2\ddoAW_{u}A_n\cdot\ddoAW_{x}B_n \right)^{-1}
   \left( \moAW_{u}A_n+\tfrac{1}{4}\Delta v^2\ddoAW_{u}A_n\cdot\moAW_{x}B_n \right) ,
\label{spectral+deform:a}
\end{multline}
or the equivalent form
\begin{multline}
  \left( 1-\tfrac{1}{2}\Delta yA_{n}(;v_+) \right)^{-1}\left( 1+\tfrac{1}{2}\Delta yA_{n}(;v_+) \right) 
  \left( 1-\tfrac{1}{2}\Delta vB_{n}(y_-;) \right)^{-1}\left( 1+\tfrac{1}{2}\Delta vB_{n}(y_-;) \right) \\
  = 
  \left( 1-\tfrac{1}{2}\Delta vB_{n}(y_+;) \right)^{-1}\left( 1+\tfrac{1}{2}\Delta vB_{n}(y_+;) \right) 
  \left( 1-\tfrac{1}{2}\Delta yA_{n}(;v_-) \right)^{-1}\left( 1+\tfrac{1}{2}\Delta yA_{n}(;v_-) \right) .
\label{spectral+deform:b}
\end{multline}
\end{proposition}
\begin{proof}
The first form follows from $ \ddoAW_{x}\ddoAW_{u}Y_n = \ddoAW_{u}\ddoAW_{x}Y_n $ and the Leibniz formulae
of (\ref{DD_calculus:a}) and (\ref{DD_calculus:b}), from which we deduce
\begin{gather*}
  \left[ 1-\tfrac{1}{16}\Delta y^2\Delta v^2 \ddoAW_{u}A_n\cdot\ddoAW_{x}B_n \right]
  \ddoAW_x\moAW_u Y 
  = \left[ \moAW_uA_n+\tfrac{1}{4}\Delta v^2\ddoAW_uA_n\cdot\moAW_xB_n \right]\moAW_u\moAW_x Y ,
  \\
  \left[ 1-\tfrac{1}{16}\Delta y^2\Delta v^2 \ddoAW_{x}B_n\cdot\ddoAW_{u}A_n \right]
  \ddoAW_u\moAW_x Y 
  = \left[ \moAW_xB_n+\tfrac{1}{4}\Delta y^2\ddoAW_xB_n\cdot\moAW_uA_n \right]\moAW_x\moAW_u Y .
\end{gather*}
Using equivalent forms of the spectral and deformation divided-difference equations (\ref{DDO:e}) and (\ref{deformDD_Y})
we see
\begin{multline*}
   Y(y_+;v_+) 
   = \left( 1-\tfrac{1}{2}\Delta yA_{n}(;v_+) \right)^{-1}\left( 1+\tfrac{1}{2}\Delta yA_{n}(;v_+) \right) Y(y_-;v_+)
 \\
   = \left( 1-\tfrac{1}{2}\Delta yA_{n}(;v_+) \right)^{-1}\left( 1+\tfrac{1}{2}\Delta yA_{n}(;v_+) \right) 
  \left( 1-\tfrac{1}{2}\Delta vB_{n}(y_-;) \right)^{-1}\left( 1+\tfrac{1}{2}\Delta vB_{n}(y_-;) \right) Y(y_-;v_-) ,
\end{multline*}
whereas
\begin{multline*}
   Y(y_+;v_+) 
   = \left( 1-\tfrac{1}{2}\Delta vB_{n}(y_+;) \right)^{-1}\left( 1+\tfrac{1}{2}\Delta vB_{n}(y_+;) \right) Y(y_+;v_-)
 \\
   = \left( 1-\tfrac{1}{2}\Delta vB_{n}(y_+;) \right)^{-1}\left( 1+\tfrac{1}{2}\Delta vB_{n}(y_+;) \right) 
  \left( 1-\tfrac{1}{2}\Delta yA_{n}(;v_-) \right)^{-1}\left( 1+\tfrac{1}{2}\Delta yA_{n}(;v_-) \right) Y(y_-;v_-) .
\end{multline*}
This gives us the form (\ref{spectral+deform:b}).
\end{proof}

\begin{remark}
Let us denote the two fixed points of the $x$-lattice by $ x_L $ and $ x_R $.
By analogy with the linear lattices let us conjecture the existence of 
fundamental solutions to (\ref{DDO:e}) about $ x=x_{L},x_{R} $ which we
denote by $ Y_{L}, Y_{R} $ respectively.
Furthermore let us define the connection matrix
\begin{equation}
  P(x;u) := \left( Y_{R}(x;u) \right)^{-1}Y_{L}(x;u) .
\label{ConnMatrix}
\end{equation}
From (\ref{DDO:e}) it is clear that $ P $ is a $\ddoAW$-constant function with respect 
to $ x $, that is to say
\begin{equation*}
  P(y_{+};u) = P(y_{-};u) .
\end{equation*}
In addition it is clear from (\ref{deformDD_Y}) that this type of deformation is also a
{\it connection preserving deformation} in the sense that
\begin{equation}
  P(x;v_{+}) = P(x;v_{-}) .
\label{ConnPreserve}
\end{equation}
\end{remark}

The compatibility relation (\ref{spectral+deform:b}) 
$ \chi\, B^*_n(y_{+};u)A^*_n(x;v_{-}) = A^*_n(x;v_{+})B^*_n(y_{-};u) $
may be written component-wise in the new variables in the more practical form as
\begin{equation}
 \chi \left[
 \mathfrak{W}_{+}(x;v_{-})\mathfrak{R}_{+}(y_{+};u) - \mathfrak{T}_{-}(x;v_{-})\mathfrak{P}_{+}(y_{+};u)
           \right] 
 = \mathfrak{W}_{+}(x;v_{+})\mathfrak{R}_{+}(y_{-};u) - \mathfrak{T}_{+}(x;v_{+})\mathfrak{P}_{-}(y_{-};u) ,
\label{S+D:a}
\end{equation}
\begin{equation}
 \chi \left[
 \mathfrak{T}_{+}(x;v_{-})\mathfrak{R}_{+}(y_{+};u) + \mathfrak{W}_{-}(x;v_{-})\mathfrak{P}_{+}(y_{+};u)
           \right]
 = \mathfrak{T}_{+}(x;v_{+})\mathfrak{R}_{-}(y_{-};u) + \mathfrak{W}_{+}(x;v_{+})\mathfrak{P}_{+}(y_{-};u) ,
\label{S+D:b}
\end{equation}
\begin{equation}
 \chi \left[
 \mathfrak{T}_{-}(x;v_{-})\mathfrak{R}_{-}(y_{+};u) + \mathfrak{W}_{+}(x;v_{-})\mathfrak{P}_{-}(y_{+};u)
           \right]
 = \mathfrak{T}_{-}(x;v_{+})\mathfrak{R}_{+}(y_{-};u) + \mathfrak{W}_{-}(x;v_{+})\mathfrak{P}_{-}(y_{-};u) ,
\label{S+D:c}
\end{equation}
\begin{equation}
 \chi \left[
 \mathfrak{W}_{-}(x;v_{-})\mathfrak{R}_{-}(y_{+};u) - \mathfrak{T}_{+}(x;v_{-})\mathfrak{P}_{-}(y_{+};u)
           \right]
 = \mathfrak{W}_{-}(x;v_{+})\mathfrak{R}_{-}(y_{-};u) - \mathfrak{T}_{-}(x;v_{+})\mathfrak{P}_{+}(y_{-};u) ,
\label{S+D:d}
\end{equation}
where $ \chi $ is defined by (\ref{twist}).
The real content of (\ref{S+D:a}-\ref{S+D:d}) lie in the case $ n\geq 1 $ as the 
$ n=0 $ evaluations of (\ref{S+D:c}) is trivially satisfied, whereas that of
(\ref{S+D:a}) and (\ref{S+D:d}) are identically satisfied due to the 
definition (\ref{twist}), and (\ref{S+D:b}) is equivalent to the consistency
identity (\ref{ST_consist}).

Having assembled all of the ingredients of our theory we have to draw them together and to
perform three tasks. The first is to parameterise the spectral coefficients in a minimal
way consistent with the constraints of the spectral structures, namely 
(\ref{spectral_bilinear}). The second task is to close this system, that is
to say relate the deformation coefficients to the spectral coefficients,
preferably with the parameterisation found in the first step. The third
task is to derive the dynamical equations for this parameterisation with 
respect to the deformation variable.
There are two avenues of approach to the problem of extracting useful information 
from the above compatibility relations (\ref{S+D:a}-\ref{S+D:d}). One way is to 
clear the denominators on both sides and resolve the resulting expressions in terms 
of $ \bigoplus_{k\geq 0}x^k+\Delta y \bigoplus_{k\geq 0}x^k $.
This is useful for certain results, as we shall see in the application. The other
method is to work with the evaluation of the spectral and deformation 
coefficients at certain key ordinates and construct a parameterisation based
upon these variables. 
A fundamental and crucial role will be played by the zeros of the polynomials
$ W^2-\Delta y^2V^2 $, $ \Theta_{n} $ and $ \chi $
as these ordinates, however these cannot be interpreted as singularities in the
spectral plane.

\begin{remark}\label{common_zeros}
The compatibility relations (\ref{S+D:a}-\ref{S+D:d}) are satisfied identically at 
the common zeros of $ (W^2-\Delta y^2 V^2)(\tilde{x};v_{\pm}) $ 
in the sense that both sides of the relations are identically zero,
that is to say at those fixed zeros independent of $ u $. In addition one can show 
that the left-hand side of the bilinear relation (\ref{Bi-deform}) vanishes at the movable
zeros in either of the two above cases. This implies that when
$ (W^2-\Delta y^2 V^2)(\tilde{x};v_{\pm}) = 0 $ then 
$ (R^2-\Delta v^2 S^2)(\tilde{y}_{\pm};u) = 0 $. 
\end{remark}

\section{$ M=3, L=1 $ case and deformation of the Askey-Wilson OPS}\label{M=3AW}
\setcounter{equation}{0}

The previous sections treated the $\ddoAW$-semi-classical orthogonal polynomial
system with general divided-difference operators $ \ddoAW $, $ \moAW $ and arbitrary
degrees $ M $, $ L $ for the spectral and deformation coefficients respectively.
Here we apply the forgoing theory to the symmetrised form of the $q$-quadratic lattice
in the elliptic sub-case, which singles out the Askey-Wilson weight and the
non-trivial examples of deformations of this weight. 
We will consider the situation of the deformation variable also on this 
sub-case of the $q$-quadratic lattice.
The deformation variable is $ u = (t+t^{-1})/2 $ with $ t=e^{i\phi} $
and the analogous relations for the nodes on the deformation lattice are 
\begin{gather}
   v_{+}+v_{-} = (q^{1/2}+q^{-1/2})u ,
\\
   v_{+}-v_{-} = \tfrac{1}{2}(q^{1/2}-q^{-1/2})(t-t^{-1}) ,
\\
   \Delta v^2 = (q^{1/2}-q^{-1/2})^2(u^2-1) ,
\\
   v_{+}v_{-} = u^2+\tfrac{1}{4}(q^{1/2}-q^{-1/2})^2 .
\end{gather}

\subsection{Moments and Integrals}\label{SSect7_1}
Our starting point is the simplest extension of the $ M=2 $ case
\begin{equation}
  W\pm \Delta y V = z^{\mp 3}\prod^{6}_{j=1}(1-a_jq^{-1/2}z^{\pm 1}) ,
\label{deform_AW_SCff:a}
\end{equation}
which is a minimal, natural extension of the Askey-Wilson case given by (\ref{M=2Sdata}). 
The reader should appreciate that
all of the conclusions we are going to draw will follow entirely from (\ref{deform_AW_SCff:a}), so
that any valid solution of the Pearson equation (\ref{spectral_DD_wgt:b}) for the weight with this data
is as acceptable as any other. The most important consideration in selecting a solution is the
support for the weight and existence of the moment data.
As we will see (i.e. in Propositions \ref{M=3L=1U} and \ref{M=3L=1T}) the only place where a specific choice for
the weight enters is in determining the initial values of recurrences or divided-difference equations.
The four fixed parameters $ a_1,\ldots,a_4 $ appear in the 
same form as they do in the Askey-Wilson weight ($ M=2 $) and so constitutes the base 
expression, and we seek to introduce an extra parameter and its associated deformation variable. 
We set $ a_5=\alpha t $, $ a_6=\alpha t^{-1} $ and designate $ \alpha, t $ as the deformation parameter and variable respectively.

\begin{remark}
One might entertain the possibility of a weight with the spectral data
\begin{equation}  
    W\pm\Delta yV = z^{\mp 2}\prod^{4}_{j=1}(1-q^{-1/2}a_j z^{\pm 1})(1-q^{-1/2}tz^{\pm 1})(1-q^{\alpha-1/2}z^{\mp 1}) , 
\end{equation}
which is a $ M=3 $ case on a $q$-quadratic spectral lattice, however it is $q$-linear in either of the deformation
variables $ t, q^{\alpha} $. This is in fact the spectral data for the weight in the integral representation of the 
very-well-poised $ {}_8W_7 $  basic hypergeometric function, see Eq. (1.13) of \cite{Ra_1986b}.
This function is the classical seed solution for the $ E^{(1)}_7 $ $q$-Painlev\'e system as constructed by \cite{KMNOY_2004}.
\end{remark}

Our first undertaking is to derive linear recurrences for the moments which will be required
in the verification of key identities from Sections \ref{SpectralS} and \ref{DeformS}, to make
contact with the $ n=0 $ seed solutions of the evolution equations derived later and for
checking purposes. Through suitable choices of parameters $ k, b, a $ the moment recurrences
can be recast as $q$-difference equations for an integral $ I_3(a_1,\ldots,a_6) $ which is
a generalisation of the Askey-Wilson integral (\ref{AWintegral}).
\begin{proposition}[\cite{Wi_2010b}]
Let $ \sigma_k $ denote the $k$-th elementary symmetric polynomial in $ a_1,\ldots,a_4 $ or $ a_1,\ldots,a_6 $
depending on the context.
The integral $ I_3(a_1,\ldots,a_6) $ satisfies a three-term recurrence in one variable, which we take without loss of generality
with respect to $ a_1 $,
\begin{multline}
   0 = \prod_{j\neq 1}(a_1a_j-1)I_{3}(a_1,\ldots)
   \\  +[1+q^{-1}-a_1(\sum_{j\neq 1}a_j-qa_1)+\prod^6_{1}a_j(a_1\sum^5_{k\neq 1}a_{k}^{-1}-q^{-1}-(q+1)a_1^2)]I_{3}(qa_1,\ldots)
   \\  +q^{-1}(\prod^6_1 a_j-1)I_{3}(q^2a_1,\ldots) .
\label{M=3int_Recur:a}
\end{multline}
In addition the integral $ I_3(a_1,\ldots,a_6) $ satisfies a three-term recurrence in two variables, taken to be
with respect to $ a_5, a_6 $, which constitutes a pure recurrence in the deformation variable $ u $
\begin{multline} 
   0 = (a_5-qa_6)\prod^{4}_{j=1}(1-a_ja_6) I_{3}(\ldots,q^2a_5,a_6)
   \\ -(a_5-a_6)\left[ (1+q)(1+qa_5a_6\sigma_2+q^2a_5^2a_6^2\sigma_4)-(qa_5-a_6)(qa_6-a_5)(q+\sigma_4)
                \right.
   \\           \left. -q(a_5+a_6)(\sigma_1+qa_5a_6\sigma_3) \right] I_{3}(\ldots,qa_5,qa_6)
   \\ +(qa_5-a_6)\prod^{4}_{j=1}(1-a_ja_5) I_{3}(\ldots,a_5,q^2a_6) .
\label{M=3int_Recur:b}
\end{multline} 
\end{proposition}
\begin{proof}
The first recurrence can be read off from (\ref{Mrecur}) specialised to $ k=0 $, after evaluating the expansion coefficients
\begin{align*}
   \delta_{3,2}(a) &= \frac{\sigma_6-1}{q^{5/2}a^2} ,
  \\
   \delta_{3,1}(a) &= \frac{1}{q^{5/2}a^2}[qa(\sigma_5-\sigma_1)-(1+q)(1+qa^2)(\sigma_6-1)] ,
  \\
   \delta_{3,0}(a) &= \frac{1}{q^{3/2}a^2}[(1+a^2+a^4)(\sigma_6-1)-(a+a^3)(\sigma_5-\sigma_1)+a^2(\sigma_4-\sigma_2)] .
\end{align*}
The second recurrence follows from combining the first with the general identity
\begin{equation*}
   a_k I(qa_j,\ldots,a_k) -a_j I(a_j,\ldots,qa_k) = (a_k-a_j)(1-a_ja_k) I(a_j,\ldots,a_k) ,
\end{equation*}
which applies to any integral with products of $ \phi_{\infty}(x;a_j)\phi_{\infty}(x;a_k) $ in 
the denominator of the integrand and any distinct pair $ a_j \neq a_k $.
\end{proof}

\begin{remark}\label{Check}
A particular solution of (\ref{M=3int_Recur:b}) which serves as a concrete example of a 
moment sequence is one taken from \cite{Wi_2010b}
\begin{multline}
  m_{0,0}(t) = 
  t^{1/2}(q\alpha t,\alpha^{-1}t^{-1},q^{1/2}\alpha t,q^{1/2}\alpha^{-1}t^{-1};q)_{\infty}
         \frac{(a_1^{-1}\sigma_4\alpha t,a_2^{-1}\sigma_4\alpha t,a_3^{-1}\sigma_4\alpha t,a_4^{-1}\sigma_4\alpha t;q)_{\infty}}
              {(a_1\alpha t,a_2\alpha t,a_3\alpha t,a_4\alpha t,t^{-2},\sigma_4\alpha^2t^2;q)_{\infty}}
  \\ \times
   {}_{8}W_{7}(q^{-1}\sigma_4\alpha^2t^2;q^{-1}\sigma_4\alpha^2,a_1\alpha t,a_2\alpha t,a_3\alpha t,a_4\alpha t;q\alpha^{-2})
  + (t \mapsto t^{-1}) ,
\label{M=3AWmoment}
\end{multline}
although this is not a general solution of the $q$-difference equation, given subsequently
as (\ref{3TermMoment}).
\end{remark}

Henceforth 
we denote $ \tilde{\sigma}_j $ as the $j$-th elementary symmetric function of the parameters 
$ q^{-1/2}a_1, \dots, q^{-1/2}a_6 $. If $ \sigma_k $ denotes the $k$-th 
elementary symmetric function of $ a_1, a_2, a_3, a_4 $ then
\begin{align}
  \tilde{\sigma}_1 & = q^{-1/2}(\sigma_1+2\alpha u) ,
  \\
  \tilde{\sigma}_2 & = q^{-1}(\sigma_2+\alpha^2+2\alpha \sigma_1 u) ,
  \\
  \tilde{\sigma}_3 & = q^{-3/2}(\sigma_3+\alpha^2\sigma_1+2\alpha \sigma_2 u) ,
  \\
  \tilde{\sigma}_4 & = q^{-2}(\sigma_4+\alpha^2\sigma_2+2\alpha \sigma_3 u) ,
  \\
  \tilde{\sigma}_5 & = q^{-5/2}(\alpha^2\sigma_3+2\alpha \sigma_4 u) ,
  \\
  \tilde{\sigma}_6 & = q^{-3}\alpha^2\sigma_4 .
\end{align}

We will find it advantageous to define two "analogs" of the integers or half-integers $ s $ by
$ [s] := q^{s}\tilde{\sigma}_6-q^{-s} $ and $ \{s\} := q^{s}\tilde{\sigma}_6+q^{-s} $.
For convenience we employ the notations $ w_{\pm} $ for the evaluations
$ w_{\pm} \equiv W(\pm 1) 
 = \pm 1-\tilde{\sigma}_1\pm \tilde{\sigma}_2-\tilde{\sigma}_3\pm \tilde{\sigma}_4-\tilde{\sigma}_5\pm \tilde{\sigma}_6 $.
We also require the coordinates $ x_j $ and $ \tilde{x}_j $ defined in terms of the
parameters by
$ x_j=\tfrac{1}{2}(a_j+a_j^{-1}) $ and $ \tilde{x}_j=\tfrac{1}{2}(q^{-1/2}a_j+q^{1/2}a_j^{-1}) $.
Finally let us define the polynomial 
$ w(z) := \prod^4_{j=1}(1-q^{-1/2}a_j z) $ which is not to be confused with the weight $ w(x;u) $.

\begin{definition}
Note that we have six free parameters at our disposal, 
$ a_1, \ldots, a_4, \alpha \in \C, n\in\Z_{\geq 0} $ and one variable $ t $
subject henceforth to the following {\it generic conditions} - 
\begin{enumerate}
 \item
  $ q \neq 1 $, 
 \item
  $ |q^{-1/2}a_j| \neq 1 $ so that $ \tilde{x}_j \notin (-1,1) $ for $ j=1,\ldots, 4 $,
 \item
  $ |q^{-1/2}\alpha | \neq 1 $ so that $ q^{-1/2}\alpha+q^{1/2}\alpha^{-1} \notin (-1,1) $,
 \item
  $ \alpha \neq 0, \pm q^{1/2} $,
 \item
  $ t \neq \pm 1 $,
 \item
  $ [n], [n+\tfrac{1}{2}] \neq 0 $.
\end{enumerate}
Situations where one or more of the above conditions are violated have to be treated separately,
which we refrain from doing here.
\end{definition}

As a consequence of (\ref{deform_AW_SCff:a}) we have
\begin{gather}
  W(x) = 4(1+\tilde{\sigma}_6)x^3-2(\tilde{\sigma}_1+\tilde{\sigma}_5)x^2
        +(\tilde{\sigma}_2+\tilde{\sigma}_4-3-3\tilde{\sigma}_6)x
        +\tilde{\sigma}_1-\tilde{\sigma}_3+\tilde{\sigma}_5 ,
\label{deform_AW_SCff:b} \\
  V(x) = \frac{1}{q^{1/2}-q^{-1/2}}
        \left[ -4(1-\tilde{\sigma}_6)x^2
        +2(\tilde{\sigma}_1-\tilde{\sigma}_5)x+1-\tilde{\sigma}_2+\tilde{\sigma}_4-\tilde{\sigma}_6
        \right] ,
\label{deform_AW_SCff:c}
\end{gather}
The theory detailed in Section \ref{SpectralS} allow us to evaluate the third spectral
data polynomial.
\begin{proposition}[\cite{Wi_2010b}]\label{M=3L=1U}
The polynomial $ U(x) $ is given by
\begin{equation}
   U(x) = \frac{4}{(q^{1/2}-q^{-1/2})}\left[ m_{0,0}\left( -2[\tfrac{1}{2}]x+q^{1/2}\tilde{\sigma}_5-q^{-1/2}\tilde{\sigma}_1 \right)-[1](m_{0,+}+m_{0,-}) \right] ,
\label{M=3U}
\end{equation}
where $ m_{0,\pm} = \int \ddoAW x\,w(x;u)z^{\pm} $.
\end{proposition}
\begin{proof}
We start with
\begin{equation}
   \Delta yU(x) = (W-\Delta yV)f(y_{+})-(W+\Delta yV)f(y_{-}) ,
\label{aux:Z}
\end{equation}
and employ the proof in \cite{Wi_2010b}. Having achieved the task of expanding the right-hand side of (\ref{aux:Z}) in
canonical basis polynomials, all we require are explicit expressions for the coefficients
given below 
\begin{align}
  \kappa_{3,2}(a) & = \frac{1}{2q^{9/2}a^3}[(1+q+q^2)(1+q^2a^2)(1+\sigma_6)-q^2a(\sigma_1+\sigma_5)] ,
  \\
  \kappa_{3,3}(a) & = -\frac{1+\sigma_6}{2q^{9/2}a^3} .
\end{align}
The result is (\ref{M=3U}).
\end{proof}

\subsection{Spectral Structure}
Our next task is to construct a minimal parameterisation of the spectral matrix $ A_n(x;u) $
and is the subject of the following proposition.
We also recall the large $ x $ expansions for the spectral data and coefficients of a more explicit nature which will be
of importance
\begin{align} 
   W\pm \Delta yV & \sim \left[ w\pm (q^{1/2}-q^{-1/2})v \right] x^M + \ldots ,
\label{AWwgt_regular}\\
   W_n & \sim \tfrac{1}{4}\left[ (1+q^n)(1+q^{-n})w+(q^{1/2}-q^{-1/2})(q^n-q^{-n})v \right] x^M + \dots ,
\\
   \Theta_n & \sim \left[ \frac{q^{n+1/2}-q^{-n-1/2}}{q^{1/2}-q^{-1/2}}w
                           +(q^{n+1/2}+q^{-n-1/2})v \right] x^{M-2} + \ldots ,
\\
   \Omega_n+V & \sim \tfrac{1}{2} \left[ \frac{q^{n}-q^{-n}}{q^{1/2}-q^{-1/2}}w
                           +(q^{n}+q^{-n})v \right] x^{M-2} + \ldots .
\end{align} 
for some constants $ w,v $ independent of $ n,x $.

\begin{proposition}\label{SS_parameterise}
Let one of our free variables, $ \lambda_n(u) $, be the zero of $ \Theta_n(x;u) $ 
with respect to $ x $. 
Also define the other independent variables for $ n \geq 0 $ by
\begin{equation}
  \nu_n(u) \equiv 2W_n(\lambda_n;u)-W(\lambda_n;u) , \quad
  \mu_n(u) \equiv \Omega_n(\lambda_n;u)+V(\lambda_n;u) ,
\end{equation}
which are related by the spectral conic equation
\begin{equation}
   \nu_n^2-W^2(\lambda_n) = \Delta(\lambda_n^2-1)\left[ \mu_n^2-V^2(\lambda_n) \right] ,
\label{SPconic:a}
\end{equation}
again valid for $ n \geq 0 $.
Let us assume that $ \lambda_n(u) \neq \pm 1 $ for all $ n, u $.
The spectral coefficients for the $ M=3 $ deformed Askey-Wilson OP system are 
parameterised by $ \lambda_n $, and either $ \nu_n $ or $ \mu_n $, through the expressions
\begin{equation}
  2W_n(x;u)-W(x;u) 
  = \frac{x^2-1}{\lambda_n^2-1}\nu_n
  +(x-\lambda_n)\left[4\{n\}(x^2-1)+\tfrac{1}{2}w_{+}\frac{x+1}{1-\lambda_n}+\tfrac{1}{2}w_{-}\frac{x-1}{1+\lambda_n}
              \right] ,
\label{deform_AW_spec:a}
\end{equation}
\begin{multline}
 \Omega_n(x;u)+V(x;u) = \mu_n+4\frac{[n]}{q^{1/2}-q^{-1/2}}x(x-\lambda_n)
 +\frac{16}{q^{1/2}-q^{-1/2}}\frac{\tilde{\sigma}_6}{[n]}
   \left[ \frac{\tilde{\sigma}_1\tilde{\sigma}_6+\tilde{\sigma}_5}{4\tilde{\sigma}_6}-\lambda_n
   \right](x-\lambda_n)
 \\
 +\frac{1}{q^{1/2}-q^{-1/2}}\frac{\{n\}}{[n]}(x-\lambda_n)
  \left[ \frac{\nu_n}{\lambda_n^2-1}+\tfrac{1}{2}w_{+}\frac{1}{1-\lambda_n}+\frac{1}{2}w_{-}\frac{1}{1+\lambda_n}
  \right] ,
\label{deform_AW_spec:c}
\end{multline}
and
\begin{equation}
  \Theta_n(x;u) = 8\frac{[n+\tfrac{1}{2}]}{q^{1/2}-q^{-1/2}}(x-\lambda_n) ,
\label{deform_AW_spec:b}
\end{equation}
for $ n \geq 0 $.
\end{proposition}
\begin{proof}
Firstly we recall the spectral data polynomials given by (\ref{deform_AW_SCff:a},
\ref{deform_AW_SCff:b},\ref{deform_AW_SCff:c}).
The product $ W^2-\Delta y^2V^2 $ plays a significant role and therefore we
define another set of elementary symmetric polynomials by
\begin{equation}
  W^2-\Delta y^2V^2 = C_{\infty}\prod^{6}_{j=1}(x-\tilde{x}_j) 
   = C_{\infty}\left[ x^6-e_1x^5+e_2x^4-e_3x^3+e_4x^2-e_5x+e_6 \right]  ,
\end{equation}
where $ C_{\infty}=64\tilde{\sigma}_6 $.
We parameterise the spectral coefficients in the following way
\begin{gather}
   2W_n-W = w_3x^3+w_2x^2+w_1x+w_0 ,
 \\
  \Theta_n = \varpi_+(x-\lambda_n) ,\quad \Theta_{n-1} = \varpi_-(x-\lambda_{n-1}) ,
 \\
  \Omega_n+V = v_2x^2+v_1x+v_0 ,
\end{gather}
where the leading order coefficients in each are trivial and given by using (\ref{deform_AW_SCff:b})
and (\ref{deform_AW_SCff:c}) in the expansions (\ref{spectral_coeffEXP:a}), 
(\ref{spectral_coeffEXP:b}) and (\ref{spectral_coeffEXP:c})
\begin{equation}
  w_3 =  4\{n\},
 \quad
 \varpi_+ = 8\frac{[n+\tfrac{1}{2}]}{q^{1/2}-q^{-1/2}},
 \quad
  v_2 = 4\frac{[n]}{q^{1/2}-q^{-1/2}}.
\end{equation}
From the fundamental bi-linear relation (\ref{spectral_bilinear}) we get a
system of quadratic polynomial equalities
\begin{align}
   w_3^2-\Delta v_2^2 & = C_{\infty} ,
\label{deform_6} \\
   2w_2w_3-2\Delta v_1v_2 & = -C_{\infty} e_1 ,
\label{deform_5} \\
   2w_1w_3+w_2^2-\Delta(v_1^2-v_2^2+2v_0v_2)+a_n^2\Delta\varpi_+\varpi_- & = C_{\infty} e_2 ,
\label{deform_4} \\
   2w_1w_2+2w_0w_3-2\Delta(v_0v_1-v_1v_2)-a_n^2\Delta\varpi_+\varpi_-(\lambda_n+\lambda_{n-1}) & = -C_{\infty} e_3 ,
\label{deform_3} \\
   w_1^2+2w_0w_2-\Delta(v_0^2-v_1^2-2v_0v_2)+a_n^2\Delta\varpi_+\varpi_-(\lambda_n\lambda_{n-1}-1) & = C_{\infty} e_4 ,
\label{deform_2} \\
   2w_0w_1+2\Delta v_0v_1+a_n^2\Delta\varpi_+\varpi_-(\lambda_n+\lambda_{n-1}) & = -C_{\infty} e_5 ,
\label{deform_1} \\
   w_0^2+\Delta v_0^2-a_n^2\Delta\varpi_+\varpi_-\lambda_n\lambda_{n-1} & = C_{\infty} e_6 ,
\label{deform_0}
\end{align}
where $ \Delta $ is, again, defined as $ (q^{1/2}-q^{-1/2})^2 $.
Now (\ref{deform_0},\ref{deform_2},\ref{deform_4},\ref{deform_6}) imply 
$ (w_0+w_2)^2+(w_1+w_3)^2=C_{\infty}(1+e_2+e_4+e_6) $ while 
(\ref{deform_1},\ref{deform_3},\ref{deform_5})
imply $ 2(w_0+w_2)(w_1+w_3)=-C_{\infty}(e_1+e_3+e_5) $. Forming the sum and difference of these
two later relations we conclude that
\begin{equation}
   w_1+w_3 = \frac{\epsilon W(1)+\epsilon' W(-1)}{2}, \quad w_0+w_2 = \frac{\epsilon W(1)-\epsilon' W(-1)}{2} ,
\label{aux5}
\end{equation}
where $ \epsilon,\epsilon'=\pm 1 $ and are yet to be determined.
Using (\ref{aux5}) along with the above definition we can solve for $ w_0,w_2 $,
which only leaves the signs $ \epsilon,\epsilon' $ unresolved. These can be
fixed by requiring that the $ n=0 $ evaluation of $ 2W_n-W $ precisely
reproduces $ W $. This is identically true for all free parameters $ a_1,\ldots,a_6 $
provided $ \epsilon=+1,\epsilon'=-1 $, and yields (\ref{deform_w0},\ref{deform_w2}). 
The above coefficients are given by
\begin{align}
  w_2 & = \frac{\nu_n}{\lambda_n^2-1}-w_3\lambda_n-\tfrac{1}{2}W(1)\frac{1}{\lambda_n-1}+\tfrac{1}{2}W(-1)\frac{1}{\lambda_n+1} ,
\label{deform_w2} \\
  w_1 & = -w_3+\tfrac{1}{2}W(1)-\tfrac{1}{2}W(-1) ,
\label{deform_w1} \\
  w_0 & =-\frac{\nu_n}{\lambda_n^2-1}+w_3\lambda_n+\tfrac{1}{2}W(1)\frac{\lambda_n}{\lambda_n-1}+\tfrac{1}{2}W(-1)\frac{\lambda_n}{\lambda_n+1} .
\label{deform_w0}
\end{align}
We observe that this condition also gives $ \nu_0=W(\lambda_0) $, as it must.
Through knowledge of $ w_2 $ and utilising (\ref{deform_5}) we can determine $ v_1 $,
which is given in (\ref{deform_v1}). 
\begin{equation}
  \frac{\Delta v_2}{w_3}v_1 = \frac{\nu_n}{\lambda_n^2-1}-w_3\lambda_n+\tfrac{1}{2}\frac{C_{\infty}e_1}{w_3}
                            -\tfrac{1}{2}W(1)\frac{1}{\lambda_n-1}+\tfrac{1}{2}W(-1)\frac{1}{\lambda_n+1} ,
\label{deform_v1}
\end{equation}
However to find the remaining coefficient $ v_0 $ we require $ \mu_n $, and this 
result is given by (\ref{deform_v0}). 
\begin{equation}
  \frac{\Delta v_2}{w_3}v_0 = \frac{\Delta v_2}{w_3}\mu_n-\frac{\lambda_n}{\lambda_n^2-1}\nu_n+\frac{C_{\infty}}{w_3}\lambda_n^2
 -\tfrac{1}{2}\frac{C_{\infty}e_1}{w_3}\lambda_n+\tfrac{1}{2}W(1)\frac{\lambda_n}{\lambda_n-1}-\tfrac{1}{2}W(-1)\frac{\lambda_n}{\lambda_n+1} .
\label{deform_v0}
\end{equation}
When we examine $ \Omega_n+V $ at $ n=0 $ (recall that $ \Omega_0=0 $), we find in addition to
previously found relations the equality $ \mu_0=V(\lambda_0) $, again confirming our
definition. Therefore we have succeeded in relating the sub-leading coefficients explicitly
in terms of two independent variables.
\end{proof}

We now address the question of representations for the three-term recurrence coefficients.
We shall find that another set of variables, although equivalent to $ \nu_n $ and 
$ \mu_n $, will lead to the simplest forms for the relations we seek. 
\begin{align}
  w_{2,n} & := \frac{\nu_n}{\lambda_n^2-1}-4\{n\}\lambda_n
      -\tfrac{1}{2}W(1)\frac{1}{\lambda_n-1}+\tfrac{1}{2}W(-1)\frac{1}{\lambda_n+1} , 
\label{xfmV:a}  \\
  v_{0,n} & := (q^{1/2}-q^{-1/2})\mu_n-4[n]\lambda_n^2
                -\frac{1}{[n]}\lambda_n\left( \{n\}w_{2,n}+4(\tilde{\sigma}_1\tilde{\sigma}_6+\tilde{\sigma}_5) \right) .
\label{xfmV:b}
\end{align}
To achieve this outcome we require the following recurrence relations.
\begin{proposition}
The dynamical variables $ w_{2,n}(u) $ and $ v_{0,n}(u) $ satisfy the following system of first order 
coupled recurrence relations in $ n $
\begin{equation}
  w_{2,n+1}-w_{2,n} = -4(q^{1/2}-q^{-1/2})[n+\tfrac{1}{2}]\lambda_n , 
\label{deform_AW_recur:a}
\end{equation}
\begin{equation}
  v_{0,n+1}+v_{0,n}
   = -\frac{[n+\tfrac{1}{2}]}{[n+1][n]}
    \left( 2\{n+\tfrac{1}{2}\}\lambda_n w_{2,n}+8[n+\tfrac{1}{2}][n]\lambda_n^2
             +4(q^{1/2}+q^{-1/2})(\tilde{\sigma}_1\tilde{\sigma}_6+\tilde{\sigma}_5)\lambda_n \right) ,
\label{deform_AW_recur:b}
\end{equation}
and valid for $ n \geq 0 $ and subject to the initial conditions $ U(\lambda_0;u) = 0 $,
$ w_{2,0} = -2(\tilde{\sigma}_1+\tilde{\sigma}_5) $ and 
$ v_{0,0} = 1-\tilde{\sigma}_2+\tilde{\sigma}_4-\tilde{\sigma}_6 $.
\end{proposition}
\begin{proof}
For the first recurrence we employ our expressions (\ref{deform_AW_spec:a}) and (\ref{deform_AW_spec:b})
in (\ref{spectral_coeff_recur:a}) and equate the coefficients. The coefficients of $ x^3 $ and $ x $ are identically
satisfied whereas that of either $ x^2 $ or $ x^0 $ yields (\ref{deform_AW_recur:a}).
The second recurrence follows from the examination of the $ x^0 $ terms in (\ref{spectral_coeff_recur:b})
and the employment of the first recurrence. We note the terms in $ x^2 $ cancel identically, as they must.
\end{proof}

Explicit evaluations of the three term recurrence coefficients can be given in terms of this alternative set of variables.
\begin{proposition}
The three-term recurrence coefficients are found to be given by
\begin{multline}
 16[n+\tfrac{1}{2}][n]^2[n-\tfrac{1}{2}] a_n^2
  = \tilde{\sigma}_6 w_{2,n}^2 + 2(\tilde{\sigma}_1\tilde{\sigma}_6+\tilde{\sigma}_5)\{n\} w_{2,n} + 2[n]^3 v_{0,n}
\\
    + 2[n]^2\left( 4\tilde{\sigma}_6+2\tilde{\sigma}_4+2\tilde{\sigma}_2\tilde{\sigma}_6+2\tilde{\sigma}_1\tilde{\sigma}_5
                   -\{n\}(1+\tilde{\sigma}_2+\tilde{\sigma}_4+\tilde{\sigma}_6)+2[n]^2 \right)
    + 4(\tilde{\sigma}_1\tilde{\sigma}_6+\tilde{\sigma}_5)^2 ,
\label{deform_AW_spec:d}
\end{multline}
for $ n \geq 0 $ assuming $ a^2_0 = 0 $, and
\begin{equation}
  [n+1][n]b_n =
  -\tfrac{1}{4}\{n+\tfrac{1}{2}\} w_{2,n} -[n+\tfrac{1}{2}][n]\lambda_n
   -\tfrac{1}{2}(q^{1/2}+q^{-1/2})(\tilde{\sigma}_1\tilde{\sigma}_6+\tilde{\sigma}_5)
\label{deform_AW_spec:e}
\end{equation}
again valid for $ n \geq 0 $.
\end{proposition}
\begin{proof}
Using (\ref{deform_4}) we can solve for $ a_n^2 $. Observe that the right-hand side of
this equation is independent of $ n $ whereas individual terms on the left-hand side are. If
we assume $ a_0^2=0 $ then the right-hand side is equal to that of the left-hand side evaluated
at $ n=0 $, which we know because we can express it simply in terms of the parameters.  
This means $ a_n^2 $ is expressible as a sum of differences and after considerable simplification 
we arrive at the expression (\ref{deform_AW_spec:d}). To find (\ref{deform_AW_spec:e}) we start with 
(\ref{spectral_coeff_recur:b}) and use our previous results for (\ref{deform_AW_spec:b}) and 
(\ref{deform_AW_spec:c}).
Examining the terms in $ x $, we find an expression for $ b_n $ in terms of $ \lambda_n, \nu_n $
and $ \lambda_{n+1}, \nu_{n+1} $, or equivalently in terms of $ \lambda_n, w_{2,n} $ and
$ \lambda_{n+1}, w_{2,n+1} $. Using (\ref{deform_AW_recur:a}) we can eliminate the $ w_{2,n+1} $
term and after some factorisation we arrive at (\ref{deform_AW_spec:e}).
\end{proof}
We will not pursue the theory for the $ n \mapsto n+1 $ recurrences any further here but refer the reader
to \cite{Wi_2010c}. Consequently we will drop the $ n $ subscript from most variables until Subsection \ref{seed}.

The forgoing parameterisation of our system given in Proposition \ref{SS_parameterise} is not useful 
in the investigations of the $u$ or
$t$-evolution of our system and we require an alternative construction. In conformance with this we 
define the auxiliary variables $ l(t) $ and 
$ \mathfrak{z}_{\pm}(t) $ by
\begin{gather}
  \lambda := \frac{1}{2}(l+l^{-1}) ,
\\
  \mathfrak{z}_{\pm} := \nu\pm\frac{1}{2}(q^{1/2}-q^{-1/2})[l-l^{-1}]\mu ,
\end{gather}
where, in the first case, the inversion is given by the branch whereby $ l \to \infty $ when 
$ \lambda \to \infty $ and an identical choice is made for the second case.
\begin{proposition}
The spectral matrix elements have an alternative parameterisation
\begin{equation}
  \mathfrak{T}_{+}(z;t) = 2a_n[n+\tfrac{1}{2}](z-z^{-1})(z-l)(1-l^{-1}z^{-1}) , 
\label{auxB1}
\end{equation}
and
\begin{multline}
  [n]\mathfrak{W}_{+}(z;t) 
  = \mathfrak{z}_{+}\frac{(z-z^{-1})(z-l^{-1})(q^n\tilde{\sigma}_6-q^{-n}lz^{-1})}{(l-l^{-1})^2}
    +\mathfrak{z}_{-}\frac{(z-z^{-1})(z-l)(q^n\tilde{\sigma}_6-q^{-n}l^{-1}z^{-1})}{(l-l^{-1})^2}
  \\
    +[n](z-z^{-1})(z-l)(1-l^{-1}z^{-1})(q^n\tilde{\sigma}_6z-q^{-n}z^{-1})
  \\
    +[\tilde{\sigma}_6\tilde{\sigma}_1+\tilde{\sigma}_5-2\tilde{\sigma}_6(l+l^{-1})](z-z^{-1})(z-l)(1-l^{-1}z^{-1}) 
  \\
    +\tfrac{1}{2}w_{+}\frac{(z+1)(z-l)(z^{-1}-l)(q^{n}\tilde{\sigma}_6-q^{-n}z^{-1})}{(l-1)^2}
    -\tfrac{1}{2}w_{-}\frac{(z-1)(z-l)(z^{-1}-l)(q^{n}\tilde{\sigma}_6+q^{-n}z^{-1})}{(l+1)^2} ,
\label{auxB3}
\end{multline}
and
\begin{multline}
  [n]\mathfrak{W}_{-}(z;t)
  = \mathfrak{z}_{-}\frac{(z-z^{-1})(z-l^{-1})(q^n\tilde{\sigma}_6lz^{-1}-q^{-n})}{(l-l^{-1})^2}
    +\mathfrak{z}_{+}\frac{(z-z^{-1})(z-l)(q^n\tilde{\sigma}_6l^{-1}z^{-1}-q^{-n})}{(l-l^{-1})^2}
  \\
    -[n](z-z^{-1})(z-l)(1-l^{-1}z^{-1})(q^n\tilde{\sigma}_6z^{-1}-q^{-n}z)
  \\
    -[\tilde{\sigma}_6\tilde{\sigma}_1+\tilde{\sigma}_5-2\tilde{\sigma}_6(l+l^{-1})](z-z^{-1})(z-l)(1-l^{-1}z^{-1}) 
  \\
    +\tfrac{1}{2}w_{+}\frac{(z+1)(z-l)(z^{-1}-l)(q^{n}\tilde{\sigma}_6z^{-1}-q^{-n})}{(l-1)^2}
    +\tfrac{1}{2}w_{-}\frac{(z-1)(z-l)(z^{-1}-l)(q^{n}\tilde{\sigma}_6z^{-1}+q^{-n})}{(l+1)^2} .
\label{auxB4}
\end{multline}
\end{proposition}
\begin{proof}
These formulae can be viewed as Laurent interpolating polynomials satisfying the 
following evaluations at the given nodes
\begin{align} 
   \mathfrak{T}_{\pm}(z^{-1};t) & = -\mathfrak{T}_{\pm}(z;t) ,
  \\
   \mathfrak{T}_{+}(l^{\pm 1};t) & = 0 ,
  \\
   \mathfrak{T}_{+}(\pm 1;t) & = 0 ,
  \\
   \mathfrak{T}_{+}(z;t) & \underset{z \to \infty}{\sim} 2a_n[n+\tfrac{1}{2}]z^2 ,
  \\
   \mathfrak{T}_{+}(z;t) & \underset{z \to 0}{\sim} -2a_n[n+\tfrac{1}{2}]z^{-2} ,
\end{align} 
and
\begin{align} 
   \mathfrak{W}_{\pm}(z^{-1};t) & =  \mathfrak{W}_{\mp}(z;t) ,
  \\
   \mathfrak{W}_{\pm}(l;t) & = \mathfrak{z}_{\pm} ,
  \\
   \mathfrak{W}_{\pm}(\pm 1;t) & = w_{\pm} ,
  \\
   \mathfrak{W}_{+}(z;t) & \underset{z \to \infty}{\sim} q^n\tilde{\sigma}_6z^3 ,
  \\
   \mathfrak{W}_{+}(z;t) & \underset{z \to 0}{\sim} q^{-n}z^{-3} .
\end{align} 
The above formulae are also a consequence of the parameterisation 
(\ref{deform_AW_spec:a}-\ref{deform_AW_spec:c}).
\end{proof}

\begin{remark}
We note that the eigenvalues of $ A^{*}_n(z;t) $ as $ z \to 0, \infty $ are 
$ q^n\tilde{\sigma}_6, q^{-n} $ whilst those at the fixed points of the 
lattice, of the matrix $ A^{*}_n(\pm 1;t) $, are $ w_{\pm} $.
Under the mapping interchanging the interior and exterior of the unit circle
in the spectral variable $ z \mapsto 1/z $ we observe that
$ \mathfrak{T}_{\pm} \mapsto -\mathfrak{T}_{\pm} $ and
$ \mathfrak{W}_{\pm} \mapsto  \mathfrak{W}_{\mp} $.
As a consequence $  A^{*}_n(z;)A^{*}_n(z^{-1};) = (W^2-\Delta y^2V^2){\rm Id} $,
and therefore the mapping corresponds to a reversal of direction on the spectral
lattice. 
With respect to the mapping $ l \mapsto l^{-1} $ we note that this must be taken
with $ \mathfrak{z}_{+} \leftrightarrow \mathfrak{z}_{-} $, and conclude that
$ \mathfrak{T}_{+}, \mathfrak{W}_{\pm} $ are invariant under this type of 
transformation. In addition $ A^{*}_n $ is symmetrical under $ t \mapsto t^{-1} $.
\end{remark}
 
\subsection{Deformation Structure}
The deformation data polynomials $ R(x;u) $, $ S(x;u) $ are computed as
\begin{equation}
  R\pm \Delta v S 
    = (1-q^{-1/2}\alpha t^{\pm 1}z)(1-q^{-1/2}\alpha t^{\pm 1}z^{-1})
    = \phi_{1}(x;q^{-1/2}\alpha t^{\pm 1}) ,
\label{deform_AW_DCff:a}
\end{equation}
which implies
\begin{align}
  R & = 1-\alpha^2 q^{-1}-2\alpha q^{-1/2}xu+2\alpha^2q^{-1}u^2 ,
\label{deform_AW_DCff:b} \\
  S & = 2\frac{\alpha}{q-1}(\alpha q^{-1/2}u-x) ,
\label{deform_AW_DCff:c}
\end{align}
indicating that $ L=1 $. We observe that the polynomial
\begin{equation}
  R^2-\Delta v^2 S^2 = 4q^{-1}\alpha^2(x-\tilde{x}_5)(x-\tilde{x}_6) ,
\end{equation}
divides
\begin{equation}
  W^2-\Delta y^2 V^2 = 64q^{-3}\sigma_4\alpha^2\prod_{j=1}^{4}(x-\tilde{x}_j)\cdot(x-\tilde{x}_5)(x-\tilde{x}_6) ,
\end{equation}
in conformity with Remark \ref{common_zeros}. The role of the common zeros $ \tilde{x}_5,\tilde{x}_6 $ will be 
crucial in the ensuing investigations. We also note that these points
can be represented on a $ u $-lattice thus
$ \tilde{x}_5 = E_{u}^{-}x_5 $ and $ \tilde{x}_6 = E_{u}^{+}x_6 $. 
 
We can compute the remaining deformation data polynomial $ T(x;u) $, which will be
of degree zero.
\begin{proposition}[\cite{Wi_2010b}]\label{M=3L=1T}
The deformation data polynomial $ T(x;u) $ has the evaluation
\begin{equation}
    T(x;u) = \frac{4\alpha}{q-1}\frac{t\,m_{0,0}(v_{-})-t^{-1}m_{0,0}(v_{+})}{t-t^{-1}} .
\label{M=3Tpoly}
\end{equation}
\end{proposition} 
\begin{proof}
We start from the formula
\begin{equation}
    \Delta vT(x;u) = (R-\Delta vS)f(x;v_{+})-(R+\Delta vS)f(x;v_{-}) ,
\label{Teval}
\end{equation}
and employ the expansion (\ref{xLarge_SF:a}) for some parameter $ a $. Setting $ a $ to
one of $ a_1,a_2,a_3,a_4 $ and utilising the identity
\begin{equation*}
   \frac{\phi_{1}(x;b)}{\phi_{n+1}(x;a)} =
   \frac{(1-abq^n)(1-ba^{-1}q^{-n})}{\phi_{n+1}(x;a)}+\frac{ba^{-1}q^{-n}}{\phi_{n}(x;a)} ,
\end{equation*}
we find, by equating coefficients, the result (\ref{M=3Tpoly}). Note that $ m_{0,0} $
doesn't depend on $ a $. If one chooses $ a $ from $ a_5, a_6 $ then one requires the
additional identity
\begin{equation*}
   \frac{1}{\phi_{n}(x;qa)} =
   \frac{(1-q^n)(a^2-q^{-n})}{\phi_{n+1}(x;a)}+\frac{q^{-n}}{\phi_{n}(x;a)} ,
\end{equation*}
in order to merge the two series in (\ref{Teval}).
Alternatively one can compute $ T $ from the initial value (\ref{deform_coeff_init:c})
using the asymptotic expression for $ \Phi_0 $ as given by (\ref{deform_coeff_Exp:b}).
\end{proof}

As a check we can verify that the spectral data polynomials $ W,V,U $ and the deformation data
polynomials $ R,S,T $ satisfy all of the consistency relations formulated in Section \ref{DeformS}.
\begin{corollary}
The spectral data polynomials $ W,V,U $ and the deformation data polynomials $ R,S,T $ satisfy
the consistency relations (\ref{wgt_consist:b}) and (\ref{ST_consist}).
\end{corollary}
\begin{proof}
The first task involving $ W, V $ and $ R, S $ is elementary, whereas the second requires a moment 
relation
\begin{multline}
   [1](m_{0,+}+m_{0,-})(t) =
   \left[ q^{-2}\alpha^2\sigma_3-q^{-1}\sigma_1+q^{-3}(q-\alpha^2)(q^2\alpha^{-1}t^{-1}+\sigma_4\alpha t) \right]m_{0,0}(t)
   \\
   -\frac{w(q^{-1/2}\alpha t)}{\alpha t}\frac{q^{-1/2}t\,m_{0,0}(q^{-1}t)-q^{1/2}t^{-1}m_{0,0}(t)}{q^{-1/2}t-q^{1/2}t^{-1}} ,
\end{multline}
and the moment $q$-difference equation in $ t $
\begin{multline}
   (q^{-1}t-qt^{-1})w(\alpha t^{-1})m_{0,0}(q^{1/2}t)
  \\
   -q^{-1/2}(q^{-1/2}t-q^{1/2}t^{-1})
    \left[ (1+q)(1+q^{-1}\alpha^2\sigma_{2}+q^{-2}\alpha^4\sigma_{4})
               +q^{-1}\alpha^2(q+\sigma_{4})(t-t^{-1})(q^{-1}t-qt^{-1})
    \right.
  \\
    \left.
                   -\alpha(q^{-1/2}t+q^{1/2}t^{-1})(\sigma_{1}+q^{-1}\alpha^2\sigma_{3})
    \right]m_{0,0}(q^{-1/2}t)
  \\
   +(t-t^{-1})w(q^{-1}\alpha t)m_{0,0}(q^{-3/2}t) = 0 ,
\label{3TermMoment}
\end{multline}
which is just (\ref{M=3int_Recur:b}) with $ a_5 \mapsto q^{-3/2}\alpha t, a_6 \mapsto q^{-1/2}\alpha t^{-1} $.
\end{proof}

We know for $ L=1 $ that the deformation coefficients can be parameterised thus
\begin{align}
 \mathfrak{R}_{\pm} & = r_{1\pm}x+r_{0\pm} ,
\label{Dparam:a} \\
 \mathfrak{P}_{\pm} & = p_{\pm} .
\label{Dparam:b}
\end{align}
Furthermore we know from the large $ x $ expansions (\ref{deform_coeff_Exp:a}-\ref{deform_coeff_Exp:e})
what the leading terms are, however in contrast to those of the spectral coefficients
these are related to the three-term recurrence coefficients $ a_n, \gamma_n $ in a nontrivial way.
\begin{corollary}
The leading order terms of the elements of the matrix $ B^{*}_n $ are given by
\begin{align}
  r_{1+} & = 4\alpha q^{-1/2}H_n\frac{t\gamma_n(v_+)}{\gamma_{n-1}(v_-)} ,
\label{deform_AW_DCff:d} \\
  r_{1-} & = 4\alpha q^{-1/2}H_n\frac{\gamma_n(v_-)}{t\gamma_{n-1}(v_+)} ,
\label{deform_AW_DCff:e} \\
  p_{+} & = 4\alpha q^{-1/2}H_n \left[ \frac{t\gamma_n(v_+)}{\gamma_n(v_-)}-\frac{\gamma_n(v_-)}{t\gamma_n(v_+)} \right] ,
\label{deform_AW_DCff:f} \\
  p_{-} & = 4\alpha q^{-1/2}\frac{a_nH_{n-1}}{a_{n-1}}
             \left[ \frac{t\gamma_{n-1}(v_+)}{\gamma_{n-1}(v_-)}-\frac{\gamma_{n-1}(v_-)}{t\gamma_{n-1}(v_+)} \right] ,
\label{deform_AW_DCff:g}
\end{align}
for $ n \geq 1 $ whilst for $ n=0 $ one can use $ \gamma_{-1}=a_0\gamma_0 $.
\end{corollary}
Up to this point we haven't exercised a choice regarding the decoupling factor $ H_n $ but one can take
$ H_n=\tfrac{1}{2}a_n $ henceforth.

\begin{remark}
We note that the eigenvalues of $ B^{*}_n(z;t) $ as $ z \to 0, \infty $ are $ r_{1\pm} $. 
By construction $ B^{*}_n $ is symmetrical with respect to $ z \mapsto z^{-1} $.
Under the mapping interchanging the interior and exterior of the unit circle
in the deformation variable $ t \mapsto 1/t $ we observe 
(see (\ref{deform_def:a}), (\ref{deform_def:b}) and (\ref{deform_def:c})) that
$ \mathfrak{P}_{\pm} \mapsto -\mathfrak{P}_{\pm} $ and
$ \mathfrak{R}_{\pm} \mapsto  \mathfrak{R}_{\mp} $.
As a consequence
\begin{equation}
   B^{*}_n(;t)B^{*}_n(;t^{-1}) = \frac{4H_n^2}{a_n(v_{-})a_n(v_{+})}(R^2-\Delta v^2S^2){\rm Id},
\end{equation}
and therefore this mapping corresponds to a reversal of direction on the deformation
lattice. 
\end{remark}

As part of an efficient methodology we will require formulae which relate the 
lower order terms, with respect to $ x $, of the deformation coefficients to the leading order term. 
\begin{proposition}\label{AB_resolve_z}
The trailing terms of the deformation coefficients are related to each other and to the spectral 
parameterisation by
\begin{align}
   p_{+} & = r_{1+}a_n(v_-)-r_{1-}a_n(v_+) ,
\label{Dclose:a}  \\
   p_{-} & = r_{1+}a_n(v_+)-r_{1-}a_n(v_-) ,
\label{Dclose:b}  \\
   r_{0+} & = \frac{r_{1+}}{\Delta v_2}\left[ -w_2(v_+)+w_2(v_-)
                           -2\sqrt{\Delta}(q^{-1/2}\alpha+q^{1/2}\alpha^{-1})(\tilde{\sigma}_6q^nt-q^{-n}t^{-1}) \right] ,
\label{Dclose:c}  \\  
   r_{0-} & = \frac{r_{1-}}{\Delta v_2}\left[ w_2(v_+)-w_2(v_-)
                           -2\sqrt{\Delta}(q^{-1/2}\alpha+q^{1/2}\alpha^{-1})(\tilde{\sigma}_6q^nt^{-1}-q^{-n}t) \right] ,
\label{Dclose:d}  
\end{align}
\end{proposition}
\begin{proof}
The formulae given above are ones which involve both shifts up and down on the $u$-lattice
of the spectral variables, 
and these are easily derived by resolving the compatibility relation (\ref{spectral+deform:b}) 
as a Laurent
polynomial in $ z $. Such a polynomial is constructed in our application by substituting 
the elements (\ref{auxB1}-\ref{auxB4}) and (\ref{Dparam:a},\ref{Dparam:b}) into (\ref{spectral+deform:b})
and collecting terms. For the $ 1,2 $ component we find that the leading order non-zero
contributions occur at $ z^{\pm 7} $ and the coefficients of both these terms will vanish
if (\ref{Dclose:a}) holds. In the case of the $ 2,1 $ component the $ z^{\pm 7} $ terms are
the leading non-trivial ones and these both vanish when (\ref{Dclose:b}) holds.
The two latter relations (\ref{Dclose:c},\ref{Dclose:d}) follow from the requirement that the
coefficients of $ z^7 $ terms of the $ 1,1 $ and $ 2,2 $ elements vanish respectively.  
Finally we observe that an independent way of verifying (\ref{Dclose:a}) and (\ref{Dclose:b})
is through a trivial combination of the formulae (\ref{deform_AW_DCff:d}-\ref{deform_AW_DCff:g})
\end{proof}

A vital part of our analysis consists of resolving the compatibility relations (\ref{S+D:a}-\ref{S+D:d}) with
respect to the zeros and poles of $ \chi(z;t) $.
\begin{corollary}
The residues of the compatibility relation given in Proposition \ref{spectral+deform} 
consist of the following equations:
\begin{itemize}
\item[(i)]
At the advanced co-ordinate $ z=z_5=\alpha t $
\begin{align}
   r_{1-}E^{-}_{u}x_5+r_{0-} & = -\frac{\mathfrak{W}_{+}(z_5;q^{1/2}t)}{\mathfrak{T}_{+}(z_5;q^{1/2}t)}p_{+} ,
\label{cse:a} \\
   r_{1+}E^{-}_{u}x_5+r_{0+} & = -\frac{\mathfrak{W}_{-}(z_5;q^{1/2}t)}{\mathfrak{T}_{-}(z_5;q^{1/2}t)}p_{-} ,
\label{cse:i}
\end{align}
\item[(ii)]
at the advanced co-ordinate $ z=E^{2+}_{u}z_6=q^{-1}\alpha t^{-1} $
\begin{align}
   r_{1-}E^{+}_{u}x_6+r_{0-} & = \frac{\mathfrak{W}_{-}(E^{2+}_{u}z_6;q^{1/2}t)}{\mathfrak{T}_{+}(E^{2+}_{u}z_6;q^{1/2}t)}p_{+} ,
\label{cse:d} \\
   r_{1+}E^{+}_{u}x_6+r_{0+} & = \frac{\mathfrak{W}_{+}(E^{2+}_{u}z_6;q^{1/2}t)}{\mathfrak{T}_{-}(E^{2+}_{u}z_6;q^{1/2}t)}p_{-} ,
\label{cse:l}
\end{align}
\item[(iii)]
at the retarded co-ordinate $ z=E^{2-}_{u}z_5=q^{-1}\alpha t $
\begin{align}
   r_{1+}E^{-}_{u}x_5+r_{0+} & = -\frac{\mathfrak{W}_{-}(E^{2-}_{u}z_5;q^{-1/2}t)}{\mathfrak{T}_{+}(E^{2-}_{u}z_5;q^{-1/2}t)}p_{+} ,
\label{cse:f} \\
   r_{1-}E^{-}_{u}x_5+r_{0-} & = -\frac{\mathfrak{W}_{+}(E^{2-}_{u}z_5;q^{-1/2}t)}{\mathfrak{T}_{-}(E^{2-}_{u}z_5;q^{-1/2}t)}p_{-} ,
\label{cse:n}
\end{align}
\item[(iv)]
and at the retarded co-ordinate $ z=z_6=\alpha t^{-1} $
\begin{align}
   r_{1+}E^{+}_{u}x_6+r_{0+} & = \frac{\mathfrak{W}_{+}(z_6;q^{-1/2}t)}{\mathfrak{T}_{+}(z_6;q^{-1/2}t)}p_{+} ,
\label{cse:g} \\
   r_{1-}E^{+}_{u}x_6+r_{0-} & = \frac{\mathfrak{W}_{-}(z_6;q^{-1/2}t)}{\mathfrak{T}_{-}(z_6;q^{-1/2}t)}p_{-} .
\label{cse:o}
\end{align}
\end{itemize}
All the deformation parameters are evaluated at $ t,u $, i.e. 
$ p_{\pm}(t), r_{1\pm}(t), r_{0\pm}(t) $.
\end{corollary}
\begin{proof}
The multiplicative factor (\ref{twist}) is given, in this case, by
\begin{equation}
   \chi(z;t) = \frac{(1-q^{-1}\alpha t^{-1}z)(1-\alpha tz^{-1})}{(1-q^{-1}\alpha tz^{-1})(1-\alpha t^{-1}z)} .
\end{equation}
At the zeros of $ \chi $ the right-hand sides of (\ref{S+D:a},\ref{S+D:b},\ref{S+D:c},\ref{S+D:d})
must vanish whereas at the poles the left-hand factors must vanish.
Thus for the zero $ z=z_5=\alpha t $ applied to (\ref{S+D:a}) we deduce (\ref{cse:a}). For the other zero at
$ z=1/E^{2+}_u z_6=q \alpha^{-1}t $ applied to (\ref{S+D:a}) we make additional use of the symmetries of
$ A^{*}_n $ and $ B^{*}_n $ with respect to inversion of $ z $ to derive (\ref{cse:d}).
At the pole $ z=E^{2-}_u z_5=q^{-1}\alpha t $ the residue of (\ref{S+D:a}) yields (\ref{cse:f})
whereas at the other pole $ z=1/z_6=\alpha^{-1}t $ we have to apply the additional symmetries to the residue 
to arrive at (\ref{cse:g}). 
However we have further relations, which are equivalent to the four relations derived above, 
by specialising the spectral variable in these ways due to the fact that 
$ (W^2-\Delta y^2V^2)(z=\alpha t,q^{-1}\alpha t^{-1};q^{1/2}t)=0 $,
$ (W^2-\Delta y^2V^2)(z=q^{-1}\alpha t,\alpha t^{-1};q^{-1/2}t)=0 $,
$ (R^2-\Delta v^2S^2)(z=q^{-1/2}\alpha t;t)=0 $ and
$ (R^2-\Delta v^2S^2)(z=q^{-1/2}\alpha t^{-1};t)=0 $, which implies
\begin{gather}
   (\mathfrak{W}_{+}\mathfrak{W}_{-}+\mathfrak{T}_{+}\mathfrak{T}_{-})(z=\alpha t,q^{-1}\alpha t^{-1};q^{1/2}t) = 0,
  \\
   (\mathfrak{W}_{+}\mathfrak{W}_{-}+\mathfrak{T}_{+}\mathfrak{T}_{-})(z=q^{-1}\alpha t,\alpha t^{-1};q^{-1/2}t) = 0,
  \\
   (\mathfrak{R}_{+}\mathfrak{R}_{-}+\mathfrak{P}_{+}\mathfrak{P}_{-})(z=q^{-1/2}\alpha t;t) = 0,
  \\
   (\mathfrak{R}_{+}\mathfrak{R}_{-}+\mathfrak{P}_{+}\mathfrak{P}_{-})(z=q^{-1/2}\alpha t^{-1};t) = 0 .
\end{gather}
This means that each of the identities (\ref{cse:a}, \ref{cse:d}, \ref{cse:f}, \ref{cse:g})
can take four forms and the remaining relations (\ref{cse:i}, \ref{cse:l}, \ref{cse:n}, \ref{cse:o}) 
are examples of just one of those forms. This equivalence also ensures that if the residue condition 
is satisfied by (\ref{S+D:a}) then it is automatically satisfied by (\ref{S+D:b},\ref{S+D:c},\ref{S+D:d})
as well.
\end{proof}

\begin{corollary}
The following product relation holds amongst the spectral coefficients evaluated
at the zeros of the spectral determinant
\begin{equation}
   \frac{\mathfrak{W}_{+}(z_5;q^{1/2}t)}{\mathfrak{T}_{+}(z_5;q^{1/2}t)}
   \frac{\mathfrak{W}_{-}(E^{2-}_{u}z_5;q^{-1/2}t)}{\mathfrak{T}_{+}(E^{2-}_{u}z_5;q^{-1/2}t)}
 =  \frac{\mathfrak{W}_{+}(z_6;q^{-1/2}t)}{\mathfrak{T}_{+}(z_6;q^{-1/2}t)}
    \frac{\mathfrak{W}_{-}(E^{2+}_{u}z_6;q^{1/2}t)}{\mathfrak{T}_{+}(E^{2+}_{u}z_6;q^{1/2}t)} .
\label{prodId}
\end{equation}
\end{corollary}
\begin{proof}
By comparing the right-hand sides of the pairs, (\ref{cse:a}, \ref{cse:n}), (\ref{cse:d},\ref{cse:o}),
(\ref{cse:f}, \ref{cse:i}) and (\ref{cse:g}, \ref{cse:l}), we see that
\begin{align}
   \frac{\mathfrak{W}_{+}(z_5;q^{1/2}t)}{\mathfrak{T}_{+}(z_5;q^{1/2}t)}p_{+}
  & = \frac{\mathfrak{W}_{+}(E^{2-}_{u}z_5;q^{-1/2}t)}{\mathfrak{T}_{-}(E^{2-}_{u}z_5;q^{-1/2}t)}p_{-} ,
\label{eSME:a}  \\
   \frac{\mathfrak{W}_{-}(E^{2+}_{u}z_6;q^{1/2}t)}{\mathfrak{T}_{+}(E^{2+}_{u}z_6;q^{1/2}t)}p_{+}
  & = \frac{\mathfrak{W}_{-}(z_6;q^{-1/2}t)}{\mathfrak{T}_{-}(z_6;q^{-1/2}t)}p_{-} ,
\label{eSME:b}  \\
   \frac{\mathfrak{W}_{-}(E^{2-}_{u}z_5;q^{-1/2}t)}{\mathfrak{T}_{+}(E^{2-}_{u}z_5;q^{-1/2}t)}p_{+}
  & = \frac{\mathfrak{W}_{-}(z_5;q^{1/2}t)}{\mathfrak{T}_{-}(z_5;q^{1/2}t)}p_{-} ,
\label{eSME:c}  \\
   \frac{\mathfrak{W}_{+}(z_6;q^{-1/2}t)}{\mathfrak{T}_{+}(z_6;q^{-1/2}t)}p_{+}
  & = \frac{\mathfrak{W}_{+}(E^{2+}_{u}z_6;q^{1/2}t)}{\mathfrak{T}_{-}(E^{2+}_{u}z_6;q^{1/2}t)}p_{-} .
\label{eSME:d}
\end{align}
By forming the cross products of the first and second, or equivalently the third 
and fourth we can establish (\ref{prodId}), assuming $ p_{+}p_{-} \neq 0 $.
Taking cross products of the first and third, or second and fourth leads to 
the identities that are trivial consequences of the determinantal relation.
\end{proof}

\begin{definition}
In conformity with the remark made immediately preceeding Corollary \ref{Bn+1}
let us define the variables with deformation arguments evaluated at the advanced and
retarded co-ordinates through the notation $ \hat{l} = l(q^{1/2}t) $, $ \check{l} = l(q^{-1/2}t) $
etc. 
\end{definition}

\begin{lemma}
One solution for the component $ p_{+} $ is given by
\begin{multline}
  q^{1/2}\alpha^{-1}(t^{-1}-t)[n+\tfrac{1}{2}]\frac{r_{1-}\hat{a}_n}{p_{+}}
  = -\frac{qt\hat{l}}{(\hat{l}-\hat{l}^{-1})^2}
     \left[
         \frac{\hat{\mathfrak{z}}_{+}}{(\alpha t-\hat{l})(qt-\alpha\hat{l})}
              +\frac{\hat{\mathfrak{z}}_{-}}{(-\alpha+qt\hat{l})(-1+\alpha t\hat{l})}
     \right]
  \\
     -\frac{1}{\alpha  t}(q^{-n}+q^{n+1}t^2\sigma_6)+\tfrac{1}{4}w(1)\frac{1}{\hat{\lambda}-1}+\tfrac{1}{4}w(-1)\frac{1}{\hat{\lambda}+1} ,
\label{soln:up}
\end{multline}
A second solution for the component $ p_{+} $ is
\begin{multline}
  q^{1/2}\alpha^{-1}(t^{-1}-t)[n+\tfrac{1}{2}]\frac{r_{1+}\check{a}_n}{p_{+}}
  =  \frac{qt\check{l}}{(\check{l}-\check{l}^{-1})^2}
     \left[
         \frac{\check{\mathfrak{z}}_{+}}{(q-\alpha t\check{l})(-\alpha+t\check{l})}
              +\frac{\check{\mathfrak{z}}_{-}}{(t-\alpha\check{l})(-\alpha t+q\check{l})}
     \right]
  \\
     -\frac{1}{\alpha t}(q^{-n}t^2+q^{n+1}\sigma_6)+\tfrac{1}{4}w(1)\frac{1}{\check{\lambda}-1}+\tfrac{1}{4}w(-1)\frac{1}{\check{\lambda}+1} .
\label{soln:dn}
\end{multline}
\end{lemma}
\begin{proof}
The first relation (\ref{soln:up}) follows from the subtraction of (\ref{cse:d}) from 
(\ref{cse:a}) and simplifying. The second relation (\ref{soln:dn}) follows from similar 
reasoning applied to (\ref{cse:f}) and (\ref{cse:g}).
\end{proof}
\begin{remark}
A profound difference between the expressions (\ref{soln:up},\ref{soln:dn}) and the corresponding
formulae in the case of the $q$-linear spectral lattice is the presence of both $ \hat{\mathfrak{z}}_{\pm} $
(or $ \check{\mathfrak{z}}_{\pm} $) whereas in the latter case only one these variables in present, 
which significantly complicates the ensuing analysis.
\end{remark}

\begin{definition}
We define the following sequence of co-ordinate transformations
\begin{align}
   \mathfrak{z}_{+} & = l^{-2}w(l)(1-q^{-1/2}\alpha t l^{-1})(1-q^{-1/2}\alpha t^{-1}l) 
                        \frac{lf-1}{f-l} ,
\label{splitXfm:a}
  \\
   \mathfrak{z}_{-} & = l^{2}w(l^{-1})(1-q^{-1/2}\alpha t l)(1-q^{-1/2}\alpha t^{-1}l^{-1})
                        \frac{f-l}{lf-1} , 
\label{splitXfm:b}
\end{align}
where the new variable $ f(t) $ is introduced.
Clearly it is apparent that under the mapping $ l \mapsto l^{-1} $ we have the interchange
$ \mathfrak{z}_{+} \leftrightarrow \mathfrak{z}_{-} $ and that the product satisfies 
the relation (\ref{SPconic:a}), i.e. $ \mathfrak{z}_{+}\mathfrak{z}_{-}=\prod^{6}_{j=1}(1-q^{-1/2}a_j l^{\pm}) $. 
Let $ \sigma_4 = q^2 s_{4}^2 $.
In addition we define the further variable $ \rho(t) $ so that
\begin{align}
   \frac{r_{1+}\check{a}_n}{p_{+}} 
  & := \frac{1}{q^{1/2}[n+\tfrac{1}{2}]}\frac{2\alpha s_4\hat{\rho}-q^{-n}t-q^{n}\alpha^2 s^2_4t^{-1}}{t-t^{-1}} ,
\label{rPxfm} \\
   \frac{r_{1-}\hat{a}_n}{p_{+}} 
  & := \frac{1}{q^{1/2}[n+\tfrac{1}{2}]}\frac{2\alpha s_4\hat{\rho}-q^{n}\alpha^2 s^2_4t-q^{-n}t^{-1}}{t-t^{-1}} .
\label{rMxfm}
\end{align}
Together these relations ensure that (\ref{Dclose:a}) is automatically satisfied.
We further define $ g(t) $ by $ 2\rho := g+g^{-1} $. We shall find that $ f,l,g $
will be our primary variables.
\end{definition}

A consequence of these definitions and the previous Lemma is the following transformation
formula.
\begin{corollary}
Let us assume $ f \neq 0,\infty $, $ f \neq l^{\pm 1} $ and $ w(f) \neq 0 $.
There exists an invertible mapping between $ \rho(t) $ and $ \lambda(t) $ with the
form
\begin{equation}
   -2s_4f\rho+1+s^2_4 f^2 = \frac{w(f)}{f^{2}+1-2f\lambda} .
\label{glXFM}
\end{equation}
\end{corollary}
\begin{proof}
Employing the definitions (\ref{rMxfm}), (\ref{splitXfm:a}) and (\ref{splitXfm:b}) in (\ref{soln:up}) 
we compute the relation
\begin{equation}
  2s_4\rho = q^{-3/2}\sigma_3-s^2_4[l+l^{-1}]+\frac{w(l)}{(l^{2}-1)(l-f)}+\frac{l^{3}w(l^{-1})}{(l^{2}-1)(lf-1)} ,
\label{LR_xfm}
\end{equation}
where we have shifted the arguments $ t\mapsto q^{-1/2}t $ of all variables.
The relation (\ref{glXFM}) then follows by performing a partial fraction expansion of the above
equation with respect to the variable $ l(t) $ (the former equation can be thought of as an
expansion with respect to $ f(t) $). We also have an alternative form of the above expression
\begin{equation}
   \lambda = q^{1/2}\frac{\sigma_3}{2\sigma_4}-\frac{1}{2s_4}(g+g^{-1})
             -\frac{1}{2(g^2-1)}\frac{w(s^{-1}_4g)}{f-s^{-1}_4g}+\frac{g^2}{2(g^2-1)}\frac{w(s^{-1}_4g^{-1})}{f-s^{-1}_4g^{-1}} .
\label{glIXFM}
\end{equation}
\end{proof}
\begin{remark}
What is surprising about this result in comparison to the situation with the $q$-linear lattice is that 
instead of a linear transformation between $ \rho $ and $ f $ we find one between $ \rho $ and $ \lambda $
with $ f $ being an intermediary. Given (\ref{LR_xfm}) it is clear there is no linear inversion possible for $ f $.  
\end{remark}

Our next task is to compute expressions for the four independent evaluations of the spectral matrix
elements given on the left-hand sides of (\ref{eSME:a}-\ref{eSME:d}) in terms of the variables introduced above.
We will carry this out in stages of successive refinement.
\begin{lemma}
The ratios of the evaluated spectral matrix elements satisfy the relations 
\begin{equation}
  \frac{\hat{a}_n}{\check{a}_n} 
  \frac{2\alpha s_4\hat{\rho}-q^{-n}t-q^{n}\alpha^2 s^2_4t^{-1}}{2\alpha s_4\hat{\rho}-q^{n}\alpha^2 s^2_4t-q^{-n}t^{-1}}
  = \frac{\mathfrak{W}_{+}(z_6;q^{-1/2}t)}{\mathfrak{T}_{+}(z_6;q^{-1/2}t)}
    \frac{\mathfrak{T}_{+}(z_5;q^{1/2}t)}{\mathfrak{W}_{+}(z_5;q^{1/2}t)}
  = \frac{\mathfrak{W}_{-}(E^{2-}_{u}z_5;q^{-1/2}t)}{\mathfrak{T}_{+}(E^{2-}_{u}z_5;q^{-1/2}t)}
    \frac{\mathfrak{T}_{+}(E^{2+}_{u}z_6;q^{1/2}t)}{\mathfrak{W}_{-}(E^{2+}_{u}z_6;q^{1/2}t)} .
\label{splitId}
\end{equation}
\end{lemma}
\begin{proof}
Subtracting (\ref{cse:a}) from (\ref{cse:d}) gives
\begin{equation}
    r_{1-}\left( E^{+}_{u}x_6-E^{-}_{u}x_5 \right) 
  = \left( \frac{\mathfrak{W}_{-}(E^{2+}_{u}z_6;q^{1/2}t)}{\mathfrak{T}_{+}(E^{2+}_{u}z_6;q^{1/2}t)}
          +\frac{\mathfrak{W}_{+}(z_5;q^{1/2}t)}{\mathfrak{T}_{+}(z_5;q^{1/2}t)} \right) p_{+} ,
\end{equation}
whilst subtracting (\ref{cse:f}) from (\ref{cse:g}) gives
\begin{equation}
    r_{1+}\left( E^{+}_{u}x_6-E^{-}_{u}x_5 \right) 
  = \left( \frac{\mathfrak{W}_{+}(z_6;q^{-1/2}t)}{\mathfrak{T}_{+}(z_6;q^{-1/2}t)}
          +\frac{\mathfrak{W}_{-}(E^{2-}_{u}z_5;q^{-1/2}t)}{\mathfrak{T}_{+}(E^{2-}_{u}z_5;q^{-1/2}t)} \right) p_{+} .
\end{equation}
Taking ratios of these two gives
\begin{multline}
    \frac{r_{1+}}{r_{1-}}
  = \frac{\mathfrak{T}_{+}(z_5;q^{1/2}t)\mathfrak{T}_{+}(E^{2+}_{u}z_6;q^{1/2}t)}
         {\mathfrak{T}_{+}(z_6;q^{-1/2}t)\mathfrak{T}_{+}(E^{2-}_{u}z_5;q^{-1/2}t)}
  \\ \times
    \frac{ \mathfrak{T}_{+}(z_6;q^{-1/2}t)\mathfrak{W}_{-}(E^{2-}_{u}z_5;q^{-1/2}t)
          +\mathfrak{T}_{+}(E^{2-}_{u}z_5;q^{-1/2}t)\mathfrak{W}_{+}(z_6;q^{-1/2}t)}
         { \mathfrak{T}_{+}(z_5;q^{1/2}t)\mathfrak{W}_{-}(E^{2+}_{u}z_6;q^{1/2}t)
          +\mathfrak{T}_{+}(E^{2+}_{u}z_6;q^{1/2}t)\mathfrak{W}_{+}(z_5;q^{1/2}t)} .
\label{auxR1}
\end{multline}
Likewise subtracting (\ref{cse:i}) from (\ref{cse:l}) yields
\begin{equation}
    r_{1+}\left( E^{+}_{u}x_6-E^{-}_{u}x_5 \right) 
  = -\left( \frac{\mathfrak{T}_{+}(E^{2+}_{u}z_6;q^{1/2}t)}{\mathfrak{W}_{-}(E^{2+}_{u}z_6;q^{1/2}t)}
           +\frac{\mathfrak{T}_{+}(z_5;q^{1/2}t)}{\mathfrak{W}_{+}(z_5;q^{1/2}t)} \right) p_{-} ,
\end{equation}
and (\ref{cse:n}) from (\ref{cse:o})
\begin{equation}
    r_{1-}\left( E^{+}_{u}x_6-E^{-}_{u}x_5 \right) 
  = -\left( \frac{\mathfrak{T}_{+}(z_6;q^{-1/2}t)}{\mathfrak{W}_{+}(z_6;q^{-1/2}t)}
           +\frac{\mathfrak{T}_{+}(E^{2-}_{u}z_5;q^{-1/2}t)}{\mathfrak{W}_{-}(E^{2-}_{u}z_5;q^{-1/2}t)} \right) p_{-} .
\end{equation}
Taking the ratio of these latter two equations gives
\begin{multline}
    \frac{r_{1+}}{r_{1-}}
  = \frac{\mathfrak{W}_{+}(z_6;q^{-1/2}t)\mathfrak{W}_{-}(E^{2-}_{u}z_5;q^{-1/2}t)}
         {\mathfrak{W}_{+}(z_5;q^{1/2}t)\mathfrak{W}_{-}(E^{2+}_{u}z_6;q^{1/2}t)}
  \\ \times
    \frac{ \mathfrak{T}_{+}(z_5;q^{1/2}t)\mathfrak{W}_{-}(E^{2+}_{u}z_6;q^{1/2}t)
          +\mathfrak{T}_{+}(E^{2+}_{u}z_6;q^{1/2}t)\mathfrak{W}_{+}(z_5;q^{1/2}t)}
         { \mathfrak{T}_{+}(z_6;q^{-1/2}t)\mathfrak{W}_{-}(E^{2-}_{u}z_5;q^{-1/2}t)
          +\mathfrak{T}_{+}(E^{2-}_{u}z_5;q^{-1/2}t)\mathfrak{W}_{+}(z_6;q^{-1/2}t)} .
\label{auxR2}
\end{multline}  
The factors appearing on the right-hand sides of these two ratios constructed have
simple evaluations, the first factor being
\begin{equation}
    \frac{\mathfrak{T}_{+}(z_5;q^{1/2}t)\mathfrak{T}_{+}(E^{2+}_{u}z_6;q^{1/2}t)}
         {\mathfrak{T}_{+}(z_6;q^{-1/2}t)\mathfrak{T}_{+}(E^{2-}_{u}z_5;q^{-1/2}t)}
  = \frac{\hat{a}^2_n}{\check{a}^2_n}
    \frac{(q^2t^2-\alpha^2)(\alpha^2t^2-1)(\alpha t-\hat{l})(\alpha t-\hat{l}^{-1})(qt-\alpha\hat{l}^{-1})(qt-\alpha\hat{l})}
         {(t^2-\alpha^2)(\alpha^2t^2-q^2)(\alpha t-q\check{l})(\alpha t-q\check{l}^{-1})(t-\alpha\check{l}^{-1})(t-\alpha\check{l})} .
\end{equation}
Employing this evaluation in the first ratio (\ref{auxR1}) yields the evaluation
for the second factor
\begin{multline}
  \frac{ 
         \mathfrak{T}_{+}(z_6;q^{-1/2}t)\mathfrak{W}_{-}(E^{2-}_{u}z_5;q^{-1/2}t)
        +\mathfrak{T}_{+}(E^{2-}_{u}z_5;q^{-1/2}t)\mathfrak{W}_{+}(z_6;q^{-1/2}t) }
       {
         \mathfrak{T}_{+}(z_5;q^{1/2}t)\mathfrak{W}_{-}(E^{2+}_{u}z_6;q^{1/2}t)
        +\mathfrak{T}_{+}(E^{2+}_{u}z_6;q^{1/2}t)\mathfrak{W}_{+}(z_5;q^{1/2}t) }
 \\
 =   \frac{\check{a}_n}{\hat{a}_n}
     \frac{(t^2-\alpha^2)(\alpha^2t^2-q^2)}{(q^2t^2-\alpha^2)(\alpha^2t^2-1)}
     \frac{(\alpha t-q\check{l})(\alpha t-q\check{l}^{-1})(t-\alpha\check{l}^{-1})(t-\alpha\check{l})}
          {(\alpha t-\hat{l})(\alpha t-\hat{l}^{-1})(qt-\alpha\hat{l}^{-1})(qt-\alpha\hat{l})}
     \frac{2\alpha s_4\hat{\rho}-q^{-n}t-q^{n}\alpha^2 s^2_4t^{-1}}{2\alpha s_4\hat{\rho}-q^{n}\alpha^2 s^2_4t-q^{-n}t^{-1}} .
\end{multline}
Employing this latter factor in the second ratio (\ref{auxR2}) we can construct two 
relations involving perfect squares, upon using (\ref{prodId}). After taking the square 
roots we need to resolve the sign ambiguity. The final result is then (\ref{splitId}).
\end{proof}

\begin{remark}
In fact we can make separate evaluations of the numerator and denominator of the last ratio,
which gives us
\begin{multline}
  \mathfrak{T}_{+}(z_5;q^{1/2}t)\mathfrak{W}_{-}(E^{2+}_{u}z_6;q^{1/2}t)
 +\mathfrak{T}_{+}(E^{2+}_{u}z_6;q^{1/2}t)\mathfrak{W}_{+}(z_5;q^{1/2}t) \\
 = \frac{2\hat{a}_n[n+\tfrac{1}{2}]}{q^{3}\alpha^5t^4}
   (q^2t^2-\alpha^2)(\alpha^2t^2-1)(\alpha^2-q)
   (\alpha t-\hat{l})(\alpha t-\hat{l}^{-1})(qt-\alpha\hat{l})(qt-\alpha\hat{l}^{-1})
   \left[ 2\alpha s_4\hat{\rho}-q^{n}\alpha^2 s^2_4t-q^{-n}t^{-1} \right] ,
\end{multline}
and
\begin{multline}
  \mathfrak{T}_{+}(z_6;q^{-1/2}t)\mathfrak{W}_{-}(E^{2-}_{u}z_5;q^{-1/2}t)
 +\mathfrak{T}_{+}(E^{2-}_{u}z_5;q^{-1/2}t)\mathfrak{W}_{+}(z_6;q^{-1/2}t) \\
 = \frac{2\check{a}_n[n+\tfrac{1}{2}]}{q^{3}\alpha^5t^4}
   (t^2-\alpha^2)(\alpha^2t^2-q^2)(\alpha^2-q)
   (\alpha t-q\check{l})(\alpha t-q\check{l}^{-1})(t-\alpha\check{l})(t-\alpha\check{l}^{-1})
   \left[ 2\alpha s_4\hat{\rho}-q^{-n}t-q^{n}\alpha^2 s^2_4t^{-1} \right] .
\label{aux11}
\end{multline}
\end{remark}

In order to proceed any further we will require representations of the individual
evaluated spectral coefficients which we give as two distinct partial-fraction expansions.
We offer these without proofs as they are the outcome of straightforward although lengthy computations.

\begin{lemma}
The evaluated spectral coefficients possess the rational function forms in the $ f,l $ variables
\begin{multline}
   2t[n][n+\tfrac{1}{2}]\hat{a}_n\frac{\mathfrak{W}_{+}(z_5;q^{1/2}t)}{\mathfrak{T}_{+}(z_5;q^{1/2}t)}
  \\
  = -q^{-1}\hat{f}^{-2}(qt-\alpha\hat{f})(q^{n}\alpha t\tilde{\sigma}_6-q^{-n}\hat{f})
     \frac{w(\hat{f})}{(\hat{f}-\hat{l})(\hat{f}-\hat{l}^{-1})}
    +\alpha^{-1}\tilde{\sigma}_6\hat{f}^{-1}(\alpha t-q^{-n}\hat{f})(q^{n}\alpha t-\hat{f})(\hat{l}+\hat{l}^{-1})
  \\
    +\alpha^{-1}\hat{f}^{-2}(q^{n}\alpha t-\hat{f})\big[
                -q^{-n}\tilde{\sigma}_6(\hat{f}^2-1)(q^n\alpha t+\hat{f})-q^{-1/2}\alpha t \tilde{\sigma}_6\sigma_1\hat{f}
  \\
                +q^{-n-5/2}\alpha^2\sigma_3\hat{f}^2+(q^{-n+2}+q^{n-2}\alpha^4\sigma_4)\hat{f}(-q^{-n-2}+q^{-4}\alpha t\sigma_4\hat{f}) \big] ,
\label{evalSpec:a}
\end{multline}
\begin{multline}
   2t[n][n+\tfrac{1}{2}]\check{a}_n\frac{\mathfrak{W}_{+}(z_6;q^{-1/2}t)}{\mathfrak{T}_{+}(z_6;q^{-1/2}t)}
  \\
  = -q^{-1}\check{f}^{-2}\frac{(t-\alpha\check{f})(\alpha t-q\check{f})(q^{n}\alpha \tilde{\sigma}_6-q^{-n}t\check{f})}{(\alpha-t\check{f})}
     \frac{w(\check{f})}{(\check{f}-\check{l})(\check{f}-\check{l}^{-1})}
  \\
    +\alpha^{-1}\tilde{\sigma}_6\check{f}^{-1}(\alpha t-q^{-n+1}\check{f})(q^{n-1}\alpha t-\check{f})(\check{l}+\check{l}^{-1})
  \\
    -\frac{[n](q-t^{2})t^{3}}{q}\frac{(t-\alpha\check{f})}{(\alpha-t\check{f})}\frac{w(\alpha t^{-1})}{(\alpha-t\check{l})(\alpha-t\check{l}^{-1})}
  \\
    +\alpha^{-1}t^{-1}\check{f}^{-2}(q^{n-1}\alpha t-\check{f})\big[ 
                -q^{-n}t\tilde{\sigma}_6(\check{f}^2-1)(q^n\alpha t+q\check{f})-q^{-n}(q-t^2)\tilde{\sigma}_6\check{f}(q^n t+\alpha\check{f})
  \\
                +q^{-2-2n}t^2\check{f}(-q^2 t+q^n\alpha \sigma_4\check{f})-q^{-1/2}\alpha t^2\tilde{\sigma}_6 \sigma_1\check{f}+q^{-\frac{3}{2}-n}\alpha^2 t\sigma_3\check{f}^2
                -\alpha^2\tilde{\sigma}_6\check{f}(t-q^{n-2}\alpha \sigma_4\check{f}) \big] ,
\label{evalSpec:b}
\end{multline} 
\begin{multline}
   2t[n][n+\tfrac{1}{2}]\hat{a}_n\frac{\mathfrak{W}_{-}(E^{2+}_{u}z_6;q^{1/2}t)}{\mathfrak{T}_{+}(E^{2+}_{u}z_6;q^{1/2}t)}
  \\
  =  q^{-1/2}\hat{f}^{-2}(\alpha t-\hat{f})(q^{n+1/2}t\tilde{\sigma}_6-q^{-n-1/2}\alpha\hat{f})
     \frac{w(\hat{f})}{(\hat{f}-\hat{l})(\hat{f}-\hat{l}^{-1})}
    -\alpha^{-1}\tilde{\sigma}_6\hat{f}^{-1}(\alpha t-q^{-n}\hat{f})(q^{n}\alpha t-\hat{f})(\hat{l}+\hat{l}^{-1})
  \\
    +\hat{f}^{-2}(q^{n}\alpha t-\hat{f})\big[
                q^{-3-n}(q^n-1)\alpha \sigma_4\hat{f}(1+\alpha t\hat{f})+q^{-1/2}t\tilde{\sigma}_6\sigma_1\hat{f}-q^{-n-5/2}\alpha \sigma_3\hat{f}^2
  \\
               -(q^{n+1/2}t\tilde{\sigma}_6-q^{-n-1/2}\alpha\hat{f})(q^{-n-1/2}+q^{-5/2}\sigma_4\hat{f}^2) \big] ,
\label{evalSpec:c}
\end{multline} 
\begin{multline}
   2t[n][n+\tfrac{1}{2}]\check{a}_n\frac{\mathfrak{W}_{-}(E^{2-}_{u}z_5;q^{-1/2}t)}{\mathfrak{T}_{+}(E^{2-}_{u}z_5;q^{-1/2}t)}
  \\
  =  q^{-1/2}\check{f}^{-2}\frac{(t-\alpha\check{f})(\alpha t-q\check{f})(q^{n+1/2}\tilde{\sigma}_6-q^{-n-1/2}\alpha t\check{f})}{(q-\alpha t\check{f})}
     \frac{w(\check{f})}{(\check{f}-\check{l})(\check{f}-\check{l}^{-1})}
  \\
    -\alpha^{-1}\tilde{\sigma}_6\check{f}^{-1}(\alpha t-q^{-n+1}\check{f})(q^{n-1}\alpha t-\check{f})(\check{l}+\check{l}^{-1})
  \\
    +\frac{[n](q-t^{2})\alpha^4t^{3}}{q^2}\frac{(\alpha t-q\check{f})}{(q-\alpha t\check{f})}\frac{w(q\alpha^{-1}t^{-1})}{(q-\alpha t\check{l})(q-\alpha t\check{l}^{-1})}
  \\
    +t^{-1}\check{f}^{-2}(q^{n-1}\alpha t-\check{f})\big[ 
               q^{1-n}\alpha^{-1}t\tilde{\sigma}_6(\check{f}^2-1)(q^{n-1}\alpha t+\check{f})+q^{-4-n}(q-t^2)\sigma_4\check{f}(q^n\alpha^3 t+q^2\check{f})
  \\
              +q^{-1/2}t^2\tilde{\sigma}_6\sigma_1\check{f}-q^{-n-3/2}\alpha t\sigma_3\check{f}^2-q^{-4-2n}\alpha(qt^2+q^{2 n}\sigma_4)\check{f}(-q^2t+q^n\alpha\sigma_4\check{f})             
               \big] .
\label{evalSpec:d}
\end{multline}
\end{lemma}

In addition the evaluations at the advanced co-ordinate have alternative forms.
\begin{lemma}
The advanced evaluated spectral coefficients are rational functions in the $ g,f $ variables
\begin{multline}
   2t[n][n+\tfrac{1}{2}]\hat{a}_n\frac{\mathfrak{W}_{+}(z_5;q^{1/2}t)}{\mathfrak{T}_{+}(z_5;q^{1/2}t)}
  \\
  =-q^{-n-1}\alpha s_4
    \frac{\left(q^n t\alpha s_4\hat{g}^{-1}-1\right)^2}{\hat{g}-\hat{g}^{-1}}
    \frac{w({s_4}^{-1}\hat{g})}{\hat{f}-{s_4}^{-1}\hat{g}}
   +q^{-n-1}\alpha s_4
    \frac{\left(q^n t\alpha s_4\hat{g}-1\right)^2}{\hat{g}-\hat{g}^{-1}} 
    \frac{w({s_4}^{-1}\hat{g}^{-1})}{\hat{f}-{s_4}^{-1}\hat{g}^{-1}}
  \\
   -q^{-n-1}\alpha\left(\hat{g}^2+\hat{g}^{-2}\right)
   -q^{-n-6}{s_4}^{-1}\left[-q^{7/2}\alpha \sigma_3+t\sigma_4(q^4-2q^{n+3}\alpha^2+q^{2n}\alpha^4 \sigma_4)\right]\left(\hat{g}+\hat{g}^{-1}\right)
  \\
   +q^{-2n-6}\alpha^{-1}\left[q^6+t^2q^{4n}\alpha^6 \sigma_4^2-t^2q^{3n+1}\alpha^4\sigma_4(q^2+\sigma_4)
                               +q^{2n+2}\alpha^2(t^2q^2+\alpha^2)\sigma_4
                        \right.
  \\
                        \left.-q^{n+4}\alpha^2(2q+\sigma_2)
                                 +q^{2n+5/2}\alpha^3 t(-q\sigma_3+\sigma_1\sigma_4)\right] ,
\label{lhs2ND:a}
\end{multline}
and
\begin{multline}
   2t[n][n+\tfrac{1}{2}]\hat{a}_n\frac{\mathfrak{W}_{-}(E^{2+}_{u}z_6;q^{1/2}t)}{\mathfrak{T}_{+}(E^{2+}_{u}z_6;q^{1/2}t)}
  \\
  = q^{-n-1}\alpha s_4
    \frac{\left(q^n t\alpha s_4\hat{g}^{-1}-1\right)^2}{\hat{g}-\hat{g}^{-1}}
    \frac{w({s_4}^{-1}\hat{g})}{\hat{f}-{s_4}^{-1}\hat{g}}
   -q^{-n-1}\alpha s_4
    \frac{\left(q^n t\alpha s_4\hat{g}-1\right)^2}{\hat{g}-\hat{g}^{-1}} 
    \frac{w({s_4}^{-1}\hat{g}^{-1})}{\hat{f}-{s_4}^{-1}\hat{g}^{-1}}
  \\
   +q^{-n-1}\alpha\left(\hat{g}^2+\hat{g}^{-2}\right)
   -q^{-n-15/2}\alpha{s_4}^{-1}\left[q^{5}\sigma_3-t\alpha\sigma_4(q^{9/2}-2q^{n+9/2}+q^{2n+5/2}\sigma_4)\right]\left(\hat{g}+\hat{g}^{-1}\right)
  \\
   +q^{-2n-11/2}\alpha\left[-q^{9/2}-q^{4n+1/2}\alpha^2 t^2\sigma_4^2+q^{3n+1/2}\alpha^2t^2\sigma_4(q^2+\sigma_4)
                               -q^{2n+5/2}(1+\alpha^2t^2)\sigma_4
                        \right.
  \\
                        \left.+q^{n+7/2}(2q+\sigma_2)
                                 -q^{2n+2}\alpha t(-q\sigma_3+\sigma_1\sigma_4)\right] .
\label{lhs2ND:b}
\end{multline}
\end{lemma}

At this stage we have accumulated enough results to deduce the dynamical equations for our system on the 
deformation lattice. 
\begin{proposition}\label{evolution_1}
The $t$-evolution of our system in the variables $ \rho, \lambda, f $ is given by a pair of coupled 
first order equations the first of which is
\begin{multline}
  \frac{2q s_4\hat{\rho}-t^{2}\check{f}^{-1}-q^2t^{-2}s^2_4\check{f}}{2s_4\check{\rho}-\check{f}^{-1}-s^2_4\check{f}}
  = \frac{q(t-\alpha \check{f})(q\check{f}-\alpha t)}{(t\check{f}-\alpha)(q-\alpha t\check{f})}
  \\
    +\frac{\alpha^{4}t^{2}(q-t^{2})}{(q-\alpha^{2})}\frac{(q\check{f}-\alpha t)}{(q-\alpha t\check{f})}
     \frac{\check{f}^{2}w(q\alpha^{-1}t^{-1})}
          {q\alpha tw(\check{f})-(q\check{f}-\alpha t)(q-\alpha t\check{f})\check{f}[2s_4\check{\rho}-\check{f}^{-1}-s^2_4\check{f}]}
  \\
    -\frac{qt^{2}(q-t^{2})}{(q-\alpha^{2})}\frac{(t-\alpha \check{f})}{(t\check{f}-\alpha)}
     \frac{\check{f}^{2}w(\alpha t^{-1})}
          {\alpha tw(\check{f})-(t\check{f}-\alpha)(t-\alpha \check{f})\check{f}[2s_4\check{\rho}-\check{f}^{-1}-s^2_4\check{f}]} .
\label{t-Evol:a}
\end{multline}
The auxiliary equation for the leading coefficient of the polynomial is
\begin{equation}
    t^2 \frac{\hat{\gamma}^2_n}{\check{\gamma}^2_n}
    = \frac{t^2+q^{2n}\alpha^2s^2_4-2q^n\alpha ts_4\hat{\rho}}{1+q^{2n}\alpha^2t^2s^2_4-2q^n\alpha ts_4\hat{\rho}} .
\label{t-Evol:f}
\end{equation}
\end{proposition}
\begin{proof}
The first of these equations is derived utilising the following steps.
We substitute the expression for $ r_{1+}\check{a}_n/p_{+} $ given by (\ref{rPxfm})
and the transformation formulae for $ \check{\mathfrak{z}}_{\pm} $ as given in
(\ref{splitXfm:a},\ref{splitXfm:b}) into the solution at the retarded co-ordinate
(\ref{soln:dn}). This yields an equation involving $ \hat{\rho} $ on one hand
and $ \check{l},\check{f} $ on the other hand. Now we perform a partial fraction expansion
of this with respect to $ \check{l} $ and this produces the following expression
\begin{multline}
   2s_4\hat{\rho} = q^{-1}t^{2}\check{f}^{-1}+qt^{-2}s^2_4\check{f}
   -\frac{(t-\alpha \check{f})(q\check{f}-\alpha t)}{\check{f}(t\check{f}-\alpha)(q-\alpha t\check{f})}
    \frac{w(\check{f})}{\check{f}^{2}+1-2\check{f}\check{\lambda}} 
  \\
   -\frac{\alpha^{4}t^{2}(q-t^{2})(q\check{f}-\alpha t)}{q(q-\alpha^{2})(q-\alpha t\check{f})}
    \frac{w(q\alpha^{-1}t^{-1})}{q^{2}+\alpha^{2}t^{2}-2q\alpha t\check{\lambda}} 
   +\frac{t^{2}(q-t^{2})(t-\alpha \check{f})}{(q-\alpha^{2})(t\check{f}-\alpha)}
    \frac{w(\alpha t^{-1})}{\alpha^{2}+t^{2}-2\alpha t\check{\lambda}} ,
\label{Evol_aux}
\end{multline}
which is a simple function of $ \check{\lambda} $. Then one substitutes for $ \check{\lambda} $
using the inversion of the transformation (\ref{glXFM}) at the retarded time and the
result is (\ref{t-Evol:a}). Alternatively one can prove this formula by substituting (\ref{evalSpec:b})
and (\ref{evalSpec:d}) into (\ref{aux11}).

We employ the variable transformation (\ref{rPxfm}) and note that the left-hand side 
of this expression has been evaluated in (\ref{deform_AW_DCff:d}) and (\ref{deform_AW_DCff:f}). 
Equating these two forms gives (\ref{t-Evol:f}).
\end{proof}

\begin{remark}\label{Quin2Quad}
The right-hand side of formula (\ref{t-Evol:a}) exhibits apparent poles at $ \alpha=t\check{f} $
and $ q=\alpha t\check{f} $. However this is not the case as the former is cancelled by opposing
contributions from the first and third terms, while the latter is cancelled by contributions
from the first and second term.
\end{remark}

Clearly (\ref{t-Evol:a}) is not manifestly invertible for $ \check{\rho} $, however
it is possible to construct a linear equation for the retarded variables by switching from
$ \rho $ to $ \lambda $.
\begin{proposition}\label{evolution_2}
The inverse to (\ref{t-Evol:a}) is given as a relation for $ \check{\lambda} $,
which we give in two alternative forms
\begin{multline}
   \check{\lambda} =
    \frac{(\hat{f}-q^n\alpha t)(t-q^n\alpha s^2_4\hat{f})}
         {(t\hat{f}-q^n\alpha)(1-q^n\alpha ts^2_4\hat{f})}\left[ \hat{\lambda}-\tfrac{1}{2}(\hat{f}+\hat{f}^{-1}) \right]
   +\frac{t^4+\hat{f}^2}{2t^2\hat{f}}
   \\
   +\frac{t^2(t^2-1)s_4^3}{2q^{n+1/2}[n+\tfrac{1}{2}]}
    \Bigg[
       -s_4^3\frac{(\hat{f}-q^n\alpha t)}{(1-q^n\alpha ts^2_4\hat{f})}
        \frac{q^{4n}\alpha^4w(q^{-n}\alpha^{-1}t^{-1}s_4^{-2})}{1+q^{2n}\alpha^2t^2s^2_4-2q^n\alpha ts_4\hat{\rho}}
       +s_4^{-3}\frac{(t-q^n\alpha s^2_4\hat{f})}{(t\hat{f}-q^n\alpha)}
        \frac{w(q^{n}\alpha t^{-1})}{t^2+q^{2n}\alpha^2s^2_4-2q^n\alpha ts_4\hat{\rho}}
    \Bigg] ,
\label{t-Evol:d}
\end{multline}
or
\begin{multline}
   \check{\lambda} =
   -\frac{(\hat{g}-q^n\alpha ts_4)(t-q^n\alpha s_4\hat{g})}
         {2(\hat{g}^2-1)(1-q^n\alpha ts_4\hat{g})(t\hat{g}-q^n\alpha s_4)}
    \frac{w(s_4^{-1}\hat{g})}{\hat{f}-s_4^{-1}\hat{g}}
   +\frac{\hat{g}^2(1-q^n\alpha ts_4\hat{g})(t\hat{g}-q^n\alpha s_4)}
         {2(\hat{g}^2-1)(\hat{g}-q^n\alpha ts_4)(t-q^n\alpha s_4\hat{g})}
    \frac{w(s_4^{-1}\hat{g}^{-1})}{\hat{f}-s_4^{-1}\hat{g}^{-1}}
   \\
   +\frac{1}{2\alpha t^3\sigma_4}
    \left[ q^{1/2}\alpha t\sigma_3+q^{5/2}(t^2-1)\{n+\tfrac{1}{2}\} \right]
   -\frac{1}{2t^2s_4}\left[ \hat{g}+\hat{g}^{-1} \right]
   \\
   +\frac{\alpha^2t(t^2-1)s_4^3}{2q^{1/2}[n+\tfrac{1}{2}]}
    \Bigg[
          q^{2n}\alpha s_4
          \frac{w(q^{-n}\alpha^{-1}t^{-1}s_4^{-2})}{(1-q^n\alpha ts_4\hat{g})(1-q^n\alpha ts_4\hat{g}^{-1})}
         -\alpha^{-1}s_4^{-1}
          \frac{w(q^{n}\alpha t^{-1})}{(t-q^n\alpha s_4\hat{g})(t-q^n\alpha s_4\hat{g}^{-1})}
    \Bigg] .
\label{t-Evol:c}
\end{multline}
\end{proposition}
\begin{proof}
One can solve the compatibility relation (\ref{spectral+deform:b}) for the $(1,2)$ 
component of $ A^*_n(z;q^{-1/2}t) $ and find $ \mathfrak{T}_{+}(z;q^{-1/2}t) $
\begin{multline} 
  \mathfrak{T}_{+}(z;q^{-1/2}t) = \frac{1}{\chi(z,t)}\frac{\check{a}_n\hat{a}_n}{4H^2_n(R+\Delta vS)(q^{1/2}z;t)(R-\Delta vS)(q^{1/2}z;t)}
  \\
  \times \left[ 
                \mathfrak{R}_{-}(q^{1/2}z;t)\mathfrak{R}_{-}(q^{-1/2}z;t)\mathfrak{T}_{+}(z;q^{1/2}t) 
               +\mathfrak{P}_{+}(q^{1/2}z;t)\mathfrak{P}_{+}(q^{-1/2}z;t)\mathfrak{T}_{-}(z;q^{1/2}t) 
         \right.
  \\     \left.
               +\mathfrak{R}_{-}(q^{1/2}z;t)\mathfrak{P}_{+}(q^{-1/2}z;t)\mathfrak{W}_{+}(z;q^{1/2}t) 
               -\mathfrak{P}_{+}(q^{1/2}z;t)\mathfrak{R}_{-}(q^{-1/2}z;t)\mathfrak{W}_{-}(z;q^{1/2}t) 
         \right] .
\label{aux10}
\end{multline} 
To simplify the calculations of the spectral matrix elements
we employ an alternative parameterisation to that of (\ref{auxB1}), (\ref{auxB3}), (\ref{auxB4}) 
\begin{multline}
  \mathfrak{W}_{+}(z;q^{1/2}t) 
  = -\frac{\alpha t(z^2-1)(qt-\alpha z)(q^{2n+1}t\tilde{\sigma}_6z-\alpha)}{z^2(1-\alpha^2t^2)(q-\alpha^2)(q^{2n+1}t^2\tilde{\sigma}_6-1)}
     \mathfrak{W}_{+}(\alpha t;q^{1/2}t) 
    \\
    +\frac{q^2\alpha t(z^2-1)(z-\alpha t)(q^{2n}\alpha t\tilde{\sigma}_6z-1)}{z^2(q^2t^2-\alpha^2)(q-\alpha^2)(q^{2n+1}t^2\tilde{\sigma}_6-1)}
     \mathfrak{W}_{-}(q^{-1}\alpha t^{-1};q^{1/2}t) 
    \\
    +\frac{(z+1)(z-\alpha t)(qt-\alpha z)(q^2-q^{2n}\alpha^2t^2\sigma_4z)}{2qtz^2(q^2-q^{2n}\alpha^2t^2\sigma_4)}w(1)
    \\
    -\frac{(z-1)(z-\alpha t)(qt-\alpha z)(q^2+q^{2n}\alpha^2t^2\sigma_4z)}{2qtz^2(q^2-q^{2n}\alpha^2t^2\sigma_4)}w(-1)
    \\
    +\frac{q^{-n-3}(z^2-1)(z-\alpha t)(qt-\alpha z)}{\alpha t^2z^3(q^2-q^{2n}\alpha^2t^2\sigma_4)}
     \left[ q^4+q^{4n}\alpha^4t^4\sigma^2_4z^2+q^{2n+3/2}\alpha^2t^2(q\sigma_3z+\sigma_1\sigma_4z-q^{1/2}(1+z^2)) \right] ,
\end{multline}
\begin{multline}
  \mathfrak{W}_{-}(z;q^{1/2}t) 
  =  \frac{\alpha t(z^2-1)(qtz-\alpha)(q^{2n+1}t\tilde{\sigma}_6-\alpha z)}{z^2(1-\alpha^2t^2)(q-\alpha^2)(q^{2n+1}t^2\tilde{\sigma}_6-1)}
     \mathfrak{W}_{+}(\alpha t;q^{1/2}t) 
    \\
    +\frac{q^2\alpha t(z^2-1)(1-\alpha tz)(z-q^{2n}\alpha t\tilde{\sigma}_6)}{z^2(q^2t^2-\alpha^2)(q-\alpha^2)(q^{2n+1}t^2\tilde{\sigma}_6-1)}
     \mathfrak{W}_{-}(q^{-1}\alpha t^{-1};q^{1/2}t) 
    \\
    +\frac{(z+1)(1-\alpha tz)(qtz-\alpha)(q^2z-q^{2n}\alpha^2t^2\sigma_4)}{2qtz^2(q^2-q^{2n}\alpha^2t^2\sigma_4)}w(1)
    \\
    +\frac{(z-1)(1-\alpha tz)(qtz-\alpha)(q^2z+q^{2n}\alpha^2t^2\sigma_4)}{2qtz^2(q^2-q^{2n}\alpha^2t^2\sigma_4)}w(-1)
    \\
    -\frac{q^{-n-3}(z^2-1)(1-\alpha tz)(qtz-\alpha)}{\alpha t^2z^3(q^2-q^{2n}\alpha^2t^2\sigma_4)}
     \left[ q^4z^2+q^{4n}\alpha^4t^4\sigma^2_4+q^{2n+3/2}\alpha^2t^2(q\sigma_3z+\sigma_1\sigma_4z-q^{1/2}(1+z^2)) \right] ,
\end{multline}
and
\begin{multline}
  \mathfrak{T}_{+}(z;q^{1/2}t) = -\frac{\alpha t(z-z^{-1})}{(q-\alpha^2)(qt^2-1)}p_{+}
    \\
     \times\left[
     \frac{(qtz-\alpha)(qtz^{-1}-\alpha)}{(\alpha^2t^2-1)}
     \frac{\mathfrak{W}_{+}(\alpha t;q^{1/2}t)}{r_{1,-}\tfrac{1}{2}(q^{-1/2}\alpha t+q^{1/2}\alpha^{-1}t^{-1})+r_{0,-}} 
     \right.
    \\
     \left.
    +q^2\frac{(z-\alpha t)(z^{-1}-\alpha t)}{(\alpha^2-q^2t^2)}
     \frac{\mathfrak{W}_{-}(q^{-1}\alpha t^{-1};q^{1/2}t)}{r_{1,-}\tfrac{1}{2}(q^{-1/2}\alpha t^{-1}+q^{1/2}\alpha^{-1}t)+r_{0,-}} 
     \right] ,
\end{multline}
\begin{multline}
  \mathfrak{T}_{-}(z;q^{1/2}t) = \frac{\alpha t(z-z^{-1})}{(q-\alpha^2)(qt^2-1)}\frac{1}{p_{+}}
    \\
     \times\left[
     \frac{(qtz-\alpha)(qtz^{-1}-\alpha)}{(\alpha^2t^2-1)}
     (r_{1,-}\tfrac{1}{2}(q^{-1/2}\alpha t+q^{1/2}\alpha^{-1}t^{-1})+r_{0,-})\mathfrak{W}_{-}(\alpha t;q^{1/2}t) 
     \right.
    \\
     \left.
    +q^2\frac{(z-\alpha t)(z^{-1}-\alpha t)}{(\alpha^2-q^2t^2)}
     (r_{1,-}\tfrac{1}{2}(q^{-1/2}\alpha t^{-1}+q^{1/2}\alpha^{-1}t)+r_{0,-})\mathfrak{W}_{+}(q^{-1}\alpha t^{-1};q^{1/2}t)
     \right] .
\end{multline}
This particular parameterisation implies that the numerator of the right-hand side of (\ref{aux10}) manifestly 
contains a factor of $ (z-z^{-1})(qtz-\alpha)(qtz^{-1}-\alpha)(z-\alpha t)(z^{-1}-\alpha t) $
which is present in the denominator. This ensures that the ratio is linear in $ x $, as it must.
Finding the zero of this linear polynomial then gives $ \check{\lambda} $, which after further substantial
manipulation and simplification yields (\ref{t-Evol:c}) and (\ref{t-Evol:d}).
\end{proof}


A crucial fact enabling further progress is the following factorisation formula for a quantity
that will subsequently figure prominently in certain discriminants.
\begin{lemma}
The bi-quadratic in $ \hat{\rho}, \check{\lambda} $ 
\begin{multline}
   16\sigma_4(\hat{\rho}^2\check{\lambda}^2-\hat{\rho}^2-\check{\lambda}^2)-8s_4(q^2+q\sigma_2+\sigma_4)\hat{\rho}\check{\lambda}+8q^{1/2}s_4(q\sigma_1+\sigma_3)\hat{\rho}
  \\
   +8q^{-1/2}(\sigma_1\sigma_4+q\sigma_3)\check{\lambda}+(q-\sigma_2)^2-4\sigma_1\sigma_3+2\sigma_4-2q^{-1}\sigma_2\sigma_4+q^{-2}\sigma^2_4 ,
\label{perfectSQ}
\end{multline}
is a perfect square which can be given in either of two ways. In the first way this is the square of
\begin{multline}
   qt^2\hat{f}^{-2}-q^{1/2}t^2\sigma_1\hat{f}^{-1}+q^{-1/2}t^{-2}\sigma_3\hat{f}-q^{-1}t^{-2}\sigma_4\hat{f}^2
  +qt^{-2}-q^{-1}t^{2}\sigma_4+q^{3/2}\{n+\tfrac{1}{2}\}\alpha^{-1}t^{-3}(t^2-1)(\hat{f}+t^4\hat{f}^{-1})
  \\
  -2qs_4(t^{-2}\hat{f}-t^2\hat{f}^{-1})\hat{\rho}-q\frac{(1-s^2_4\hat{f}^2)(\hat{f}-q^n\alpha t)(t-q^n\alpha s^2_4\hat{f})}
         {\hat{f}^{2}(t\hat{f}-q^n\alpha)(1-q^n\alpha ts^2_4\hat{f})}
    \frac{w(\hat{f})}{1+s^2_4\hat{f}^2-2s_4\hat{f}\hat{\rho}}
  \\
  +\frac{q^{1/2}\alpha t(t^2-1)s^3_4}{[n+\tfrac{1}{2}]}
    \Bigg[
         \frac{q^{2n}\alpha^2s_4^3(1-q^{2n}\alpha^2t^2s^2_4)(\hat{f}-q^n\alpha t)w(q^{-n}\alpha^{-1}t^{-1}s_4^{-2})}
              {(1-q^n\alpha ts^2_4\hat{f})(1+q^{2n}\alpha^2t^2s^2_4-2q^n\alpha ts_4\hat{\rho})}
  \\
        +\frac{q^{-2n}\alpha^{-2}s_4^{-3}(t^2-q^{2n}\alpha^2s^2_4)(t-q^n\alpha s^2_4\hat{f})w(q^{n}\alpha t^{-1})}
              {(t\hat{f}-q^n\alpha)(t^2+q^{2n}\alpha^2s^2_4-2q^n\alpha ts_4\hat{\rho})}
    \Bigg] ,
\label{t-Aux:c}
\end{multline}
and in the second way this is the square of
\begin{multline}
   q^{1/2}\sigma_3t^{-2}\check{f}-q^{-1/2}\sigma_1t^2\check{f}^{-1}
  -(t^2\check{f}^{-2}+q^2t^{-2}s^2_4)\left[ \check{f}^2-1+q^{-1}(q+\alpha^2)\alpha^{-1}t^{-1}(q-t^2)\check{f} \right]
  \\
  +2\check{\lambda}(t^2\check{f}^{-1}-q^2t^{-2}s^2_4\check{f})
  +\frac{q(\check{f}^2-1)(q\check{f}-\alpha t)(t-\alpha\check{f})}{\check{f}^2(t\check{f}-\alpha)(q-\alpha t\check{f})}
   \frac{w(\check{f})}{\check{f}^2+1-2\check{f}\check{\lambda}}
  \\
  +\frac{t(q-t^2)}{q\alpha(q-\alpha^2)}
   \Bigg[
        \frac{t-\alpha\check{f}}{t\check{f}-\alpha}
        \frac{q^2(t^2-\alpha^2)w(\alpha t^{-1})}{\alpha^2+t^2-2\alpha t\check{\lambda}}
       +\frac{q\check{f}-\alpha t}{q-\alpha t\check{f}}
        \frac{\alpha^4(q^2-\alpha^2t^2)w(q\alpha^{-1}t^{-1})}{q^2+\alpha^2t^2-2q\alpha t\check{\lambda}}
   \Bigg] .
\label{t-Aux:d}
\end{multline}
\end{lemma}
\begin{proof}
In the first way we substitute (\ref{t-Evol:c}) into (\ref{perfectSQ}), whereas in the second
way we substitute (\ref{Evol_aux}) into (\ref{perfectSQ}). After considerable simplification we
arrive at the two results.
\end{proof}

We are now in a position to derive the evolution equations for the $ f $-variable, firstly
in the advanced direction.
\begin{proposition}\label{evolution_3}
Assuming $ \check{f}\neq 0 $, $ t\check{f}-\alpha\neq 0 $, and $ q-\alpha t\check{f}\neq 0 $
we have the forward evolution for $ f $ 
\begin{multline}
   \hat{f} = 
  \left\{ 
   \frac{2t^2(t^2-1)s^2_4}{q^{n-1/2}[n+\tfrac{1}{2}]}
   \Bigg[
        -\frac{q^{4n}\alpha^4s^4_4w(q^{-n}\alpha^{-1}t^{-1}s^{-2}_4)}{1+q^{2n}\alpha^2t^2s^2_4-2q^n\alpha ts_4\hat{\rho}}
        +\frac{w(q^{n}\alpha t^{-1})}{t^2+q^{2n}\alpha^2s^2_4-2q^n\alpha ts_4\hat{\rho}}
   \Bigg]
  \right.
  \\
  -\sigma_2-q^{-1}(1-2t^2)\sigma_4-q(1-2t^{-2})+2\check{\lambda}\left[ 2qs_4\hat{\rho}-t^2\check{f}^{-1}+q^2t^{-2}s^2_4\check{f} \right]
  \\
  -q^{1/2}\sigma_3t^{-2}\check{f}+q^{-1/2}\sigma_1t^2\check{f}^{-1}
  +(t^2\check{f}^{-2}+q^2t^{-2}s^2_4)\left[ \check{f}^2-1+q^{-1}(q+\alpha^2)\alpha^{-1}t^{-1}(q-t^2)\check{f} \right]
  \\
  -\frac{q(\check{f}^2-1)(q\check{f}-\alpha t)(t-\alpha\check{f})}{\check{f}^2(t\check{f}-\alpha)(q-\alpha t\check{f})}
   \frac{w(\check{f})}{\check{f}^2+1-2\check{f}\check{\lambda}}
  \\
  \left.
  -\frac{t(q-t^2)}{q\alpha(q-\alpha^2)}
   \Bigg[
        \frac{t-\alpha\check{f}}{t\check{f}-\alpha}
        \frac{q^2(t^2-\alpha^2)w(\alpha t^{-1})}{\alpha^2+t^2-2\alpha t\check{\lambda}}
       +\frac{q\check{f}-\alpha t}{q-\alpha t\check{f}}
        \frac{\alpha^4(q^2-\alpha^2t^2)w(q\alpha^{-1}t^{-1})}{q^2+\alpha^2t^2-2q\alpha t\check{\lambda}}
   \Bigg]
  \right\}
  \\ \times
  \Bigg\{
  \frac{2q^{1/2}t(t^2-1)}{[n+\tfrac{1}{2}]}
  \Bigg[
  -\frac{q^{2n}\alpha^3s^6_4w(q^{-n}\alpha^{-1}t^{-1}s^{-2}_4)}{1+q^{2n}\alpha^2t^2s^2_4-2q^n\alpha ts_4\hat{\rho}}
  +\frac{\alpha s^4_4w(q^{n}\alpha t^{-1})}{t^2+q^{2n}\alpha^2s^2_4-2q^n\alpha ts_4\hat{\rho}}
  \Bigg]
  \\
  -2q^{3/2}\alpha^{-1}t^{-3}\{n+\tfrac{1}{2}\}(t^2-1)-2q^{-1}t^{-2}(q^{1/2}\sigma_3-2t^2\sigma_4\check{\lambda})
  +4qt^{-2}s_4\hat{\rho}
  \Bigg\}^{-1} .
\label{t-Evol:b}
\end{multline}
where $ \hat{\rho} $ is given by (\ref{t-Evol:a}).
\end{proposition}
\begin{proof}
Whilst (\ref{t-Evol:c}) can be primarily viewed as a linear equation for $ \check{\lambda} $ it
can also be viewed as a quadratic equation for $ \hat{f} $, and as such, possesses a discriminant
which contains, apart from explicit squared factors, the factor (\ref{perfectSQ}). By substituting
for $ \hat{\rho} $ using (\ref{Evol_aux}) we can employ the result (\ref{t-Aux:d}), and
effect a factorisation of the quadratic into linear factors. The choice of the factors can be
settled by consideration of the known solution for $ n=0 $ and it transpires that the negative branch is
appropriate. This yields (\ref{t-Evol:b}).
\end{proof}

Our last task is to derive the inverse to (\ref{t-Evol:b}) and complete the system
of evolution equations.
\begin{proposition}\label{evolution_4}
The inverse of the evolution for $ f $ is given by
\begin{multline}
  \check{f} = \Bigg\{
   \frac{2t^2(q-t^2)}{q-\alpha^2}
   \Bigg[
        \frac{w(\alpha t^{-1})}{\alpha^2+t^2-2\alpha t\check{\lambda}}
       -q^{-1}\alpha^4
        \frac{w(q\alpha^{-1}t^{-1})}{q^2+\alpha^2t^2-2q\alpha t\check{\lambda}}
   \Bigg]
  \\
  +4s_4\hat{\rho}\check{\lambda}+2s_4(t^{-2}\hat{f}-t^2\hat{f}^{-1})\hat{\rho}
  -t^2\hat{f}^{-2}+q^{-1/2}t^2\sigma_1\hat{f}^{-1}-q^{-1}\sigma_2-q^{-3/2}t^{-2}\sigma_3\hat{f}+q^{-2}t^{-2}\sigma_4\hat{f}^2
  \\
  -q^{-2}t^{-2}(-2q+t^2-t^4)\sigma_4-q^{-1}t^{-2}(q(1+t^2)-2t^4)-q^{1/2}\{n+\tfrac{1}{2}\}\alpha^{-1}t^{-3}(t^2-1)(\hat{f}+t^4\hat{f}^{-1})
  \\
  +\frac{(1-s^2_4\hat{f}^2)(\hat{f}-q^n\alpha t)(t-q^n\alpha s^2_4\hat{f})}
         {\hat{f}^{2}(t\hat{f}-q^n\alpha)(1-q^n\alpha ts^2_4\hat{f})}
    \frac{w(\hat{f})}{1+s^2_4\hat{f}^2-2s_4\hat{f}\hat{\rho}}
  \\
  -\frac{\alpha t(t^2-1)s^3_4}{q^{1/2}[n+\tfrac{1}{2}]}
    \Bigg[
         \frac{q^{2n}\alpha^2s_4^3(1-q^{2n}\alpha^2t^2s^2_4)(\hat{f}-q^n\alpha t)w(q^{-n}\alpha^{-1}t^{-1}s_4^{-2})}
              {(1-q^n\alpha ts^2_4\hat{f})(1+q^{2n}\alpha^2t^2s^2_4-2q^n\alpha ts_4\hat{\rho})}
  \\
        +\frac{q^{-2n}\alpha^{-2}s_4^{-3}(t^2-q^{2n}\alpha^2s^2_4)(t-q^n\alpha s^2_4\hat{f})w(q^{n}\alpha t^{-1})}
              {(t\hat{f}-q^n\alpha)(t^2+q^{2n}\alpha^2s^2_4-2q^n\alpha ts_4\hat{\rho})}
    \Bigg]
  \Bigg\}
  \\ \times
  \Bigg\{
   \frac{2\alpha t(q-t^2)}{q-\alpha^2}
    \Bigg[
         \frac{w(\alpha t^{-1})}{\alpha^2+t^2-2\alpha t\check{\lambda}}
        -\frac{\alpha^2w(q\alpha^{-1}t^{-1})}{q^2+\alpha^2t^2-2q\alpha t\check{\lambda}}
    \Bigg]
  \\
  -2q^{-1/2}t^{-2}\sigma_3+2q^{-2}\alpha^{-1}t^{-3}(q-t^2)(q+\alpha^2)\sigma_4+4s_4\hat{\rho}+4q^{-1}t^{-2}\sigma_4\check{\lambda}
  \bigg\}^{-1} .
\label{t-Evol:e}
\end{multline}
where $ \check{\lambda} $ is given by (\ref{t-Evol:c}).
\end{proposition}
\begin{proof}
In a similar manner to the previous proof we can solve (\ref{Evol_aux}) for $ \check{f} $.
On appearance this polynomial should be a quintic in $ \check{f} $, however in line with
Remark \ref{Quin2Quad}, we observe this contains a factor of $ \check{f}(t\check{f}-\alpha)(q-\alpha t\check{f}) $
and we only have to deal with a quadratic. Upon examining this quadratic we find that the
discriminant contains, again, the factor (\ref{perfectSQ}) as the only manifestly square-free factor.  
Upon substituting for $ \check{\lambda} $ using (\ref{t-Evol:c}) we can now employ (\ref{t-Aux:c})
and factorise the quadratic. As in the previous case consideration of the $ n=0 $ solution
resolves the sign ambiguity in favour of the negative root, yielding (\ref{t-Evol:e}).
\end{proof}

\subsection{Seed Solution}\label{seed}
We can now make contact with the earlier theory characterising the moments, as given in
Subsection \ref{SSect7_1}, through a study of the classical "seed" solutions to the coupled recurrence
system given above. We now append a subscript on the variables to indicate the $ n-$value.
\begin{proposition}
The recurrence relations (\ref{t-Evol:a}) and (\ref{t-Evol:b}) admit the classical "seed" solution at $ n=0 $
\begin{equation}
    f_0(t) = q^{-1/2}\alpha t ,
\label{CLsoln:a}
\end{equation}
and
\begin{equation}
   2qs_4\rho_0(t)
   = \frac{C(q^{1/2}\alpha^{-1}t+\sigma_4q^{-1/2}\alpha t^{-1})m(t)+D(q^{3/2}\alpha^{-1}t^{-1}+\sigma_4q^{-3/2}\alpha t)m(q^{-1}t)}
          {Cm(t)+Dm(q^{-1}t)} ,
\label{CLsoln:b}
\end{equation}
where $ m(t) $ is the general solution to the second-order $q$-difference equation
\begin{multline}
   (t^2-q^2)t^4 w(\alpha t^{-1}) m(q^{1/2}t)
   \\
   + \frac{\hat{D}}{\hat{C}}
     \left[ (t^2-q^2)t^4 w(\alpha t^{-1})+q^3(t^2-1) w(q^{-1}\alpha t)
            -q^{-3}(q-\alpha^2)(q^2-\sigma_4 \alpha^2)(t^2-1)(t^2-q)(t^2-q^2)
     \right] m(q^{-1/2}t)
   \\
   + \frac{\hat{D}\check{D}}{\hat{C}\check{C}} q^3(t^2-1) w(q^{-1}\alpha t) m(q^{-3/2}t) 
   = 0 .
\label{CLsoln:c}
\end{multline}
\end{proposition}
\begin{proof}
It is clear that both (\ref{evalSpec:c}) and (\ref{evalSpec:d}) both vanish when 
(\ref{CLsoln:a}) holds at $ n=0 $, therefore together they satisfy (\ref{splitId}).
Given (\ref{CLsoln:a}) we observe that two terms vanish in (\ref{t-Evol:a}) and $ \rho_0(t) $
satisfies the discrete Riccati equation
\begin{equation}
  2s_4\hat{\rho}_0-\alpha^{-1}t-\alpha t^{-1}s^2_4
  = \frac{qt^4w(\alpha t^{-1})\left[ 2qs_4\check{\rho}_0-q^2\alpha^{-1}t^{-1}-\alpha t s^2_4 \right]}
         {\alpha t(q-t^{2})(q-\alpha^{2})\left[ 2qs_4\check{\rho}_0-q^2\alpha^{-1}t^{-1}-\alpha t s^2_4\right]+q^4w(q^{-1}\alpha t)} .
\end{equation}
Alternatively we can derive this relation by specialising (\ref{t-Evol:d}) under (\ref{CLsoln:a})
which yields
\begin{equation}
    t^2+\alpha^2-2\alpha t\check{\lambda}_0 = \frac{t^4w(\alpha t^{-1})}{t^2+\alpha^2s^2_4-2\alpha ts_4\hat{\rho}_0} ,
\end{equation}
and then using (\ref{glXFM}) to substitute for $ \check{\lambda}_0 $.
Making the standard linearising transformation
\begin{equation}
    2qs_4\rho_0(t) = \frac{Am(t)+Bm(q^{-1}t)}{Cm(t)+Dm(q^{-1}t)} ,
\end{equation}
we find, in our solution for the decoupling factors, that
$ \hat{A} = (q\alpha^{-1}t+\sigma_4 q^{-1}\alpha t^{-1})\hat{C} $ and
$ \check{B} = (q^2\alpha^{-1}t^{-1}+\sigma_4 q^{-2}\alpha t)\check{D} $
and arrive at (\ref{CLsoln:c}).
\end{proof}

\begin{remark}
Our explicit initial orthogonal polynomial variables are one specialisation of the classical 
solutions given above. As remarked in the proof of Proposition \ref{SS_parameterise} we have 
$ \mu_0 = V(\lambda_0) $ and $ \nu_0 = W(\lambda_0) $ at $ n=0 $, where $ \lambda_0 $ is the unique zero
of the right-hand side of (\ref{M=3U}) 
\begin{equation}
   2[\tfrac{1}{2}]\lambda_0 = -q^{-1/2}\tilde{\sigma}_1+q^{1/2}\tilde{\sigma}_5-[1]\frac{m_{0,+}+m_{0,-}}{m_{0,0}} .
\end{equation}
We easily see that 
$ \mathfrak{z}_{\pm,0}(t) = l_0^{\mp 3}\prod^{6}_{j=1}(1-q^{-1/2}a_j l_0^{\pm 1}) $
and from the transformations (\ref{splitXfm:a},\ref{splitXfm:b}) we deduce that
$ f_0(t) = q^{-1/2}\alpha t $. 
The elements of the spectral matrix $ A^{*}_0 $ are
\begin{align}
   \mathfrak{W}_{+,0}(z;t) & = z^{-3}w(z)(1-q^{-1/2}\alpha tz)(1-q^{-1/2}\alpha t^{-1}z) = W+\Delta yV ,\\ 
   \mathfrak{W}_{-,0}(z;t) & = z^{3}w(z^{-1})(1-q^{-1/2}\alpha tz^{-1})(1-q^{-1/2}\alpha t^{-1}z^{-1}) = W-\Delta yV ,\\ 
   \mathfrak{T}_{+,0}(z;t) & = 2a_0[\tfrac{1}{2}](z-z^{-1})(z-l_0)(1-l_0^{-1}z^{-1}) ,\\ 
   \mathfrak{T}_{-,0}(z;t) & = 0 ,
\end{align}
whilst those of the deformation matrix $ B^{*}_0 $ are
\begin{align}
   \mathfrak{R}_{+,0}(z;t) & = -\frac{a_0\hat{\gamma}_0}{\check{a}_0\check{\gamma}_0}(1-q^{-1/2}\alpha tz)(1-q^{-1/2}\alpha tz^{-1}) ,\\ 
   \mathfrak{R}_{-,0}(z;t) & = -\frac{a_0\check{\gamma}_0}{\hat{a}_0\hat{\gamma}_0}(1-q^{-1/2}\alpha t^{-1}z)(1-q^{-1/2}\alpha t^{-1}z^{-1}) ,\\ 
   \mathfrak{P}_{+,0}(z;t) & = 2q^{-1/2}\alpha a_0\left[ t\frac{\hat{\gamma}_0}{\check{\gamma}_0}-t^{-1}\frac{\check{\gamma}_0}{\hat{\gamma}_0} \right] ,\\ 
   \mathfrak{P}_{-,0}(z;t) & = 0 .
\end{align}
We also note from (\ref{t-Evol:f}) and (\ref{ops_gammaDelta},\ref{ops_Hdet}) that
\begin{equation}
     2\alpha s_4 \rho_0(t)
  = \frac{(1+q^{-1}\alpha^2s^2_4t^{2})m_{0,0}(q^{-1}t)-(1+q\alpha^2s^2_4t^{-2})m_{0,0}(t)}{q^{-1/2}t\,m_{0,0}(q^{-1}t)-q^{1/2}t^{-1}m_{0,0}(t)} ,
\end{equation}
which is precisely the case of (\ref{CLsoln:b}) with $ D=-q^{-1}t^2C $. With this solution we
also deduce that the moment recurrence (\ref{3TermMoment}) for $ m_{0,0} $ coincides with (\ref{CLsoln:c}).
\end{remark}

At the beginning of this section, in Remark \ref{Check}, we noted that some explicit solutions 
to the moment recurrences were given in \cite{Wi_2010b}. This knowledge enables the possibility 
of checking any aspect of the foregoing theory either symbolically or numerically to
essentially unlimited precision. We wish to report that extensive checking of all the key
relations for the cases $ n=0 $ and $ n=1 $ has been carried out employing Mathematica code
and utilities where possible by exact symbolic means or if not by numerical means to a level
of one part in $ 10^{20}\to 10^{30} $ or better at random exact values of the input 
parameters. These checks have also utilised a body of unreported work \cite{Wi_2010c} covering
the system of Laguerre-Freud recurrences under $ n \mapsto n+1 $ for general $ n \in\Z_{\geq 0}$. 

We now come to the question regarding the identity of this system in the Sakai scheme. 
Whilst not providing an explicit transformation from our parameters and variables to those appearing
in the canonical coupled $q$-difference system \cite{GR_1999,Sa_2001},
we offer unambiguous evidence that it is one of the classical yet full parameter set cases of the $ E^{(1)}_7 $ $q$-Painlev\'e system
as can be seen from the following inspection of their classical solutions \cite{MSY_2003,KMNOY_2003,KMNOY_2004}, 
and the recent systematic study of "two-Casorati" determinantal forms of their classical solutions \cite{Masuda_2009}.
Clearly the $\tau$-functions of our construction, given by (\ref{ops_Hdet}) and (\ref{ops_Sdet}), are of the 
"two-Casorati" determinant form because if we choose $ b=a_r, a=a_s $ with $ r,s \in \{1,2,3,4\} $ say
and this is employed in the definition (\ref{ops_moment}) we have a moment determinant with
elements $ m_{0,0}(q^j a_r, q^k a_s) $, $ j,k \in \Z_{\geq 0} $ where $ m_{0,0} $ given by any
$q$-constant linear combination of the two $ {}_8W_7 $ solutions appearing in (\ref{M=3AWmoment}).

Several tasks, that have arisen in the course of our study, remain unfinished and
we conclude by detailing them and the issues involved. 
Clarification of the explicit relationship of the evolution system (\ref{t-Evol:a},\ref{t-Evol:d},\ref{t-Evol:b},\ref{t-Evol:e})
with the canonical $q$-difference equations as say given in \cite{KMNOY_2003,KMNOY_2004}
is required. It remains to complete the construction of the $ n \to n+1 $ recurrence relations
which occupy a special place in our approach, but figure as Schlesinger transformations in the
integrable theory. In fact our framework can easily treat the
Schlesinger transformations of this example (or any for that matter), and in particular the 
$ n \to n+1 $ recurrences, or the $ a_{j}, \alpha \mapsto qa_{j}, q\alpha, j=1,\ldots,4 $ mappings.
These latter transformations are manifested in our context as specialised 
Christoffel-Uvarov transformations, however we will postpone this undertaking.
 
Significant progress has been made in finding the analog of an isomonodromic system for the elliptic Painlev\'e 
courtesy of a preprint by Eric Rains \cite{Ra_2007}, and in the work by Yamada \cite{Ya_2009a}. 
The approach taken in this former work is very much in the spirit of the present study and it would be 
natural to expect that a limiting case of the results reported there would correspond our own.
Very recently Yamada \cite{Ya_2010} has given Lax pairs for the $ E^{(1)}_8 $, $ E^{(1)}_7 $ 
and $ E^{(1)}_6 $ $q$-Painlev\'e equations by reformulating the $ E^{(1)}_8 $ elliptic Painlev\'e
Lax pair system and taking limits $ E^{(1)}_8 \to E^{(1)}_7 \to E^{(1)}_6 $. One would expect
expect that a gauge transformation and co-ordinate transformations would link our Lax Pair with his.
 
\section{Acknowledgments}
This research has been supported by the Australian Research Council.
The author wishes to acknowledge Alphonse Magnus, whose seminal study formed the inspiration for this work and
has freely offered many clear perceptions into the subject,
the generous insights shared by Eric Rains, the kind interest of Mizan Rahman, 
the enthusiastic encouragement from Alexander Its
and the support of Peter Forrester.

\bibliographystyle{plain}
\bibliography{moment,random_matrices,nonlinear}

\end{document}